\documentclass[a4j,10pt,showkeys]{article}

\usepackage{amsmath,amsthm,amssymb}
\usepackage{geometry}
\usepackage[dvipsnames]{xcolor}
\usepackage{cancel}
\usepackage{mathtools}
\usepackage{bm}
\usepackage{color}
\usepackage{mathrsfs}
\usepackage{hyperref}
\hypersetup{
    colorlinks=true,
    citecolor=NavyBlue,
    linkcolor=Red,
    urlcolor=OliveGreen,
}
\usepackage[capitalize]{cleveref}

\theoremstyle{plain}
\newtheorem{theo}{Theorem}[section]
\newtheorem{lemm}{Lemma}[section]
\newtheorem{prop}{Proposition}[section]
\newtheorem{cor}{Corollary}[section]

\theoremstyle{definition}
\newtheorem{defi}{Definition}[section]
\newtheorem{rem}{Remark}[section]
\newtheorem{exam}{Example}[section]

\newtheorem{assum}{Assumption} % セクション番号と紐付けない

\makeatletter
\let\c@lemm\c@theo
\let\c@prop\c@theo
\let\c@cor\c@theo
\let\c@defi\c@theo
\let\c@rem\c@theo
\let\c@exam\c@theo
\let\c@prob\c@theo
\makeatother

\crefname{theo}{Theorem}{Theorems}
\crefname{lemm}{Lemma}{Lemmas}
\crefname{prop}{Proposition}{Propositions}
\crefname{cor}{Corollary}{Corollaries}
\crefname{claim}{Claim}{Claims}
\crefname{defi}{Definition}{Definitions}
\crefname{rem}{Remark}{Remarks}
\crefname{assum}{Assumption}{Assumptions}
\crefname{exam}{Example}{Examples}
\crefname{prob}{Problem}{Problems}

\newcommand{\bE}{\mathbb{E}}
\newcommand{\bF}{\mathbb{F}}

\newcommand{\bN}{\mathbb{N}}

\newcommand{\bP}{\mathbb{P}}
\newcommand{\bQ}{\mathbb{Q}}
\newcommand{\bR}{\mathbb{R}}

\newcommand{\bW}{\mathbb{W}}

\newcommand{\cA}{\mathcal{A}}
\newcommand{\cB}{\mathcal{B}}

\newcommand{\cD}{\mathcal{D}}
\newcommand{\cE}{\mathcal{E}}
\newcommand{\cF}{\mathcal{F}}
\newcommand{\cG}{\mathcal{G}}
\newcommand{\cH}{\mathcal{H}}

\newcommand{\cM}{\mathcal{M}}

\newcommand{\cP}{\mathcal{P}}

\newcommand{\cS}{\mathcal{S}}
\newcommand{\cT}{\mathcal{T}}

\newcommand{\cV}{\mathcal{V}}

\newcommand{\cX}{\mathcal{X}}

\newcommand{\sC}{\mathscr{C}}

\newcommand{\sE}{\mathscr{E}}

\newcommand{\sS}{\mathscr{S}}

\newcommand{\sY}{\mathscr{Y}}

\newcommand{\fp}{\mathfrak{p}}
\newcommand{\fq}{\mathfrak{q}}
\newcommand{\fa}{\mathfrak{a}}
\newcommand{\bfP}{\mathbf{P}}
\newcommand{\bfQ}{\mathbf{Q}}

\newcommand{\ep}{\varepsilon}
\newcommand{\diff}{\mathrm{d}}
\newcommand{\dmu}{\,\mu(\diff\theta)}
\newcommand{\supp}{{\mathrm{supp}\,\mu}}
\newcommand{\Law}{\mathrm{Law}}

\newcommand{\TV}{\mathrm{TV}}
\newcommand{\KL}{\mathrm{KL}}
\newcommand{\UE}{\mathrm{UE}}
\newcommand{\LG}{\mathrm{LG}}
\newcommand{\BDG}{\mathrm{BDG}}
\newcommand{\subLG}{\mathrm{subLG}}
\newcommand{\Lyap}{\mathrm{Lyap}}
\newcommand{\Lip}{\mathrm{Lip}}
\newcommand{\loc}{\mathrm{loc}}
\newcommand{\op}{\mathrm{op}}
\newcommand{\tr}{\mathrm{tr}}
\newcommand{\fractional}{\mathrm{frac}}
\newcommand{\relmiddle}[1]{\mathrel{}\middle#1\mathrel{}}
\newcommand{\1}{\mbox{\rm{1}}\hspace{-0.25em}\mbox{\rm{l}}}
\DeclareMathOperator*{\esssup}{ess\,sup}
\makeatletter
\newcommand{\vnorm}{\@ifstar\@opnorms\@opnorm}
\newcommand{\@opnorms}[1]{%
  \left|\mkern-1.5mu\left|\mkern-1.5mu\left|
   #1
  \right|\mkern-1.5mu\right|\mkern-1.5mu\right|
}
\newcommand{\@opnorm}[2][]{%
  \mathopen{#1|\mkern-1.5mu#1|\mkern-1.5mu#1|}
  #2
  \mathclose{#1|\mkern-1.5mu#1|\mkern-1.5mu#1|}
}
\makeatother

\providecommand{\keywords}[1]{\textbf{Keywords:} #1}
\makeatletter

\@addtoreset{equation}{section}
\makeatother
\def\widebar{\accentset{{\cc@style\underline{\mskip10mu}}}}
\numberwithin{equation}{section}
\allowdisplaybreaks

\title{Exponential ergodicity and finite-dimensional approximation for Markovian lifts of stochastic Volterra equations}

\author{
Yushi Hamaguchi\footnote{Graduate School of Science, Department of Mathematics, Kyoto University. Email: \href{mailto:hamaguchi@math.kyoto-u.ac.jp}{hamaguchi@math.kyoto-u.ac.jp}}\ \footnote{The author was supported by JSPS KAKENHI Grant Number 22K13958.}
}

\begin{document}
\maketitle

%% Abstract

\begin{abstract}
This paper investigates the long-time asymptotics and the existence of stationary solutions for a class of stochastic Volterra equations (SVEs). To address the non-Markovian nature of SVEs, we employ a Markovian lifting technique, formulating a Markovian lift as the solution to a stochastic evolution equation (SEE) on a Gelfand triplet. Our main objective is to establish the ergodicity of this Markovian lift via the generalized Harris' theorem, which in turn yields the asymptotic results for the original SVE. Despite the challenges posed by the highly degenerate, infinite-dimensional nature of the SEE, we achieve this by constructing a generalized coupling and a distance function that exploit the structural properties arising from the non-local operators in its coefficients. Furthermore, we prove that the invariant probability measure and, more generally, the stationary law on the path space of the SEE can be weakly approximated by those of finite-dimensional SDEs. This yields a novel approximation result for the stationary solution of the original SVE, while offering a rigorous mathematical framework that supports the validity of the Markovian embedding concept widely utilized in statistical physics.
\end{abstract}

%% Keywords

\keywords
Stochastic Volterra equation; stochastic evolution equation; Markovian lift; ergodicity.

%% MSC


%37A25 Ergodicity, mixing, rates of mixing
%60H15 Stochastic partial differential equations (aspects of stochastic analysis)
%45D05 Volterra integral equations
%60G22 Fractional processes, including fractional Brownian motion

%60H20 Stochastic integral equations
%60H50 Regularization by noise
%45A05 Linear integral equations
%93E20 Optimal stochastic control
%93B52 Feedback control
%49K45 Optimality conditions for problems involving randomness
%45G05 Singular nonlinear integral equations
%49N15 Duality theory (optimization)
%45B05 Fredholm integral equations
%34A08 Fractional ordinary differential equations and fractional differential inclusions
%26A33 Fractional derivatives and integrals

%%%%%%%%%%%%%%%%%%%%%%%%%%%%%%%%%%
%%%%%%%%%%%%%%%%%%%%%%%%%%%%%%%%%%
%% Section
%%%%%%%%%%%%%%%%%%%%%%%%%%%%%%%%%%
%%%%%%%%%%%%%%%%%%%%%%%%%%%%%%%%%%

\section{Introduction}\label{sec_intro}

The analysis of the long-time asymptotic behavior and the existence of stationary solutions for stochastic differential equations (SDEs) is a fundamental problem in stochastic analysis. In this paper, we investigate these properties for \emph{stochastic Volterra equations} (SVEs) of the following form:
\begin{equation}\label{eq_SVE}
	X_t=x(t)+\int^t_0K_b(t-s)b(X_s)\,\diff s+\int^t_0K_\sigma(t-s)\sigma(X_s)\,\diff W_s,\ \ t>0.
\end{equation}
Here, $b:\bR^n\to\bR^n$ and $\sigma:\bR^n\to\bR^{n\times d}$ denote the drift and diffusion coefficients, respectively, $W$ is a $d$-dimensional Brownian motion, and $x$ is a given $\bR^n$-valued function called the forcing term, which is possibly random but independent of the Brownian motion. Introducing the (matrix-valued) functions $K_b,K_\sigma:(0,\infty)\to\bR^{n\times n}$, which are called kernels, allows us to model dynamics exhibiting both sample path roughness as well as memory effects that cannot be captured by SDEs. Prominent examples include the rough volatility model in mathematical finance \cite{BaFrFuGaJaRo23} and the (over-damped) generalized Langevin equation in statistical physics \cite{KuToHa91,LiLiLu17,Zw01}. However, the inherent non-Markovian nature of the SVE \eqref{eq_SVE} poses significant challenges for the analysis of long-time asymptotics, as standard results on the ergodicity of Markov processes---which are central to the study of SDEs---cannot be directly applied. Furthermore, the solution $X$ to \eqref{eq_SVE} typically falls outside the framework of semimartingales, which presents another significant challenge, as the classical It\^{o} calculus is no longer directly available.

In order to overcome the difficulties arising from these non-Markovian and non-semimartingale features, various ``Markovian lifting'' techniques have been developed in the literature. By introducing suitable Markovian dynamics called \emph{Markovian lifts}, which typically evolve in infinite-dimensional spaces, one can represent the solution of an SVE as a projection of such a lift. There are at least two well-established procedures for constructing Markovian lifts of SVEs. The first is based on representing the kernels by the Laplace transforms of certain measures, an approach adopted, for example, in \cite{CaCo98,CuTe20,Ha24,Ha25a,HaSt19} across various frameworks. This approach is related to the Markovian embedding for generalized Langevin equations \cite{Go12,Ku04} and multi-factor approximation of rough volatility models \cite{AbiJaEu19,AlKe24,BaBr23}; see also \cite{AbiJaMiPh21,Ha25b} for applications to optimal control problems of SVEs. The second is based on the analysis of the dynamics of auxiliary processes, such as conditional future paths or the past paths, as seen in \cite{BeDeKr22,GaPa25,Ha05,ViZh19}. We also refer to the recent work \cite{BiBoCaFr25}, where an abstract framework is introduced to unify the two distinct approaches in \cite{BeDeKr22,Ha24}. Notably, the approach established in our previous works \cite{Ha24,Ha25a} allows for the study of Markovian lifts within the framework of stochastic evolution equations (SEEs) on Gelfand triplets. This framework is particularly tractable, by virtue of the Hilbert space structure and a connection to well-established theory of monotone stochastic partial differential equations (SPDEs) initiated by Pardoux \cite{Pa72,Pa75} and further developed by Krylov and Rozovski\u{\i} \cite{KrRo79}; see also the monographs \cite{GaMa10,LiRo15,Pa21} for a comprehensive treatment of monotone SPDEs.

In the present paper, we adopt and further generalize the methodology of \cite{Ha24,Ha25a}. Specifically, under the assumption that the kernels $K_b$ and $K_\sigma$ admit the representations $K_b(t)=\int_{[0,\infty)}e^{-\theta t}M_b(\theta)\dmu$ and $K_\sigma(t)=\int_{[0,\infty)}e^{-\theta t}M_\sigma(\theta)\dmu$ for some Borel measure $\mu$ on $[0,\infty)$ and measurable maps $M_b,M_\sigma:[0,\infty)\to\bR^{n\times n}$ satisfying suitable integrability conditions, we transform the SVE \eqref{eq_SVE} into the following SEE:
\begin{equation}\label{eq_SEE}
	\diff Y_t(\theta)=-\theta Y_t(\theta)\,\diff t+M_b(\theta)b(\mu[Y_t])\,\diff t+M_\sigma(\theta)\sigma(\mu[Y_t])\,\diff W_t,\ \ t\geq0,\ \theta\in[0,\infty).
\end{equation}
Here, $\mu[Y_t]$ represents a non-local term defined by the integral $\mu[Y_t]=\int_{[0,\infty)}Y_t(\theta)\dmu$. The SEE is rigorously formulated within a Gelfand triplet of Hilbert spaces $\cV\hookrightarrow\cH\hookrightarrow\cV^*$, which consists of weighted $L^2$ spaces with respect to the measure $\mu$. The weighted integrability structure of this Gelfand triplet accounts for the singularity of the kernels $K_b(t)$ and $K_\sigma(t)$ at $t=0$. In this paper, we assume these kernels to be merely $L^1$ and $L^2$ near the origin, respectively, thereby allowing for their divergence as $t\downarrow0$. This choice of integrability is minimal, as it corresponds to the necessary conditions for the Lebesgue and stochastic integrals in \eqref{eq_SVE} to be well-defined. The resulting solution $Y=(Y_t)_{t\geq0}$ to the SEE \eqref{eq_SEE} is a time-homogeneous Markov process on the Hilbert space $\cH$, from which the solution to the original SVE \eqref{eq_SVE} is recovered via the formula $X_t=\mu[Y_t]$. In this sense, the infinite-dimensional process $Y$ serves as a Markovian lift of $X$. For a more detailed exposition of this framework, the reader is referred to \cref{sec_pre}.

Our primary objective is to establish the ergodicity of this Markovian lift. Specifically, our main result (\cref{theo_main}) provides a spectral gap-type estimate and the exponential weak convergence of transition probabilities associated with the SEE \eqref{eq_SEE} toward its unique invariant probability measure. These results, in turn, ensure the long-time asymptotics and the existence of stationary solutions for the original SVE \eqref{eq_SVE}, as shown in \cref{cor_stationary-path}. To the best of our knowledge, the only existing work addressing the ergodicity of such a lift for SVEs is the preprint \cite{BiBoCaFr25}; see also \cite{BeDeKr22} for a study on the long-time limiting distributions of SVEs using an alternative class of lifts based on conditional future paths, while assuming strong regularity for the kernels. In \cite{BiBoCaFr25}, the authors introduce an abstract framework of Markovian lifts to prove a form of ergodicity without relying on the specific structure of the SEE. However, such generality necessitates stringent technical assumptions on the Lipschitz constants and growth conditions of the coefficients (similar constraints are also imposed in \cite{BiBoFr25} where limit distributions for SVEs are studied without resorting to Markovian lifts), and a spectral gap cannot be expected therein. In contrast, the present paper significantly relaxes these constraints by employing more natural Lyapunov-type conditions and instead assuming uniform ellipticity for the diffusion coefficient $\sigma$, while providing not only exponential weak ergodicity but also a spectral gap-type estimate for the SEE \eqref{eq_SEE}. The spectral gap is crucial in its own right, as it ensures the stability of invariant probability measures, as will be elaborated below. Our results are also comparable to the works \cite{BeFrKr23,FrJi24,JaPaSp25}, where ergodicity for a particular class of SVEs called affine Volterra processes is investigated by relying on the analysis of Volterra-type Riccati equations arising from the specific affine structure. We also refer to the recent preprint \cite{GnPa25}, which establishes the so-called ``fake stationarity'' for SVEs with a linear drift coefficient $b$. For a given solution process, this concept refers to the invariance under the time-shifts of merely marginal distributions or moments up to a certain degree. In contrast, we do not assume a linear or affine structure for the coefficients, and the resulting stationary solution for the original SVE \eqref{eq_SVE} possesses strict stationarity, meaning that all finite-dimensional distributions are invariant under the time-shifts. Detailed statements of our main results are provided in \cref{sec_main}.

The proof of our main result is based on the \emph{generalized Harris' theorem} developed by Hairer, Mattingly, and Scheutzow \cite{HaMaSc11} and further extended in \cite{Bu14,DuFoMo16}; see also the monograph \cite{Ku18}. This theorem extends the classical Harris' theorem (see, e.g., \cite{HaMa11,MeTw09}), originally developed to establish ergodicity with respect to the total variation distance, to a broader setting involving weighted Wasserstein-type distances. Such an extension is crucial in our framework, as the Markovian lift generally lacks the strong Feller property (see \cite[Theorem 3.4]{Ha24}); consequently, ergodicity with respect to the total variation distance is not to be expected. We review the precise statement of the generalized Harris' theorem from \cite{HaMaSc11} in \cref{sec_pre}. Within this framework, the most critical and challenging step is the construction of a suitable distance(-like) function $d$ that satisfies the ``contraction'' and the ``$d$-smallness'' conditions required in this theorem. In \cite{BuKuSc20}, the authors provide a set of verifiable sufficient conditions for the assumptions of the generalized Harris' theorem in terms of the notion of a \emph{generalized coupling} between transition probabilities. Accordingly, our task reduces to the construction of an appropriate generalized coupling and a corresponding distance-like function. However, this construction is still non-trivial, as it requires a deep analysis tailored to the specific structure of the model. A particularly formidable challenge we have to overcome is that the SEE \eqref{eq_SEE} is a \emph{highly degenerate, infinite-dimensional system}, meaning that the state space $\cH$ is typically infinite-dimensional while the noise, represented by the Brownian motion $W$, is finite-dimensional. In fact, the driving noise $W$ can be as low as one-dimensional if the original SVE \eqref{eq_SVE} is so, whereas the state space of the SEE \eqref{eq_SEE} remains infinite-dimensional. This stands in stark contrast to some existing works on ergodicity for ``effectively elliptic'' degenerate SPDEs via the generalized coupling approach (e.g., \cite{BuKuSc20,GlMaRi17,Ha02,HaMaSc11,Ng23}), where the noise dimension is assumed to be sufficiently large to ensure that the determining modes are adequately excited. Despite these challenges, we succeed in constructing a generalized coupling and a distance function satisfying the required conditions in the general results of \cite{BuKuSc20,HaMaSc11} by exploiting the specific structural properties of the SEE \eqref{eq_SEE} induced by the non-local operator $\mu[\cdot]$ in its coefficients. These specific structural features are typically lost in abstract settings, such as in \cite{BiBoCaFr25}, where ergodicity for an abstract Markovian lift of an SVE is obtained via direct estimates of the \emph{synchronous coupling} between transition probabilities under much more restrictive conditions mentioned above. It is also worth noting that, as a byproduct of our proof, we derive the \emph{asymptotic log-Harnack inequality} for the SEE \eqref{eq_SEE}. This result extends our previous work \cite[Section 4]{Ha24}, which was limited to scalar kernels, to the more general setting involving matrix-valued kernels, where $K_b$ and $K_\sigma$ may differ and $K_b$ is not necessarily square-integrable. This inequality is of independent interest, as it implies several key properties for the associated Markov semigroup, including the \emph{asymptotic strong Feller property}, \emph{asymptotic irreducibility}, and \emph{uniqueness of the invariant probability measure}, while providing a \emph{gradient estimate} and an \emph{asymptotic heat kernel estimate}; we refer to \cite[Theorem 2.1]{BaWaYu19} for these concepts and general results. The detailed proofs of our main results are deferred to \cref{sec_proof}.

As an application of the aforementioned spectral gap result, \cref{theo_approximation-IPM} demonstrates that the invariant probability measure and, more generally, the stationary law on the path space for the SEE \eqref{eq_SEE} can be weakly approximated by those of finite-dimensional Markovian SDEs. These results translate to the original SVE \eqref{eq_SVE} via \cref{cor_approximation-IPM}, extending previous studies on Markovian approximations of SVEs, such as \cite{AbiJaEu19,AlKe24,BaBr23}. While those works focus on the finite-dimensional approximation of solutions to SVEs over finite time horizons, the approximation of stationary solutions has remained an open problem. Furthermore, our approximation result provides a rigorous justification for the underlying principle of the ``Markovian embedding'' concept, a heuristic yet powerful idea in statistical physics that dates back to the pioneering works of Mori \cite{Mo65} and Zwanzig \cite{Zw73} and has been widely utilized in the analysis of generalized Langevin equations (see, e.g., \cite{Go12,Ku04} and references therein). Specifically, this approach involves representing a non-Markovian process as a projection of a higher-dimensional Markovian system through the introduction of auxiliary variables. This procedure typically assumes that the kernels are expressed as finite sums of exponentials, in which case the associated auxiliary Markovian dynamics are finite-dimensional. In contrast, for general kernels (such as the power-law kernels discussed below), the associated Markovian dynamics become infinite-dimensional and are typically interpreted as ``formal limits'' in the physics literature \cite{Ga09,Ku04}. While the ergodicity of such infinite-dimensional systems has been rigorously established in specific settings \cite{GlHeMcNg20}, the connection between the stationary solutions of finite-dimensional approximations and those of the limiting infinite-dimensional system has not been fully explored. Although our Markovian lift adopts a slightly different formulation from the classical Mori--Zwanzig approach employed in \cite{GlHeMcNg20}, our results provide a rigorous justification for the convergence of stationary solutions associated with these approximations. In this way, our work provides a rigorous mathematical framework to support the validity of Markovian embedding-type procedures across a wide range of physical contexts. In particular, our results are expected to provide a solid mathematical basis for related concepts such as the so-called functional Fokker--Planck equation formally derived in \cite{KaSo20} in the context of statistical physics (see also \cite{HaTh82} and \cite[Chapter 13]{Ga09} for formal developments in the SPDE literature). Detailed statements of our approximation results are provided in \cref{sec_approximation}.

To conclude this introductory section, let us discuss the applicability and limitations of our results, with a view toward applications to the theory of generalized Langevin equations, a central topic in statistical physics. Such equations describe the dynamics of a particle in contact with a heat bath where the friction exhibits a memory effect \cite{KuToHa91,Zw01}. In \cite{LiLiLu17}, the authors derive a one-dimensional over-damped generalized Langevin equation with a power-law kernel. By normalizing physical constants such as the temperature and the Boltzmann constant, the resulting dimensionless equation is expressed as:
\begin{equation}\label{eq_GLE}
	X_t=X_0-\frac{1}{\Gamma(\alpha)}\int^t_0(t-s)^{\alpha-1}V'(X_s)\,\diff s+G_t,\ \ t>0,
\end{equation}
where $\alpha\in(0,1)$ is a constant, $V:\bR\to\bR$ is an external potential, and $G$ is a centered Gaussian process. To satisfy the fluctuation--dissipation theorem (a fundamental principle in statistical physics establishing a precise relationship between the friction kernel and the random force, required for the system to reach thermal equilibrium \cite{Zw01}), the process $G$ must be a fractional Brownian motion with Hurst parameter $\alpha/2$, scaled by the constant factor $\sqrt{2/\Gamma(1+\alpha)}$ (see \cite[Sections 2 and 3]{LiLiLu17}). In this case, by virtue of the Mandelbrot--van Ness representation of fractional Brownian motion (see, e.g., \cite[Theorem 1.3.1]{Mi08}), the equation \eqref{eq_GLE} can be rewritten as
\begin{equation}\label{eq_GLE-FDT}
	X_t=x(t)-\frac{1}{\Gamma(\alpha)}\int^t_0(t-s)^{\alpha-1}V'(X_s)\,\diff s+\frac{\sqrt{2\sin\frac{\pi\alpha}{2}}}{\Gamma(\frac{1+\alpha}{2})}\int^t_0(t-s)^{\frac{\alpha-1}{2}}\,\diff W_s,\ \ t>0,
\end{equation}
with the forcing term
\begin{equation*}
	x(t)=X_0+\frac{\sqrt{2\sin\frac{\pi\alpha}{2}}}{\Gamma(\frac{1+\alpha}{2})}\int^0_{-\infty}\left\{(t-s)^{\frac{\alpha-1}{2}}-(-s)^{\frac{\alpha-1}{2}}\right\}\,\diff W_s,\ \ t>0,
\end{equation*}
where $W=(W_t)_{t\in\bR}$ is a two-sided Brownian motion. This model corresponds to the SVE \eqref{eq_SVE} with fractional (or power-law) kernels $K_b(t)=\frac{1}{\Gamma(\alpha_b)}t^{\alpha_b-1}$ and $K_\sigma(t)=\frac{1}{\Gamma(\alpha_\sigma)}t^{\alpha_\sigma-1}$, where the exponents $\alpha_b\in(0,1)$ and $\alpha_\sigma\in(\frac{1}{2},1)$ are related by $\alpha_b=2\alpha_\sigma-1=\alpha$. Notably, the two kernels $K_b$ and $K_\sigma$ are distinct, and both exhibit a singularity at $t=0$. More importantly, for $\alpha\leq1/2$, a gap in regularity emerges: $K_b$ is only locally integrable, while $K_\sigma$ is locally square-integrable near the origin. While existing works \cite{BeDeKr22,BiBoCaFr25,CaCo98,CuTe20,GaPa25,HaSt19,ViZh19}, including our own \cite{Ha24, Ha25a}, typically require both kernels to be locally square-integrable, the general framework of Markovian lifts established in this paper (\cref{sec_pre}) removes this limitation. Specifically, by extending \cite{Ha24, Ha25a}, it allows for a locally $L^1$ drift kernel $K_b$, which is essential for the analysis of the over-damped generalized Langevin equation but remains beyond the reach of prior methodologies. It should be noted, however, that our main results on ergodicity in \cref{sec_main} and \cref{sec_approximation} require the kernels $K_b(t)$ and $K_\sigma(t)$ to decay exponentially as $t\to\infty$, thereby excluding the pure power-law kernels mentioned above. Instead, our main results apply to tempered fractional kernels (also known as gamma-fractional kernels) $K_b(t)=\frac{1}{\Gamma(\alpha_b)}t^{\alpha_b-1}e^{-\kappa_bt}$ and $K_\sigma(t)=\frac{1}{\Gamma(\alpha_\sigma)}t^{\alpha_\sigma-1}e^{-\kappa_\sigma t}$ with $\alpha_b\in(0,1)$, $\alpha_\sigma\in(\frac{1}{2},1)$ and $\kappa_b,\kappa_\sigma>0$. Such kernels arise in dynamics exhibiting the subdiffusive--diffusive crossover behavior, characterized by subdiffusive behavior in the short-time regime followed by a transition to normal diffusion in the long-time limit \cite{LiSaKa17,MoSaSaPaChMe18,SaSoMeCh17}.

The remainder of this paper is structured as follows. \cref{sec_pre} provides the necessary preliminaries on the generalized Harris' theorem and the Markovian lifting framework. Our main results on ergodicity and their associated consequences are stated in \cref{sec_main}, while the key steps of the proofs for these results are deferred to \cref{sec_proof}. \cref{sec_approximation} is dedicated to the detailed analysis of the finite-dimensional approximation of stationary solutions. Finally, some technical lemmas are collected in the \hyperref[appendix]{Appendix}.

%%%%%%%%%%%%%%%%%%%%%%%%%%%%%%%%%%
%%%%%%%%%%%%%%%%%%%%%%%%%%%%%%%%%%
%% Section
%%%%%%%%%%%%%%%%%%%%%%%%%%%%%%%%%%
%%%%%%%%%%%%%%%%%%%%%%%%%%%%%%%%%%

\section{Preliminaries}\label{sec_pre}

This section provides the necessary preliminaries for our analysis. First, we summarize the notation used throughout this paper. Next, we recall the generalized Harris' theorem developed in \cite{HaMaSc11} and introduce the Markovian lifting framework for SVEs by following and further extending the approach in \cite{Ha24,Ha25a}; these two frameworks serve as the main tools in the subsequent sections.

Throughout this paper, the standard Euclidean norm and inner product on a finite-dimensional Euclidean space are denoted by $|\cdot|$ and $\langle\cdot,\cdot\rangle$, respectively. For a matrix $A$, $A^\top$ denotes its transpose, and $|A|_\op$ denotes its operator norm. For a square matrix $A$, $\tr[A]$ denotes its trace. We denote the $n\times n$ identity matrix by $I_{n\times n}$.

For a separable Banach space $(\cX,\|\cdot\|_\cX)$ and each $T>0$, we denote by $C([0,T];\cX)$ the separable Banach space of $\cX$-valued continuous functions on $[0,T]$ with the norm $\|f\|_{C([0,T];\cX)}:=\sup_{t\in[0,T]} \|f(t)\|_\cX$. Similarly, we define $L^2(0,T;\cX)$ as the separable Banach space of $\diff t$-equivalence classes of square-integrable functions on $(0,T)$ with the norm $\|f\|_{L^2(0,T;\cX)}:=(\int^T_0\|f(t)\|_\cX^2\,\diff t)^{1/2}$. Based on these, we denote by $C([0,\infty);\cX)$ the set of $\cX$-valued continuous functions on $[0,\infty)$, equipped with the metric
\begin{equation}\label{eq_metric-C}
	d_{C([0,\infty);\cX)}(f,g):=\sum_{T\in\bN}\frac{1}{2^T}\left\{\|f-g\|_{C([0,T];\cX)}\wedge1\right\},
\end{equation}
which makes it a complete separable metric space. Similarly, $L^2_\loc(0,\infty;\cX)$ denotes the set of $\diff t$-equivalence classes of Borel measurable maps $f:(0,\infty)\to\cX$ that are square-integrable on every finite interval $(0,T)$. This space is equipped with the metric
\begin{equation}\label{eq_metric-L^2}
	d_{L^2_\loc(0,\infty;\cX)}(f,g):=\sum_{T\in\bN}\frac{1}{2^T}\left\{\|f-g\|_{L^2(0,T;\cX)}\wedge1\right\}.
\end{equation}
This also makes $L^2_\loc(0,\infty;\cX)$ a complete separable metric space. Throughout this paper, we strictly distinguish individual functions from their equivalence classes when considering their distributions.

For normed spaces $\cX_1,\cX_2$, $L(\cX_1,\cX_2)$ denotes the set of bounded linear operators from $\cX_1$ to $\cX_2$. It is equipped with the operator norm $\|\cdot\|_{L(\cX_1,\cX_2)}$. When $\cX_1$ and $\cX_2$ are separable Hilbert spaces, $L_2(\cX_1,\cX_2)\subset L(\cX_1,\cX_2)$ denotes the space of Hilbert--Schmidt operators from $\cX_1$ to $\cX_2$, which is a separable Hilbert space equipped with the Hilbert--Schmidt norm $\|\cdot\|_{L_2(\cX_1,\cX_2)}$.

For each measurable space $(E,\cE)$, $\cP(E)$ denotes the set of all probability measures on $E$. For each $x\in E$, $\delta_x\in\cP(E)$ denotes the Dirac measure at the point $x$. For each $A\in\cE$, $\1_A:E\to\{0,1\}$ denotes the indicator function of $A$. For $\nu_1,\nu_2\in\cP(E)$, $\sC(\nu_1,\nu_2)\subset\cP(E\times E)$ denotes the set of couplings between $\nu_1$ and $\nu_2$; that is, $\nu\in\sC(\nu_1,\nu_2)$ if $\nu$ is a probability measure on the product space $(E\times E,\cE\otimes\cE)$ whose first and second marginals are $\nu_1$ and $\nu_2$, respectively. For $\nu_1,\nu_2\in\cP(E)$, $d_\TV(\nu_1,\nu_2):=\sup_{A\in\cE}|\nu_1(A)-\nu_2(A)|$ denotes the total variation distance, which is known to admit the following coupling representation:
\begin{equation*}
	d_\TV(\nu_1,\nu_2)=\inf_{\nu\in\sC(\nu_1,\nu_2)}\int_{E\times E}\1_{\{x_1\neq x_2\}}\,\nu(\diff x_1,\diff x_2).
\end{equation*}
For $E$-valued random variable $\xi$ on a probability space $(\Omega,\cF,\bP)$, we denote its law under $\bP$ by $\Law_\bP(\xi):=\bP\circ\xi^{-1}\in\cP(E)$. The expectation under $\bP$ is denoted by $\bE_\bP[\cdot]$, or simply by $\bE[\cdot]$ if the underlying probability measure $\bP$ is clear from the context. For each probability kernel $P:E\times\cE\to[0,1]$ on $E$ and each bounded measurable function $f:E\to\bR$, we define $Pf:E\to\bR$ by $(Pf)(x):=\int_Ef(x')\,P(x,\diff x')$ for $x\in E$. This definition extends to measurable functions $f:E\to[0,\infty]$, in which case $Pf$ may take values in $[0,\infty]$. Furthermore, for each $\nu\in\cP(E)$, we define $P^*\nu\in\cP(E)$ by $(P^*\nu)(A):=\int_EP(x,A)\,\nu(\diff x)$ for $A\in\cE$. When $E$ is a topological space, it is always assumed to be equipped with the Borel $\sigma$-algebra $\cE=\cB(E)$.

%%%%%%%%%%%%
%% Subsection
%%%%%%%%%%%%

\subsection{Generalized Harris' theorem}

We now recall the generalized Harris' theorem developed by Hairer, Mattingly, and Scheutzow \cite{HaMaSc11}. This theorem provides a general criterion for spectral gap-type estimates and the weak convergence of Markov semigroups toward their unique invariant probability measures.

Let $E$ be a Polish space. As in \cite{HaMaSc11}, we call a function $d:E\times E\to[0,\infty)$ a \emph{distance-like function} if it is symmetric, lower semi-continuous, and such that $d(x_1,x_2)=0$ if and only if $x_1=x_2$. Given a distance-like function $d:E\times E\to[0,\infty)$, we define $\bW_d:\cP(E)\times\cP(E)\to[0,\infty]$ by
\begin{equation}\label{eq_defi-W_d}
	\bW_d(\nu_1,\nu_2):=\inf_{\nu\in\sC(\nu_1,\nu_2)}\int_{E\times E}d(x_1,x_2)\,\nu(\diff x_1,\diff x_2),\ \ \nu_1,\nu_2\in\cP(E).
\end{equation}
Note that if $d$ is a distance function on $E$, $\bW_d$ coincides with the usual $L^1$-Wasserstein metric. On the other hand, if $d(x_1,x_2)=\1_{\{x_1\neq x_2\}}$, $\bW_d$ coincides with the total variation distance. Following \cite{HaMaSc11}, we adopt the following pair of notions which are central to our analysis.

%% Definition

\begin{defi}
Let $P:E\times\cB(E)\to[0,1]$ be a probability kernel on a Polish space $E$, and let $d:E\times E\to[0,1]$ be a distance-like function (bounded by $1$).
\begin{itemize}
\item
We say that the distance-like function $d$ is \emph{contracting} for $P$ if there exists a constant $\alpha\in(0,1)$ such that
\begin{equation*}
	\bW_d(P(x_1,\cdot),P(x_2,\cdot))\leq\alpha d(x_1,x_2)
\end{equation*}
for any $x_1,x_2\in E$ with $d(x_1,x_2)<1$.
\item
A set $B\subset E$ is said to be \emph{$d$-small} for $P$ if there exists a constant $\ep>0$ such that
\begin{equation*}
	\bW_d(P(x_1,\cdot),P(x_2,\cdot))\leq1-\ep
\end{equation*}
for any $x_1,x_2\in B$.
\end{itemize}
\end{defi}

%% Remark

\begin{rem}
The contractivity condition above imposes no restrictions on pairs $(x_1,x_2)$ with $d(x_1,x_2)=1$, even though this set might be large. Notably, if we take $d(x_1,x_2)=\1_{\{x_1\neq x_2\}}$, which corresponds to the total variation distance, the contractivity is trivially satisfied by \emph{any} probability kernel since the condition $d(x_1,x_2)<1$ implies $x_1=x_2$. In this setting, the notion of a $d$-small set coincides with the notion of a ``small set'', one of the crucial assumptions in the classical Harris' theorem on ergodicity with respect to the total variation distance (see, e.g., \cite{HaMa11,MeTw09}). For more detailed discussions, see \cite[Section 4]{HaMaSc11}. Furthermore, \cite[Example 3.1]{Bu14} provides a simple example of a probability kernel $P$ on $([0,1),\cB([0,1)))$ where the entire state space $[0,1)$ is $d$-small with respect to the Euclidean metric $d$, whereas there are no non-trivial small sets in the classical sense.
\end{rem}

Let $\{P_t\}_{t\geq0}$ be a Markov semigroup, that is, a family of probability kernels on a Polish space $E$ such that $P_0f=f$ and $P_{s+t}f=P_sP_tf$ for any $s,t\geq0$ and any bounded measurable functions $f:E\to\bR$. Below we summarize some fundamental terminology used throughout this paper:
\begin{itemize}
\item
We say that $\{P_t\}_{t\geq0}$ satisfies the \emph{Feller property} if the map $E\ni x\mapsto P_tf(x)$ is continuous for any $t\geq0$ and any bounded continuous function $f:E\to\bR$.
\item
We say that $\{P_t\}_{t\geq0}$ is \emph{measurable} if the map $[0,\infty)\times E\ni(t,x)\mapsto P_tf(x)$ is $\cB([0,\infty))\otimes\cB(E)$-measurable for any bounded measurable function $f:E\to\bR$. This condition is clearly satisfied if $\{P_t\}_{t\geq0}$ is stochastically continuous and has the Feller property; here, \emph{stochastic continuity} means that $\lim_{t\downarrow0}P_tf(x)=f(x)$ for all $x\in E$ and every bounded continuous function $f:E\to\bR$. Indeed, these properties in conjunction with the semigroup property imply that $(t,x)\mapsto P_tf(x)$ is jointly measurable for every bounded continuous function $f:E\to\bR$, which extends to all bounded measurable functions by a standard monotone class argument.
\item
A probability measure $\pi\in\cP(E)$ is called an \emph{invariant probability measure} for $\{P_t\}_{t\geq0}$ if $P^*_t\pi=\pi$ for any $t\geq0$.
\item
A measurable function $V:E\to[0,\infty)$ is called a \emph{Lyapunov function} for $\{P_t\}_{t\geq0}$ if there exist constants $\gamma_V,C_V,K_V>0$ such that the bound
\begin{equation}\label{eq_Lyapunov-def}
	P_tV(x)\leq C_Ve^{-\gamma_Vt}V(x)+K_V
\end{equation}
holds for any $x\in E$ and any $t\geq0$.
\end{itemize}

%% Theorem

\begin{theo}[The generalized Harris' theorem; {\cite[Theorem 4.8 and Corollary 4.11]{HaMaSc11}}]\label{theo_HaMaSc11}
Let $\{P_t\}_{t\geq0}$ be a measurable Markov semigroup on a Polish space $E$ satisfying the Feller property\footnote{In \cite[Corollary 4.11]{HaMaSc11}, the measurability of $\{P_t\}_{t\geq0}$ is implicitly assumed.}. Assume that $\{P_t\}_{t\geq0}$ admits a continuous Lyapunov function $V:E\to[0,\infty)$. Suppose further that there exists a distance-like function $d:E\times E\to[0,1]$ satisfying the following conditions:
\begin{itemize}
\item
$d_0\leq\sqrt{d}$ for some compatible metric $d_0$ on $E$.
\item
There exists $t_1>0$ such that, for every $t\geq t_1$, the distance-like function $d$ is contracting for $P_t$.
\item
There exists $t_2>0$ such that, for every $t\geq t_2$, the level set $\{x\in E\,|\,V(x)\leq4K_V\}$ is $d$-small for $P_t$, where $K_V>0$ is the constant appearing in \eqref{eq_Lyapunov-def}.
\end{itemize}
Then, $\{P_t\}_{t\geq0}$ possesses a unique invariant probability measure $\pi\in\cP(E)$. Furthermore, there exist constants $r>0$ and $t_0>0$ such that the following spectral gap-type estimate holds:
\begin{equation}\label{eq_HaMaSc11-spectral-gap}
	\bW_{d_V}\big(P_t^*\nu_1,P_t^*\nu_2\big)\leq e^{-rt}\bW_{d_V}(\nu_1,\nu_2)
\end{equation}
for all $\nu_1,\nu_2\in\cP(E)$ and all $t\geq t_0$. Here, $d_V:E\times E\to[0,\infty)$ is a distance-like function defined by
\begin{equation*}
	d_V(x_1,x_2):=\sqrt{d(x_1,x_2)\big(1+V(x_1)+V(x_2)\big)}
\end{equation*}
for $x_1,x_2\in E$.
\end{theo}

%% Remark

\begin{rem}
\begin{itemize}
\item[(i)]
If the distance-like function $d:E\times E\to[0,1]$ itself is a compatible metric on $E$, then the first assumption trivially holds.
\item[(ii)]
If $V$ is a Lyapunov function for $\{P_t\}_{t\geq0}$, then every invariant probability measure $\pi$ must satisfy $\int_EV(x)\,\pi(\diff x)<\infty$; see \cite[Lemma 4.1]{Bu14}. Hence, taking $\nu_1=\delta_x$ and $\nu_2=\pi$ in \eqref{eq_HaMaSc11-spectral-gap} immediately yields the following exponential ergodicity: there exists a constant $C>0$ such that
\begin{equation*}
	\bW_{d_V}\big(P_t(x,\cdot),\pi\big)\leq C\big(1+V(x)^{1/2}\big)e^{-rt}
\end{equation*}
for all $x\in E$ and all $t\geq0$.
\item[(iii)]
As shown in \cite[Proposition 2.8]{HaStVo14}, the estimate \eqref{eq_HaMaSc11-spectral-gap} yields the $L^2$-spectral gap estimate
\begin{equation*}
	\sup_{\substack{f\in L^2(E,\pi)\\f\neq\pi f}}\frac{\|P_tf-\pi f\|_{L^2(E,\pi)}}{\|f-\pi f\|_{L^2(E,\pi)}}\leq e^{-rt},
\end{equation*}
provided that the set of bounded, $d_V$-Lipschitz functions is dense in $L^2(E,\pi)$ and the Markov semigroup $\{P_t\}_{t\geq0}$ is reversible with respect to $\pi$.\footnote{Note that the density of the set of bounded, $d_V$-Lipschitz functions in $L^2(E,\pi)$ is automatically satisfied whenever $d_0\leq\sqrt{d}$ for some compatible metric $d_0$ on $E$.} Here, $(L^2(E,\pi),\|\cdot\|_{L^2(E,\pi)})$ denotes the Hilbert space of (equivalence classes of) $\pi$-square-integrable real-valued functions, and $\pi f:=\int_Ef(x)\,\pi(\diff x)$.
\end{itemize}
\end{rem}

In order to apply the generalized Harris' theorem, the crucial and practically challenging step is to construct a distance(-like) function $d$ which satisfies all the requirements in \cref{theo_HaMaSc11}; the construction of a Lyapunov function $V$ is equally essential and is typically addressed separately. In \cite{BuKuSc20}, the authors provide a set of verifiable sufficient conditions for the required assumptions in \cref{theo_HaMaSc11} in terms of the so-called \emph{generalized coupling} between the transition probabilities $P_t(x_1,\cdot)$ and $P_t(x_2,\cdot)$ with $x_1,x_2\in E$. Hence, the problem we have to address reduces to how to construct a ``good'' generalized coupling and a distance(-like) function $d$. However, this latter task remains non-trivial, as it requires a deep analysis depending on each model. We address this issue within our framework of Markovian lifts of SVEs, to be introduced in the next subsection, where the specific structure of the lifted SEE \eqref{eq_SEE} plays a crucial role. For more details, see the discussion in \cref{subsec_proof-coupling}.

%%%%%%%%%%%%
%% Subsection
%%%%%%%%%%%%

\subsection{Markovian lifts of SVEs}

In this subsection, we recall and further generalize the Markovian lifting framework developed in our previous works \cite{Ha24,Ha25a} (see also \cite{Ha25b}). Before doing so, let us present the following standard well-posedness result for the SVE \eqref{eq_SVE}.

%% Proposition

\begin{prop}\label{prop_SVE}
Let $K_b,K_\sigma:(0,\infty)\to\bR^{n\times n}$ be measurable maps such that $\int^T_0|K_b(t)|_\op\,\diff t<\infty$ and $\int^T_0|K_\sigma(t)|_\op^2\,\diff t<\infty$ for any $T>0$. Suppose that $b:\bR^n\to\bR^n$ and $\sigma:\bR^n\to\bR^{n\times d}$ are measurable maps such that $|b(x)-b(x')|\leq C_{b,\Lip}|x-x'|$ and $|\sigma(x)-\sigma(x')|\leq C_{\sigma,\Lip}|x-x'|$ for all $x,x'\in\bR^n$ with some constants $C_{b,\Lip},C_{\sigma,\Lip}>0$. Assume we are given a filtered probability space $(\Omega,\cF,\bF,\bP)$ satisfying the usual conditions, a $d$-dimensional $\bF$-Brownian motion $W$, and an $\cF_0\otimes\cB((0,\infty))$-measurable map $x:\Omega\times(0,\infty)\to\bR^n$ such that $\int^T_0|x(t)|^2\,\diff t<\infty$ a.s.\ for any $T>0$. Then, there exists a unique progressively measurable process $X=(X_t)_{t>0}$ satisfying $\int^T_0|X_t|^2\,\diff t<\infty$ a.s.\ for any $T>0$ such that the equality in \eqref{eq_SVE} holds for $\diff t\otimes\diff\bP$-a.e.\ $(t,\omega)\in(0,\infty)\times\Omega$. The uniqueness is understood in the sense that any two such solutions coincide $\diff t\otimes\diff\bP$-a.e.\ on $(0,\infty)\times\Omega$.
\end{prop}

The above result can be proved by applying \cite[Proposition 2.4]{Ha21} (see also \cite[Theorem 3.1]{Zh10}) in conjunction with a standard localization argument with respect to the forcing term $x$; we omit the details here for brevity as the procedure is standard in the context of stochastic analysis. We view the solution space of the SVE \eqref{eq_SVE} as the Polish space $L^2_\loc(0,\infty;\bR^n)$ equipped with the complete metric defined by \eqref{eq_metric-L^2}. By virtue of the Yamada--Watanabe-type result in \cite{Ku14}, uniqueness in law holds for the SVE \eqref{eq_SVE} under the setting of \cref{prop_SVE}. Consequently, the law of the $\diff t$-equivalence class of the solution $X$ on $L^2_\loc(0,\infty;\bR^n)$ is uniquely determined by the law of the $\diff t$-equivalence class of the forcing term $x$.

%% Remark

\begin{rem}
In the setting of \cref{prop_SVE} where the kernels $K_b$ and $K_\sigma$ are assumed to be merely (locally) $L^1$ and $L^2$, we cannot expect pathwise time regularity for the solution $X$. To recover pathwise regularity, additional conditions on the kernels $K_b$ and $K_\sigma$ are required; see, e.g., \cite[Theorem 3.3]{Zh10}.
\end{rem}

We now present a Markovian lifting framework. First, let us introduce the following definition.

%% Definition

\begin{defi}\label{defi_lifting-basis}
We call a triplet $(\mu,M_b,M_\sigma)$ a \emph{lifting basis} if $\mu$ is a Borel measure on $[0,\infty)$, $M_b,M_\sigma:[0,\infty)\to\bR^{n\times n}$ are matrix-valued Borel measurable maps, and they satisfy
\begin{align}
	\label{eq_integrability-mu}
	&\int_{[0,\infty)}(1+\theta)^{-1/2}\dmu<\infty,\\
	\label{eq_integrability-M_b}
	&\int_{[0,\infty)}(1+\theta)^{-3/2}|M_b(\theta)|_\op^2\dmu<\infty\ \ \text{and}\\
	\label{eq_integrability-M_sigma}
	&\int_{[0,\infty)}(1+\theta)^{-1/2}|M_\sigma(\theta)|_\op^2\dmu<\infty.
\end{align}
We say that a pair of kernels $K_b,K_\sigma:(0,\infty)\to\bR^{n\times n}$ is \emph{liftable} if there exists a lifting basis $(\mu,M_b,M_\sigma)$ such that
\begin{equation}\label{eq_liftable}
	K_b(t)=\int_{[0,\infty)}e^{-\theta t}M_b(\theta)\dmu\ \ \text{and}\ \ K_\sigma(t)=\int_{[0,\infty)}e^{-\theta t}M_\sigma(\theta)\dmu
\end{equation}
for any $t\in(0,\infty)$. In this case, we say that the lifting basis $(\mu,M_b,M_\sigma)$ generates the pair $(K_b,K_\sigma)$.
\end{defi}

%% Remark

\begin{rem}
\begin{itemize}
\item[(i)]
A slight modification of \cite[Lemma 4.1]{Ha25b} shows that a pair $(K_b,K_\sigma)$ of kernels is liftable if and only if each of the matrix components $K_b^{i,j},K_\sigma^{i,j}:(0,\infty)\to\bR$ for $i,j\in\{1,\dots,n\}$ is of the form
\begin{equation*}
	K_b^{i,j}=K_b^{i,j,+}-K_b^{i,j,-}\ \ \text{and}\ \ K_\sigma^{i,j}=K_\sigma^{i,j,+}-K_\sigma^{i,j,-}
\end{equation*}
for some completely monotone functions $K_b^{i,j,\pm},K_\sigma^{i,j,\pm}:(0,\infty)\to[0,\infty)$ such that
\begin{equation}\label{eq_integrability-K}
	\int^T_0K_b^{i,j,\pm}(t)\,\diff t<\infty\ \ \text{and}\ \ \int^T_0t^{-1/2}K_\sigma^{i,j,\pm}(t)\,\diff t<\infty
\end{equation}
for any $T\in(0,\infty)$. In particular, by virtue of \cite[Lemma 2.1]{Ha24}, every liftable pair $(K_b,K_\sigma)$ satisfies
\begin{equation}\label{eq_integrability-K'}
	\int^T_0|K_b(t)|_\op\,\diff t<\infty\ \ \text{and}\ \ \int^T_0|K_\sigma(t)|_\op^2\,\diff t<\infty
\end{equation}
for any $T\in(0,\infty)$.
\item[(ii)]
If $(K_{b,1},K_{\sigma,1})$ and $(K_{b,2},K_{\sigma,2})$ are liftable, then for any $A_{b,1},A_{b,2},A_{\sigma,1},A_{\sigma,2}\in\bR^{n\times n}$, the pair $(A_{b,1}K_{b,1}+A_{b,2}K_{b,2},A_{\sigma,1}K_{\sigma,1}+A_{\sigma,2}K_{\sigma,2})$ is also liftable. Indeed, given a lifting basis $(\mu_i,M_{b,i},M_{\sigma,i})$ which generates $(K_{b,i},K_{\sigma,i})$ for each $i\in\{1,2\}$, the lifting basis $(\mu,M_b,M_\sigma)$ given by
\begin{equation*}
	\mu=\mu_1+\mu_2,\ \ M_b=\frac{\diff\mu_1}{\diff\mu}A_{b,1}M_{b,1}+\frac{\diff\mu_2}{\diff\mu}A_{b,2}M_{b,2},\ \ M_\sigma=\frac{\diff\mu_1}{\diff\mu}A_{\sigma,1}M_{\sigma,1}+\frac{\diff\mu_2}{\diff\mu}A_{\sigma,2}M_{\sigma,2},
\end{equation*}
generates $(A_{b,1}K_{b,1}+A_{b,2}K_{b,2},A_{\sigma,1}K_{\sigma,1}+A_{\sigma,2}K_{\sigma,2})$.
\item[(iii)]
For a given liftable pair $(K_b,K_\sigma)$ of kernels, a lifting basis $(\mu,M_b,M_\sigma)$ generating it is not unique. Indeed, there is flexibility in the choice of the measure $\mu$; however, once a measure $\mu$ is given, the maps $M_b$ and $M_\sigma$ satisfying \eqref{eq_liftable} are determined up to a $\mu$-null set.
\item[(iv)]
The seemingly technical set of integrability conditions \eqref{eq_integrability-mu}, \eqref{eq_integrability-M_b} and \eqref{eq_integrability-M_sigma} for a lifting basis $(\mu,M_b,M_\sigma)$ relates to the conditions \eqref{eq_integrability-K} and \eqref{eq_integrability-K'} for the generated kernels $K_b$ and $K_\sigma$. This set of conditions is natural when we consider the case where $K_b=K_\sigma=KI_{n\times n}$ for some scalar completely monotone function $K:(0,\infty)\to[0,\infty)$, a case adopted in \cite{Ha24,Ha25a}. Indeed, in this case, Bernstein's theorem ensures that there exists a \emph{unique} Borel measure $\mu$ on $[0,\infty)$ such that $K(t)=\int_{[0,\infty)}e^{-\theta t}\dmu$ for any $t>0$. By \cite[Lemma 2.1]{Ha24}, the condition $\int^T_0t^{-1/2}K(t)\,\diff t<\infty$ for any $T\in(0,\infty)$ is equivalent to the integrability condition \eqref{eq_integrability-mu} for the measure $\mu$; furthermore, both conditions imply that $\int^T_0K(t)^2\,\diff t<\infty$ for any $T\in(0,\infty)$, which is a necessary condition for the Volterra-type stochastic integral $\int^t_0K(t-s)\,\diff W_s$ to be well-defined. In this case, the triplet $(\mu,I_{n\times n},I_{n\times n})$ is a lifting basis which generates $(KI_{n\times n},KI_{n\times n})$ in the sense of \cref{defi_lifting-basis}. Our set of integrability conditions \eqref{eq_integrability-mu}, \eqref{eq_integrability-M_b} and \eqref{eq_integrability-M_sigma} thus serves as a natural generalization of the benchmark case to settings where the two matrix-valued kernels $K_b$ and $K_\sigma$ may differ and $K_b$ is only (locally) $L^1$. Similar integrability conditions are also employed in \cite{Ha25b}.
%Also, there are any other choices of a tuple $(\mu,M_b,M_\sigma)$ which satisfies \eqref{eq_liftable} but does not satisfy the set of integrability conditions \eqref{eq_integrability-mu}, \eqref{eq_integrability-M_b} and \eqref{eq_integrability-M_sigma}. However, as will be seen in the subsequent discussion, these seemingly technical integrability conditions play crucial roles in constructing a tractable framework of Markovian lifts of SVEs with matrix-valued kernels. Similar integrability conditions are assumed in \cite{Ha25b} as well.
\end{itemize}
\end{rem}

%% Example

\begin{exam}\label{exam_liftable}
\begin{itemize}
\item[(i)]
(Sum-of-exponentials type kernels)
Consider the kernels
\begin{equation*}
	K^{\exp}_b(t)=\sum^N_{i=1}e^{-\kappa_it}M_{b,i},\ \ K^{\exp}_\sigma(t)=\sum^N_{i=1}e^{-\kappa_it}M_{\sigma,i},\ \ t>0,
\end{equation*}
where $M_{b,i},M_{\sigma,i}\in\bR^{n\times n}$ for each $i\in\{1,\dots,N\}$ with $N\in\bN$, and $\kappa_i\geq0$ are distinct constants. Such kernels are frequently employed to approximate more general kernels, as seen in the studies of the Markovian embedding for generalized Langevin equations \cite{Go12,Ku04} and multi-factor approximation of rough volatility models \cite{AbiJaEu19,AlKe24,BaBr23}. Clearly, the pair of kernels $(K^{\exp}_b,K^{\exp}_\sigma)$ above is liftable. As a lifting basis $(\mu,M_b,M_\sigma)$ which generates this pair, we can take
\begin{equation*}
	\mu=\sum^N_{i=1}\delta_{\kappa_i},\ \ M_b(\theta)=\sum^N_{i=1}\1_{\{\kappa_i\}}(\theta)M_{b,i},\ \ M_\sigma(\theta)=\sum^N_{i=1}\1_{\{\kappa_i\}}(\theta)M_{\sigma,i}.
\end{equation*}
In this case, we have $\supp=\{\kappa_i\}^N_{i=1}$.
\item[(ii)]
(Tempered fractional kernels, also known as gamma-fractional kernels)
Consider the kernels
\begin{equation*}
	K^\fractional_b(t)=\frac{1}{\Gamma(\alpha_b)}t^{\alpha_b-1}e^{-\kappa_bt}I_{n\times n},\ \ K^\fractional_\sigma(t)=\frac{1}{\Gamma(\alpha_\sigma)}t^{\alpha_\sigma-1}e^{-\kappa_\sigma t}I_{n\times n},\ \ t>0,
\end{equation*}
where $\alpha_b\in(0,1)$, $\alpha_\sigma\in(\frac{1}{2},1)$ and $\kappa_b,\kappa_\sigma\in[0,\infty)$. The case $\kappa_b=\kappa_\sigma=0$ corresponds to the fractional kernels appearing in the Riemann--Liouville and Caputo fractional derivatives \cite{SaKiMa87}, which arise in the over-damped generalized Langevin equation \eqref{eq_GLE-FDT}. In contrast, the case where $\kappa_b$ and/or $\kappa_\sigma$ are positive corresponds to exponentially tempered fractional kernels. Such kernels arise in dynamics exhibiting the subdiffusive--diffusive crossover behavior \cite{LiSaKa17,MoSaSaPaChMe18,SaSoMeCh17}. Notably, unlike $\alpha_\sigma$, we allow $\alpha_b$ to be smaller than $1/2$, so that $K^\fractional_b$ is not necessarily locally square-integrable. This is crucial in application to the over-damped generalized Langevin equation \eqref{eq_GLE-FDT} derived in \cite{LiLiLu17}, where the exponents $\alpha_b$ and $\alpha_\sigma$ are related by $\alpha_b=2\alpha_\sigma-1\in(0,1)$ in view of the fluctuation--dissipation theorem. The pair of kernels $(K^\fractional_b,K^\fractional_\sigma)$ above is liftable. To see this, let $\gamma_b$ and $\gamma_\sigma$ be constants such that
\begin{equation*}
	(2\alpha_b-1)\vee\frac{1}{2}<\gamma_b<\left(2\alpha_b+\frac{1}{2}\right)\wedge1\ \ \text{and}\ \ (2\alpha_\sigma-1)\vee\frac{1}{2}<\gamma_\sigma<\left(2\alpha_\sigma-\frac{1}{2}\right)\wedge1.
\end{equation*}
For example, we can take $\gamma_b=\frac{\alpha_b+1}{2}$ and $\gamma_\sigma=\alpha_\sigma$, but now we keep the above flexibility. Set
\begin{align*}
	&\mu(\diff\theta)=\left\{(\theta-\kappa_b)^{-\gamma_b}\1_{(\kappa_b,\infty)}(\theta)+(\theta-\kappa_\sigma)^{-\gamma_\sigma}\1_{(\kappa_\sigma,\infty)}(\theta)\right\}\,\diff\theta,\\
	&M_b(\theta)=\Gamma(\alpha_b)^{-1}\Gamma(1-\alpha_b)^{-1}\left\{(\theta-\kappa_b)^{-\gamma_b}+(\theta-\kappa_\sigma)^{-\gamma_\sigma}\1_{(\kappa_\sigma,\infty)}(\theta)\right\}^{-1}(\theta-\kappa_b)^{-{\alpha_b}}\1_{(\kappa_b,\infty)}(\theta)I_{n\times n},\\
	&M_\sigma(\theta)=\Gamma(\alpha_\sigma)^{-1}\Gamma(1-\alpha_\sigma)^{-1}\left\{(\theta-\kappa_b)^{-\gamma_b}\1_{(\kappa_b,\infty)}(\theta)+(\theta-\kappa_\sigma)^{-\gamma_\sigma}\right\}^{-1}(\theta-\kappa_\sigma)^{-{\alpha_\sigma}}\1_{(\kappa_\sigma,\infty)}(\theta)I_{n\times n}.
\end{align*}
By a direct computation, one can show that $(\mu,M_b,M_\sigma)$ is a lifting basis which generates $(K^\fractional_b,K^\fractional_\sigma)$ in the sense of \cref{defi_lifting-basis}. In this case, we have $\supp=[\kappa_b\wedge\kappa_\sigma,\infty)$.
\end{itemize}
\end{exam}

Let $\mu$ be a Borel measure on $[0,\infty)$ satisfying \eqref{eq_integrability-mu}. We denote by $\cH=\cH_\mu$ the set of all $\mu$-equivalence classes of Borel measurable maps $y:[0,\infty)\to\bR^n$ such that
\begin{equation*}
	\|y\|_\cH:=\left(\int_{[0,\infty)}(1+\theta)^{-1/2}|y(\theta)|^2\dmu\right)^{1/2}<\infty.
\end{equation*}
We equip $\cH$ with the inner product
\begin{equation*}
	\langle y_1,y_2\rangle_\cH:=\int_{[0,\infty)}(1+\theta)^{-1/2}\langle y_1(\theta),y_2(\theta)\rangle\dmu,\ \ y_1,y_2\in\cH,
\end{equation*}
which induces the norm $\|\cdot\|_\cH$ on $\cH$. Then, $(\cH,\|\cdot\|_\cH,\langle\cdot,\cdot\rangle_\cH)$ is a separable Hilbert space. We identify the topological dual $\cH^*$ of $\cH$ with $\cH$ itself by the Riesz isomorphism $y\mapsto\langle y,\cdot\rangle_\cH$. Also, we denote by $\cV=\cV_\mu$ the set of all $\mu$-equivalence classes of Borel measurable maps $y:[0,\infty)\to\bR^n$ such that
\begin{equation*}
	\|y\|_\cV:=\left(\int_{[0,\infty)}(1+\theta)^{1/2}|y(\theta)|^2\dmu\right)^{1/2}<\infty.
\end{equation*}
We equip $\cV$ with the norm $\|\cdot\|_\cV$. Then $(\cV,\|\cdot\|_\cV)$ is a separable and reflexive Banach space, and the embedding $\cV\hookrightarrow\cH$ is continuous and dense. Under the identification of $\cH^*$ with $\cH$, we have a continuous and dense embedding $\cH\hookrightarrow\cV^*$, where the topological dual $\cV^*$ of $\cV$ is now identified with the space of all $\mu$-equivalence classes of Borel measurable maps $y:[0,\infty)\to\bR^n$ such that
\begin{equation*}
	\|y\|_{\cV^*}:=\left(\int_{[0,\infty)}(1+\theta)^{-3/2}|y(\theta)|^2\dmu\right)^{1/2}<\infty.
\end{equation*}
Note that the duality pairing $\langle\cdot,\cdot\rangle_{\cV^*,\cV}$ between $\cV^*$ and $\cV$ is compatible with the inner product $\langle\cdot,\cdot\rangle_\cH$ on $\cH$ in the sense that $\langle y_1,y_2\rangle_{\cV^*,\cV}=\langle y_1,y_2\rangle_\cH$ whenever $y_1\in\cH\subset\cV^*$ and $y_2\in\cV\subset\cH$. Hence, we have a \emph{Gelfand triplet} $\cV\hookrightarrow\cH\hookrightarrow\cV^*$.

As shown in \cite[Lemma 2.5]{Ha24}, the space $\cV$ is continuously embedded into $L^1(\mu)=L^1(\mu;\bR^n)$, but there is no relation between $\cH$ and $L^1(\mu)$ in general. For each $y\in\cV$, we set
\begin{equation*}
	\mu[y]:=\int_{[0,\infty)}y(\theta)\dmu.
\end{equation*}
Then, the map $\mu[\cdot]:\cV\to\bR^n$ is a bounded linear operator. Moreover, as shown in \cite[Lemma 2.5]{Ha24}, the following holds:

%% Lemma

\begin{lemm}\label{lemm_mu-ep}
For any $\ep>0$, there exists a constant $C_{\mu,\ep}>0$ such that
\begin{equation*}
	\big|\mu[y]\big|^2\leq\ep\|y\|_\cV^2+C_{\mu,\ep}\|y\|_\cH^2
\end{equation*}
for any $y\in\cV$.
\end{lemm}

%% Remark

\begin{rem}
\begin{itemize}
\item[(i)]
Since $\|y_1+y_2\|_\cV^2+\|y_1-y_2\|_\cV^2=2(\|y_1\|_\cV^2+\|y_2\|_\cV^2)$ for any $y_1,y_2\in\cV$, the Banach space $(\cV,\|\cdot\|_\cV)$ becomes a Hilbert space equipped with the inner product which induces the norm $\|\cdot\|_\cV$. Similarly, the Banach space $(\cV^*,\|\cdot\|_{\cV^*})$ becomes a Hilbert space with the inner product which induces the norm $\|\cdot\|_{\cV^*}$.
\item[(ii)]
We set $\|y\|_\cV=\infty$ for each $y\in\cH\setminus\cV$. Then, the function $\|\cdot\|_\cV:\cH\to[0,\infty]$ is lower semi-continuous and hence $\cB(\cH)$-measurable (see, e.g., \cite[Exercise 4.2.3]{LiRo15}).
\item[(iii)]
We have $\cB(\cV)=\{A\cap\cV\,|\,A\in\cB(\cH)\}\subset\cB(\cH)$ and $\cB(\cH)=\{A\cap\cH\,|\,A\in\cB(\cV^*)\}\subset\cB(\cV^*)$, where the Borel $\sigma$-algebras are generated by their respective norm topology. Indeed, since the natural embedding $i:\cV\to\cH$ is continuous and injective, the Lusin--Suslin theorem (see, e.g., \cite[Theorem 15.1]{Ke95}) implies that $B=i(B)\in\cB(\cH)$ for any $B\in\cB(\cV)$, and hence $\cB(\cV)\subset\{A\cap\cV\,|\,A\in\cB(\cH)\}\subset\cB(\cH)$. On the other hand, again by the continuity of $i:\cV\to\cH$, we have $\{A\cap\cV\,|\,A\in\cB(\cH)\}\subset\cB(\cV)$, and hence $\cB(\cV)=\{A\cap\cV\,|\,A\in\cB(\cH)\}\subset\cB(\cH)$. The relations $\cB(\cH)=\{A\cap\cH\,|\,A\in\cB(\cV^*)\}\subset\cB(\cV^*)$ can be proved similarly.
\end{itemize}
\end{rem}

Let $(\mu,M_b,M_\sigma)$ be a lifting basis, and let $b:\bR^n\to\bR^n$ and $\sigma:\bR^n\to\bR^{n\times d}$ be Borel measurable maps. We regard \eqref{eq_SEE} as an SEE on the Gelfand triplet $\cV\hookrightarrow\cH\hookrightarrow\cV^*$ specified above. Before defining the notion of a solution to the SEE \eqref{eq_SEE}, let us make some remarks on the coefficients.

%% Remark

\begin{rem}\label{rem_SEE-framework}
\begin{itemize}
\item[(i)]
By the integrability condition \eqref{eq_integrability-M_b}, the map $\cM_b:x\mapsto(\theta\mapsto M_b(\theta)x)$ is a bounded linear operator from $\bR^n$ to $\cV^*$, and we have
\begin{equation*}
	\|\cM_b\|_{L(\bR^n;\cV^*)}\leq\left(\int_{[0,\infty)}(1+\theta)^{-3/2}|M_b(\theta)|_\op^2\dmu\right)^{1/2}<\infty.
\end{equation*}
Similarly, by \eqref{eq_integrability-M_sigma}, the map $\cM_\sigma:x\mapsto(\theta\mapsto M_\sigma(\theta)x)$ is a bounded linear operator from $\bR^n$ to $\cH$, and we have
\begin{equation*}
	\|\cM_\sigma\|_{L(\bR^n;\cH)}\leq\left(\int_{[0,\infty)}(1+\theta)^{-1/2}|M_\sigma(\theta)|_\op^2\dmu\right)^{1/2}<\infty.
\end{equation*}
\item[(ii)]
For each $t\geq0$, define $\cS(t):\cH\to\cH$ by $(\cS(t)y)(\theta):=e^{-\theta t}y(\theta)$, $\theta\in[0,\infty)$, for $y\in\cH$. As shown in \cite[Lemma 2.5]{Ha24}, $\{S(t)\}_{t\geq0}$ is a contraction semigroup on $\cH$, and the associated infinitesimal generator $\cA:\cD(\cA)\subset\cH\to\cH$ is given by
\begin{align*}
	&\cD(\cA)=\left\{y\in\cH\relmiddle|\int_{[0,\infty)}(1+\theta)^{3/2}|y(\theta)|^2\dmu<\infty\right\},\\
	&(\cA y)(\theta)=-\theta y(\theta),\ \ \theta\in[0,\infty),\ \ \text{for $y\in\cD(\cA)$}.
\end{align*}
Moreover, it holds that
\begin{equation}\label{eq_A-nonpositive}
	\langle\cA y,y\rangle_\cH=\|y\|_\cH^2-\|y\|_\cV^2\leq0
\end{equation}
for any $y\in\cD(\cA)$. Furthermore, $\cS(t)$ is a bounded linear operator from $\cH$ to $\cV$ for any $t>0$, and $\cS(\cdot):y\mapsto(\cS(t)y)_{t\in(0,T]}$ is a bounded linear operator from $\cH$ to $L^2(0,T;\cV)$ for any $T>0$. For more details, see \cite[Section 2]{Ha24}.
\item[(iii)]
For each $t\geq0$, it is easy to see that $\|\cS(t)y\|_{\cV^*}\leq\|y\|_{\cV^*}$ for any $y\in\cH$. Consequently, the operator $\cS(t):\cH\to\cH$ (viewed as a map into $\cV^*$) uniquely extends to a contraction linear operator $\widetilde{\cS}(t):\cV^*\to\cV^*$. The family $\{\widetilde{\cS}(t)\}_{t\geq0}$ then forms a contraction semigroup on the larger Hilbert space $\cV^*$. Following the same argument as in the proof of \cite[Lemma 2.5]{Ha24}, we can show that the infinitesimal generator $\widetilde{\cA}:\cD(\widetilde{\cA})\subset\cV^*\to\cV^*$ of $\{\widetilde{\cS}(t)\}_{t\geq0}$ is given by
\begin{align*}
	&\cD(\widetilde{\cA})=\left\{y\in\cV^*\relmiddle|\int_{[0,\infty)}(1+\theta)^{-3/2}\theta^2|y(\theta)|^2\dmu<\infty\right\},\\
	&(\widetilde{\cA}y)(\theta)=-\theta y(\theta),\ \ \theta\in[0,\infty),\ \ \text{for $y\in\cD(\widetilde{\cA})$}.
\end{align*}
Note that $\cD(\widetilde{\cA})=\cV$ with the norm equivalence $\frac{1}{\sqrt{2}}\|y\|_\cV\leq\|y\|_{\cD(\widetilde{\cA})}\leq\|y\|_\cV$ for any $y\in\cV$, where $\|y\|_{\cD(\widetilde{\cA})}:=(\|\widetilde{\cA}y\|_{\cV^*}^2+\|y\|_{\cV^*}^2)^{1/2}$. The operator $\widetilde{\cA}:\cV\to\cV^*$ is the unique extension of the operator $\cA:\cD(\cA)\to\cH$ (viewed as a map into $\cV^*$). Combining the equality \eqref{eq_A-nonpositive}, the density of $\cD(\cA)$ in $\cV$, and the compatibility of the duality pairing ${}_{\cV^*}\langle\cdot,\cdot\rangle_\cV$ with the inner product $\langle\cdot,\cdot\rangle_\cH$ in $\cH$, we obtain
\begin{equation}\label{eq_A-nonpositive'}
	\langle\widetilde{\cA}y,y\rangle_{\cV^*,\cV}=\|y\|_\cH^2-\|y\|_\cV^2\leq0
\end{equation}
for any $y\in\cV$. For simplicity, in what follows, the operators $\widetilde{\cS}(t)$ and $\widetilde{\cA}$ will be denoted again by $\cS(t)$ and $\cA$, respectively.
\end{itemize}
\end{rem}

%% Definition

\begin{defi}\label{defi_SEE-solution}
Let $(\mu,M_b,M_\sigma)$ be a lifting basis, and let $b:\bR^n\to\bR^n$ and $\sigma:\bR^n\to\bR^{n\times d}$ be Borel measurable maps. Suppose that we are given a $d$-dimensional Brownian motion $W$ on a complete probability space $(\Omega,\cF,\bP)$, with respect to a filtration $\bF=(\cF_t)_{t\geq0}$ satisfying the usual conditions. We say that an $\cH$-valued adapted process $Y$ is a solution of the SEE \eqref{eq_SEE} if the following hold:
\begin{itemize}
\item[(i)]
The sample path $t\mapsto Y_t$ is strongly continuous in $\cH$ a.s., and satisfies $\int^T_0\|Y_t\|_\cV^2\,\diff t<\infty$ for any $T>0$ a.s.
\item[(ii)]
For any $T>0$, $\int^T_0\big\{|b(\mu[Y_t])|+|\sigma(\mu[Y_t])|^2\big\}\,\diff t<\infty$ a.s.
\item[(iii)]
The following equality holds in $\cV^*$ for any $t\geq0$ a.s.:
\begin{equation}\label{eq_SEE-sol-strong}
	Y_t=Y_0+\int^t_0\big\{\cA Y_s+\cM_bb(\mu[Y_s])\big\}\,\diff s+\int^t_0\cM_\sigma\sigma(\mu[Y_s])\,\diff W_s.
\end{equation}
\end{itemize}
\end{defi}

%% Remark

\begin{rem}\label{rem_SEE-solution}
\begin{itemize}
\item[(i)]
By the progressive measurability of the $\cH$-valued process $Y$ and the integrability condition in \cref{defi_SEE-solution} (i), combined with the fact that $\cB(\cV)=\{A\cap\cV\,|A\in\cB(\cH)\}\subset\cB(\cH)$, the process $\widetilde{Y}:=Y\1_\cV(Y)$ is a $\cV$-valued progressively measurable process such that $Y_t=\widetilde{Y}_t$ in $\cH$ for $\diff t\otimes\diff\bP$-a.e.\ $(t,\omega)\in[0,\infty)\times\Omega$. The processes $\mu[Y_t]$ and $\cA Y_t$ are then defined by $\mu[\widetilde{Y}_t]$ and $\cA\widetilde{Y}_t$, which are $\bR^n$-valued and $\cV^*$-valued progressively measurable processes, respectively. For simplicity, we denote this version $\widetilde{Y}$ again by $Y$ when no confusion can arise.
\item[(ii)]
In the right-hand side of \eqref{eq_SEE-sol-strong}, the integral with respect to $\diff s$ is understood as a Bochner integral in $\cV^*$, while the integral with respect to $\diff W_s$ is understood as a stochastic integral in $\cH$. These integrals are well-defined thanks to the integrability conditions (i) and (ii) in \cref{defi_SEE-solution}; see also \cref{rem_SEE-framework}. The solution $Y$ is an $\cH$-valued continuous adapted process, but it is not necessarily a semimartingale in $\cH$; it can only be regarded as a semimartingale when viewed as a process in $\cV^*$. This is because the $\cV^*$-valued Bochner integral in \eqref{eq_SEE-sol-strong} (which indeed takes values in $\cH$) is not necessarily of bounded variation as an $\cH$-valued process.
\item[(iii)]
The definition of the solution to the SEE \eqref{eq_SEE} follows the standard ``variational approach'' for monotone SPDEs on a Gelfand triplet, a framework initiated by Pardoux \cite{Pa72,Pa75} and further developed by Krylov and Rozovski\u{\i} \cite{KrRo79}; see also the monographs \cite{GaMa10,LiRo15,Pa21}. Indeed, as demonstrated in the proof of \cref{prop_SEE-well-posed} below, the SEE \eqref{eq_SEE} with Lipschitz continuous coefficients $b:\bR^n\to\bR^n$ and $\sigma:\bR^n\to\bR^{n\times d}$ satisfies the standard ``monotonicity conditions'' on the Gelfand triplet $\cV\hookrightarrow\cH\hookrightarrow\cV^*$. Solutions in the sense of \cref{defi_SEE-solution} are commonly referred to as ``variational solutions'' in the literature. Furthermore, such solutions can be interpreted as ``analytically strong solutions'' on the largest Hilbert space $\cV^*$ by viewing the operator $\cA:\cV\to\cV^*$ as the infinitesimal generator of the contraction semigroup $\{\cS(t)\}_{t\geq0}$ on $\cV^*$. In this setting, the coefficients $y\mapsto\cM_bb(\mu[y\1_\cV(y)])$ and $y\mapsto \cM_\sigma\sigma(\mu[y\1_\cV(y)])$ are regarded as measurable maps from $\cV^*$ to $\cV^*$ and $L_2(\bR^d;\cV^*)$, respectively. However, it should be noted that these coefficients are generally not continuous as maps from $\cV^*$ (or even from $\cH$) to their respective target spaces.
\end{itemize}
\end{rem}

%% Proposition

\begin{prop}\label{prop_SEE-well-posed}
Suppose that we are given a $d$-dimensional Brownian motion $W$ defined on a complete probability space $(\Omega,\cF,\bP)$, with respect to a filtration $\bF=(\cF_t)_{t\geq0}$ satisfying the usual conditions. Let $(\mu,M_b,M_\sigma)$ be a lifting basis, and let $b:\bR^n\to\bR^n$ and $\sigma:\bR^n\to\bR^{n\times d}$ be measurable maps. Assume that there exist constants $C_{b,\Lip},C_{\sigma,\Lip}>0$ such that $|b(x)-b(x')|\leq C_{b,\Lip}|x-x'|$ and $|\sigma(x)-\sigma(x')|\leq C_{\sigma,\Lip}|x-x'|$ for all $x,x'\in\bR^n$. Then, the following assertions hold:
\begin{itemize}
\item[(i)]
For any $\cH$-valued, $\cF_0$-measurable random variable $Y_0$, there exists a unique solution $Y=(Y_t)_{t\geq0}$ to the SEE \eqref{eq_SEE} with the prescribed initial condition $Y_0$. Furthermore, this solution forms a time-homogeneous Markov process on $\cH$.
\item[(ii)]
There exists a constant $C>0$, which depends only on the lifting basis $(\mu,M_b,M_\sigma)$ and the constants $C_{b,\Lip}$ and $C_{\sigma,\Lip}$, such that, for any two solutions $Y$ and $Y'$ of the SEE \eqref{eq_SEE} driven by the same $d$-dimensional Brownian motion $W$, we have
\begin{equation}\label{eq_SEE-apriori}
	\bE\left[\sup_{t\in[0,T]}\|Y_t\|_\cH^2+\int^T_0\|Y_t\|_\cV^2\,\diff t\relmiddle|\cF_0\right]\leq Ce^{CT}\left(1+\|Y_0\|_\cH^2\right)
\end{equation}
and
\begin{equation}\label{eq_SEE-apriori'}
	\bE\left[\sup_{t\in[0,T]}\|Y_t-Y'_t\|_\cH^2+\int^T_0\|Y_t-Y'_t\|_\cV^2\,\diff t\relmiddle|\cF_0\right]\leq Ce^{CT}\|Y_0-Y'_0\|_\cH^2
\end{equation}
for any $T>0$ a.s.
\item[(iii)]
Denote by $(K_b,K_\sigma)$ the liftable pair of kernels generated by $(\mu,M_b,M_\sigma)$. Let $Y$ be a solution of the SEE \eqref{eq_SEE}. Then, the $\bR^n$-valued progressively measurable process $X_t:=\mu[Y_t]$, $t\geq0$, is the solution of the SVE \eqref{eq_SVE} with the forcing term given by $x(t)=\mu[\cS(t)Y_0]=\int_{[0,\infty)}e^{-\theta t}Y_0(\theta)\dmu$ for $t>0$.
\end{itemize}
\end{prop}

%% Proof

\begin{proof}
To establish the well-posedness stated in assertion (i), we apply a general result concerning the well-posedness of monotone SPDEs (see, e.g., \cite[Theorem 4.2.4]{LiRo15}), combined with a standard localization technique. To this end, we verify the following conditions (H1)--(H4) required in \cite[Theorem 4.2.4]{LiRo15} (with $\alpha=2$ in their notation):
\begin{itemize}
\item[(H1)]
(\emph{Hemicontinuity})
For any $y_1,y_2,y_3\in\cV$, the function
\begin{equation*}
	\bR\ni\lambda\mapsto\big\langle \cA(y_1+\lambda y_2)+\cM_bb\big(\mu[y_1+\lambda y_2]\big),y_3\big\rangle_{\cV^*,\cV}\in\bR
\end{equation*}
is continuous.
\item[(H2)]
(\emph{Weak monotonicity})
There exists a constant $C_{\mathrm{H}2}\in\bR$ such that, for any $y_1,y_2\in\cV$, it holds that
\begin{align*}
	&2\big\langle \cA(y_1-y_2)+\cM_b\big(b(\mu[y_1])-b(\mu[y_2])\big),y_1-y_2\big\rangle_{\cV^*,\cV}+\big\|\cM_\sigma\big(\sigma(\mu[y_1])-\sigma(\mu[y_2])\big)\big\|_{L_2(\bR^d;\cH)}^2\\
	&\leq C_{\mathrm{H}2}\|y_1-y_2\|_\cH^2.
\end{align*}
\item[(H3)]
(\emph{Coercivity})
There exist constants $C_{\mathrm{H}3,1}\in\bR$, $C_{\mathrm{H}3,2}>0$ and $C_{\mathrm{H}3,3}>0$ such that, for any $y\in\cV$, it holds that
\begin{equation*}
	2\big\langle \cA y+\cM_bb(\mu[y]),y\big\rangle_{\cV^*,\cV}+\big\|\cM_\sigma\sigma(\mu[y])\big\|_{L_2(\bR^d;\cH)}^2\leq C_{\mathrm{H}3,1}\|y\|_\cH^2-C_{\mathrm{H}3,2}\|y\|_\cV^2+C_{\mathrm{H}3,3}.
\end{equation*}
\item[(H4)]
(\emph{Boundedness})
There exists a constant $C_{\mathrm{H}4}>0$ such that, for any $y\in\cV$, it holds that
\begin{equation*}
	\big\|\cA y+\cM_bb(\mu[y])\big\|_{\cV^*}\leq C_{\mathrm{H}4}\big(1+\|y\|_\cV\big).
\end{equation*}
\end{itemize}
First, since $\cA:\cV\to\cV^*$, $\mu[\cdot]:\cV\to\bR^n$, $b:\bR^n\to\bR^n$ and $\cM_b:\bR^n\to\cV^*$ are (strongly) continuous, we see that condition (H1) (hemicontinuity) holds. Second, observe that, for any $y_1,y_2\in\cV$,
\begin{align*}
	&2\big\langle \cA(y_1-y_2)+\cM_b\big(b(\mu[y_1])-b(\mu[y_2])\big),y_1-y_2\big\rangle_{\cV^*,\cV}+\big\|\cM_\sigma\big(\sigma(\mu[y_1])-\sigma(\mu[y_2])\big)\big\|_{L_2(\bR^d;\cH)}^2\\
	&\leq2\langle\cA(y_1-y_2),y_1-y_2\rangle_{\cV^*,\cV}+2\|\cM_b\|_{L(\bR^n;\cV^*)}C_{b,\Lip}\big|\mu[y_1-y_2]\big|\|y_1-y_2\|_\cV\\
	&\hspace{1cm}+\|\cM_\sigma\|_{L(\bR^n;\cH)}^2C_{\sigma,\Lip}^2\big|\mu[y_1-y_2]\big|^2\\
	&\leq2\|y_1-y_2\|_\cH^2-\frac{3}{2}\|y_1-y_2\|_\cV^2+\big(2\|\cM_b\|_{L(\bR^n;\cV^*)}^2C_{b,\Lip}^2+\|\cM_\sigma\|_{L(\bR^n;\cH)}^2C_{\sigma,\Lip}^2\big)\big|\mu[y_1-y_2]\big|^2,
\end{align*}
where we used \eqref{eq_A-nonpositive'} and Young's inequality in the last line. By using \cref{lemm_mu-ep} with $\ep>0$ given by $\ep=2^{-1}(2\|\cM_b\|_{L(\bR^n;\cV^*)}^2C_{b,\Lip}^2+\|\cM_\sigma\|_{L(\bR^n;\cH)}^2C_{\sigma,\Lip}^2)^{-1}$, we obtain
\begin{equation}\label{eq_SEE-monotonicity}
\begin{split}
	&2\big\langle \cA(y_1-y_2)+\cM_b\big(b(\mu[y_1])-b(\mu[y_2])\big),y_1-y_2\big\rangle_{\cV^*,\cV}+\big\|\cM_\sigma\big(\sigma(\mu[y_1])-\sigma(\mu[y_2])\big)\big\|_{L_2(\bR^d;\cH)}^2\\
	&\leq\widetilde{C}_{\mathrm{H}2}\|y_1-y_2\|_\cH^2-\|y_1-y_2\|_\cV^2,
\end{split}
\end{equation}
where $\widetilde{C}_{\mathrm{H}2}:=2+C_{\mu,\ep}(2\|\cM_b\|_{L(\bR^n;\cV^*)}^2C_{b,\Lip}^2+\|\cM_\sigma\|_{L(\bR^n;\cH)}^2C_{\sigma,\Lip}^2)$ with $C_{\mu,\ep}>0$ arising from \cref{lemm_mu-ep}. In particular, condition (H2) (weak monotonicity) is satisfied. Similarly, noting the linear growth properties of $b:\bR^n\to\bR^n$ and $\sigma:\bR^n\to\bR^{n\times d}$, a calculation analogous to the one above confirms that condition (H3) (coercivity) also holds. Lastly, since $\cA:\cV\to\cV^*$, $\cM_b:\bR^n\to\cV^*$ and $\mu[\cdot]:\cV\to\bR^n$ are bounded linear operators, and $b:\bR^n\to\bR^n$ exhibits linear growth, we conclude that condition (H4) (boundedness) is satisfied.

Based on conditions (H1)--(H4) established above, \cite[Theorem 4.2.4]{LiRo15} implies that, provided the $\cH$-valued, $\cF_0$-measurable initial condition $Y_0$ is square-integrable, the SEE \eqref{eq_SEE} admits a (unique) solution $Y$ satisfying
\begin{equation*}
	\bE\left[\sup_{t\in[0,T]}\|Y_t\|_\cH^2+\int^T_0\|Y_t\|_\cV^2\,\diff t\right]<\infty
\end{equation*}
for any $T>0$. For a general $Y_0$ that is not necessarily square-integrable, we consider the $\cF_0$-measurable partition $\{A_N\}_{N\in\bN}$ of $\Omega$ defined by $A_N:=\{N-1\leq\|Y_0\|_\cH<N\}$. Denoting by $Y^N=(Y^N_t)_{t\geq0}$ the solution to the SEE \eqref{eq_SEE} with the initial condition $Y_0\1_{A_N}$ and setting $Y_t:=\sum^\infty_{N=1}Y^N_t\1_{A_N}$ for $t\geq0$, we see that $Y=(Y_t)_{t\geq0}$ is a solution to the SEE \eqref{eq_SEE} in the sense of \cref{defi_SEE-solution} with the initial condition $Y_0$. Uniqueness of the solution follows from the estimate \eqref{eq_SEE-apriori'} in assertion (ii), which is established below. The fact that the solution to the SEE \eqref{eq_SEE} forms a time-homogeneous Markov process follows, once again, from the general result \cite[Proposition 4.3.5]{LiRo15} on monotone SPDEs.

Next, we prove assertion (ii). Specifically, we focus on the proof of estimate \eqref{eq_SEE-apriori'}, as estimate \eqref{eq_SEE-apriori} can be derived in an analogous manner. Let $Y$ and $Y'$ be two solutions of \eqref{eq_SEE}. Applying It\^{o}'s formula for the squared norm $\|\cdot\|_\cH^2$ (see \cite[Theorem 4.2.5]{LiRo15} or \cref{lemm_Ito} below) to the $\cH$-valued process $Y-Y'$ yields
\begin{align*}
	&\big\|Y_t-Y'_t\big\|_\cH^2\\
	&=\big\|Y_0-Y'_0\big\|_\cH^2+2\int^t_0\Big\langle Y_s-Y'_s,\cM_\sigma\big(\sigma(\mu[Y_s])-\sigma(\mu[Y'_s])\big)\,\diff W_s\Big\rangle_\cH\\
	&\hspace{0.2cm}+\int^t_0\left\{2\Big\langle \cA (Y_s-Y'_s)+\cM_b\big(b(\mu[Y_s])-b(\mu[Y'_s])\big),Y_s-Y'_s\Big\rangle_{\cV^*,\cV}+\Big\|\cM_\sigma\big(\sigma(\mu[Y_s])-\sigma(\mu[Y'_s])\big)\Big\|_{L_2(\bR^d;\cH)}^2\right\}\,\diff s
\end{align*}
for any $t\geq0$ a.s. By the monotonicity estimate \eqref{eq_SEE-monotonicity}, we have
\begin{equation}\label{eq_SEE-apriori-Ito}
\begin{split}
	&\big\|Y_t-Y'_t\big\|_\cH^2+\int^t_0\big\|Y_s-Y'_s\big\|_\cV^2\,\diff s\\
	&\leq\big\|Y_0-Y'_0\big\|_\cH^2+\widetilde{C}_{\mathrm{H}2}\int^t_0\big\|Y_s-Y'_s\big\|_\cH^2\,\diff s+2\int^t_0\Big\langle Y_s-Y'_s,\cM_\sigma\big(\sigma(\mu[Y_s])-\sigma(\mu[Y'_s])\big)\,\diff W_s\Big\rangle_\cH
\end{split}
\end{equation}
for any $t\geq0$ a.s. For each $N\in\bN$, set $\tau_N:=\inf\{t\geq0\,|\,\|Y_t-Y'_t\|_\cH^2>N\,\text{or}\,\int^t_0\|Y_s-Y'_s\|_\cV^2\,\diff s>N\}$. Since $Y$ and $Y'$ are $\cH$-valued continuous adapted processes such that $\int^T_0\|Y_t\|_\cV^2\,\diff t<\infty$ and $\int^T_0\|Y'_t\|_\cV^2\,\diff t<\infty$ a.s.\ for any $T>0$, $\{\tau_N\}_{N\in\bN}$ is an increasing sequence of stopping times such that $\lim_{N\to\infty}\tau_N=\infty$ a.s. Fix $N\in\bN$. Since $\cM_\sigma\in L(\bR^n;\cH)$, $\sigma:\bR^n\to\bR^{n\times d}$ is Lipschitz continuous, and $\mu[\cdot]\in L(\cV;\bR^n)$, by the definition of the stopping time $\tau_N$, we see that the stochastic integral in the right-hand side of \eqref{eq_SEE-apriori-Ito} stopped at $\tau_N$ is a martingale. Hence, taking the conditional expectations $\bE[\cdot|\cF_0]$ on both sides of \eqref{eq_SEE-apriori-Ito}, we obtain
\begin{equation}\label{eq_SEE-apriori-1}
	\bE\left[\int^{t\wedge\tau_N}_0\big\|Y_s-Y'_s\big\|_\cV^2\,\diff s\relmiddle|\cF_0\right]\leq\big\|Y_0-Y'_0\big\|_\cH^2+\widetilde{C}_{\mathrm{H}2}\int^t_0\bE\left[\sup_{r\in[0,s\wedge\tau_N]}\big\|Y_r-Y'_r\big\|_\cH^2\relmiddle|\cF_0\right]\,\diff s
\end{equation}
for any $t\geq0$ a.s. Furthermore, applying the conditional Burkholder--Davis--Gundy inequality\footnote{For any one-dimensional continuous martingale $M=(M_t)_{t\geq0}$ on a filtered probability space $(\Omega,\cF,\bF,\bP)$ and any $A\in\cF_0$, applying the standard Burkholder--Davis--Gundy inequality to the martingale $(M_t\1_A)_{t\geq0}$ yields that $\bE[\sup_{t\in[0,T]}|M_t|\1_A]\leq C_\BDG\bE[\langle M\rangle_T^{1/2}\1_A]$, which implies that $\bE[\sup_{t\in[0,T]}|M_t||\cF_0]\leq C_\BDG\bE[\langle M\rangle_T^{1/2}|\cF_0]$ a.s.\ for any $T>0$.} to \eqref{eq_SEE-apriori-Ito} shows that, for some universal constant $C_\BDG>0$,
\begin{align*}
	&\bE\left[\sup_{s\in[0,t\wedge\tau_N]}\big\|Y_s-Y'_s\big\|_\cH^2\relmiddle|\cF_0\right]\\
	&\leq\big\|Y_0-Y'_0\big\|_\cH^2+\widetilde{C}_{\mathrm{H}2}\int^t_0\bE\left[\sup_{r\in[0,s\wedge\tau_N]}\big\|Y_r-Y'_r\big\|_\cH^2\relmiddle|\cF_0\right]\,\diff s\\
	&\hspace{0.5cm}+2C_\BDG\bE\left[\left(\int^{t\wedge\tau_N}_0\big\|Y_s-Y'_s\big\|_\cH^2\big\|\cM_\sigma\big(\sigma(\mu[Y_s])-\sigma(\mu[Y'_s])\big)\big\|_{L_2(\bR^d;\cH)}^2\,\diff s\right)^{1/2}\relmiddle|\cF_0\right]
\end{align*}
for any $t\geq0$ a.s. Since $\bE[\sup_{s\in[0,t\wedge\tau_N]}\|Y_s-Y'_s\|_\cH^2|\cF_0]<\infty$ a.s.\ by the definition of the stopping time $\tau_N$, the above estimate and Young's inequality yield
\begin{equation}\label{eq_SEE-apriori-2}
\begin{split}
	\frac{1}{2}\bE\left[\sup_{s\in[0,t\wedge\tau_N]}\big\|Y_s-Y'_s\big\|_\cH^2\relmiddle|\cF_0\right]&\leq\big\|Y_0-Y'_0\big\|_\cH^2+\widetilde{C}_{\mathrm{H}2}\int^t_0\bE\left[\sup_{r\in[0,s\wedge\tau_N]}\big\|Y_r-Y'_r\big\|_\cH^2\relmiddle|\cF_0\right]\,\diff s\\
	&\hspace{0.5cm}+2C_\BDG^2\bE\left[\int^{t\wedge\tau_N}_0\big\|\cM_\sigma\big(\sigma(\mu[Y_s])-\sigma(\mu[Y'_s])\big)\big\|_{L_2(\bR^d;\cH)}^2\,\diff s\relmiddle|\cF_0\right]
\end{split}
\end{equation}
for any $t\geq0$ a.s. Concerning the last conditional expectation in the right-hand side of \eqref{eq_SEE-apriori-2}, since $\cM_\sigma\in L(\bR^n;\cH)$, $\sigma:\bR^n\to\bR^{n\times d}$ is Lipschitz continuous, and $\mu[\cdot]\in L(\cV;\bR^n)$, there exists a constant $\widetilde{C}>0$, which depends only on $\mu,M_\sigma$ and the Lipschitz constant $C_{\sigma,\Lip}$ of $\sigma$, such that
\begin{equation}\label{eq_SEE-apriori-3}
	\bE\left[\int^{t\wedge\tau_N}_0\big\|\cM_\sigma\big(\sigma(\mu[Y_s])-\sigma(\mu[Y'_s])\big)\big\|_{L_2(\bR^d;\cH)}^2\,\diff s\relmiddle|\cF_0\right]\leq\widetilde{C}\bE\left[\int^{t\wedge\tau_N}_0\big\|Y_s-Y'_s\big\|_\cV^2\,\diff s\relmiddle|\cF_0\right]
\end{equation}
for any $t\geq0$ a.s. Combining \eqref{eq_SEE-apriori-1}, \eqref{eq_SEE-apriori-2} and \eqref{eq_SEE-apriori-3}, we obtain
\begin{align*}
	&\bE\left[\sup_{s\in[0,t\wedge\tau_N]}\big\|Y_s-Y'_s\big\|_\cH^2+\int^{t\wedge\tau_N}_0\big\|Y_s-Y'_s\big\|_\cV^2\,\diff s\relmiddle|\cF_0\right]\\
	&\leq C\big\|Y_0-Y'_0\big\|_\cH^2+C\int^t_0\bE\left[\sup_{r\in[0,s\wedge\tau_N]}\big\|Y_r-Y'_r\big\|_\cH^2\relmiddle|\cF_0\right]\,\diff s
\end{align*}
for any $t\geq0$ a.s.\ with a constant $C>0$ depending only on the lifting basis $(\mu,M_b,M_\sigma)$ and the constants $C_{b,\Lip}$ and $C_{\sigma,\Lip}$. Thus, Gronwall's inequality implies that
\begin{equation*}
	\bE\left[\sup_{t\in[0,T\wedge\tau_N]}\big\|Y_t-Y'_t\big\|_\cH^2+\int^{T\wedge\tau_N}_0\big\|Y_t-Y'_t\big\|_\cV^2\,\diff t\relmiddle|\cF_0\right]\leq Ce^{CT}\big\|Y_0-Y'_0\big\|_\cH^2
\end{equation*}
for any $T>0$ a.s. Then, taking the limit $N\to\infty$ and using the conditional monotone convergence theorem yield the desired estimate \eqref{eq_SEE-apriori'}. The estimate \eqref{eq_SEE-apriori} can be proved in the same manner as above by using the coercivity (H3) instead of the monotonicity \eqref{eq_SEE-monotonicity}; hence, its proof is omitted.

Now we prove assertion (iii). Let $Y=(Y_t)_{t\geq0}$ be a solution to the SEE \eqref{eq_SEE} in the sense of \cref{defi_SEE-solution}. By virtue of a standard localization technique with respect to the initial condition $Y_0$, without loss of generality, we may assume that $\bE[\|Y_0\|_\cH^2]<\infty$. In this case, by \eqref{eq_SEE-apriori}, we have
\begin{equation}\label{eq_SEE-integrability}
	\bE\left[\sup_{t\in[0,T]}\|Y_t\|_\cH^2+\int^T_0\|Y_t\|_\cV^2\,\diff t\right]<\infty.
\end{equation}
Note that the solution $Y$ to the SEE \eqref{eq_SEE} in the sense of \cref{defi_SEE-solution} can be regarded as an analytically strong solution on $\cV^*$; see \cref{rem_SEE-solution}. Hence, the arguments in \cite[Chapter 5]{DaPrZa14} show that it is also an analytically mild solution on $\cV^*$ in the sense that
\begin{equation}\label{eq_SEE-mild}
	Y_t=\cS(t)Y_0+\int^t_0\cS(t-s)\cM_bb(\mu[Y_s])\,\diff s+\int^t_0\cS(t-s)\cM_\sigma\sigma(\mu[Y_s])\,\diff W_s\ \ \text{in $\cV^*$}
\end{equation}
for any $t\geq0$ a.s.; see also \cite[Proposition 2.9 and Proposition 2.10]{FrKn01} for more details on relationships between analytically strong, weak and mild solutions for SEEs on separable Hilbert spaces. Let $m\in\bN$, and define maps $\mu_m[\cdot]:\cV^*\to\bR^n$ and $\overline{\mu}_m[\cdot]:\cV\to\bR^n$ by
\begin{align*}
	&\mu_m[y]:=\int_{[0,m)}y(\theta)\dmu\ \ \text{for $y\in\cV^*$},\ \ \text{and}\\
	&\overline{\mu}_m[y]:=\mu[y]-\mu_m[y]=\int_{[m,\infty)}y(\theta)\dmu\ \ \text{for $y\in\cV$}.
\end{align*}
Clearly, $\mu_m[\cdot]$ is a bounded linear operator from $\cV^*$ to $\bR^n$, and $\overline{\mu}_m[\cdot]$ is a bounded linear operator from $\cV$ to $\bR^n$. Furthermore, by the Cauchy--Schwarz inequality, we have
\begin{equation*}
	\big\|\overline{\mu}_m[\cdot]\big\|_{L(\cV;\bR^n)}=\sup_{\substack{y\in\cV\\\|y\|_\cV\leq1}}\big|\overline{\mu}_m[y]\big|\leq\sup_{\substack{y\in\cV\\\|y\|_\cV\leq1}}\int_{[m,\infty)}|y(\theta)|\dmu\leq\left(\int_{[m,\infty)}(1+\theta)^{-1/2}\dmu\right)^{1/2}.
\end{equation*}
Noting the integrability condition \eqref{eq_integrability-mu} for the measure $\mu$, the dominated convergence theorem yields
\begin{equation}\label{eq_mu-approximation}
	\lim_{m\to\infty}\big\|\overline{\mu}_m[\cdot]\big\|_{L(\cV;\bR^n)}=0.
\end{equation}
We apply the operator $\mu_m[\cdot]\in L(\cV^*;\bR^n)$ to both sides of \eqref{eq_SEE-mild}. Since any bounded linear operator from $\cV^*$ to $\bR^n$ commutes with Bochner and stochastic integrals on $\cV^*$ (see, e.g., \cite[Proposition 1.6 and Proposition 4.30]{DaPrZa14}), we obtain
\begin{equation}\label{eq_SEE-mild-mu_m}
	\mu_m[Y_t]=\mu_m[\cS(t)Y_0]+\int^t_0\mu_m[\cS(t-s)\cM_b]b(\mu[Y_s])\,\diff s+\int^t_0\mu_m[\cS(t-s)\cM_\sigma]\sigma(\mu[Y_s])\,\diff W_s\ \ \text{in $\bR^n$}
\end{equation}
a.s.\ for any $t\geq0$. We take the limit $m\to\infty$ in each term of \eqref{eq_SEE-mild-mu_m}. Recalling that $x(t)=\mu[\cS(t)Y_0]$, $K_b(t)=\mu[\cS(t)\cM_b]$ and $K_\sigma(t)=\mu[\cS(t)\cM_\sigma]$ for a.e.\ $t\geq0$, to establish assertion (iii), it suffices to show that, for any $T>0$,
\begin{align}
	\label{eq_SEE-Volterra-1}
	&\lim_{m\to\infty}\bE\left[\int^T_0\big|\overline{\mu}_m[Y_t]\big|^2\,\diff t\right]=0,\\
	\label{eq_SEE-Volterra-2}
	&\lim_{m\to\infty}\bE\left[\int^T_0\big|\overline{\mu}_m[\cS(t)Y_0]\big|^2\,\diff t\right]=0,\\
	\label{eq_SEE-Volterra-3}
	&\lim_{m\to\infty}\bE\left[\int^T_0\left|\int^t_0\overline{\mu}_m[\cS(t-s)\cM_b]b(\mu[Y_s])\,\diff s\right|^2\,\diff t\right]=0,\ \ \text{and}\\
	\label{eq_SEE-Volterra-4}
	&\lim_{m\to\infty}\bE\left[\int^T_0\left|\int^t_0\overline{\mu}_m[\cS(t-s)\cM_\sigma]\sigma(\mu[Y_s])\,\diff W_s\right|^2\,\diff t\right]=0.
\end{align}
First, \eqref{eq_SEE-Volterra-1} follows from \eqref{eq_SEE-integrability} and \eqref{eq_mu-approximation}. Similarly, since the map $\cS(\cdot):y\mapsto(\cS(t)y)_{t\in(0,T]}$ is a bounded linear operator from $\cH$ to $L^2(0,T;\cV)$, again by \eqref{eq_mu-approximation} we obtain \eqref{eq_SEE-Volterra-2}. As for the expectation in \eqref{eq_SEE-Volterra-3}, by Young's convolution inequality, we have
\begin{equation*}
	\bE\left[\int^T_0\left|\int^t_0\overline{\mu}_m[\cS(t-s)\cM_b]b(\mu[Y_s])\,\diff s\right|^2\,\diff t\right]\leq\left(\int^T_0\big|\overline{\mu}_m[\cS(t)\cM_b]\big|_\op\,\diff t\right)^2\,\bE\left[\int^T_0\big|b(\mu[Y_t])\big|^2\,\diff t\right].
\end{equation*}
Since $b:\bR^n\to\bR^n$ exhibits linear growth and $\mu[\cdot]\in L(\cV;\bR^n)$, it follows from \eqref{eq_SEE-integrability} that $\bE[\int^T_0|b(\mu[Y_t])|^2\,\diff t]<\infty$. Furthermore, observe that
\begin{align*}
	\int^T_0\big|\overline{\mu}_m[\cS(t)\cM_b]\big|_\op\,\diff t&\leq\int^T_0\int_{[m,\infty)}e^{-\theta t}|M_b(\theta)|_\op\dmu\,\diff t\\
	&\leq\int_{[m,\infty)}\theta^{-1}|M_b(\theta)|_\op\dmu\\
	&\leq\left(\int_{[m,\infty)}\theta^{-1/2}\dmu\right)^{1/2}\left(\int_{[m,\infty)}\theta^{-3/2}|M_b(\theta)|_\op^2\dmu\right)^{1/2},
\end{align*}
where we used Tonelli's theorem and the elementary inequality $\int^T_0e^{-\theta t}\,\diff t\leq\theta^{-1}$ in the second inequality, and the Cauchy--Schwarz inequality in the third inequality. Recalling the integrability conditions \eqref{eq_integrability-mu} and \eqref{eq_integrability-M_b} on $\mu$ and $M_b$, the dominated convergence theorem ensures that the last term in the above estimate converges to zero as $m\to\infty$. Hence, we obtain \eqref{eq_SEE-Volterra-3}. It remains to prove \eqref{eq_SEE-Volterra-4}. By It\^{o}'s isometry and Young's convolution inequality, we obtain
\begin{align*}
	\bE\left[\int^T_0\left|\int^t_0\overline{\mu}_m[\cS(t-s)\cM_\sigma]\sigma(\mu[Y_s])\,\diff W_s\right|^2\,\diff t\right]&=\bE\left[\int^T_0\int^t_0\big|\overline{\mu}_m[\cS(t-s)\cM_\sigma]\sigma(\mu[Y_s])\big|^2\,\diff s\diff t\right]\\
	&\leq\int^T_0\big|\overline{\mu}_m[\cS(t)\cM_\sigma]\big|_\op^2\,\diff t\,\bE\left[\int^T_0\big|\sigma(\mu[Y_t])\big|^2\,\diff t\right]
\end{align*}
for any $m\in\bN$. Since $\sigma:\bR^n\to\bR^{n\times d}$ exhibits linear growth and $\mu[\cdot]\in L(\cV;\bR^n)$, it follows from \eqref{eq_SEE-integrability} that $\bE[\int^T_0|\sigma(\mu[Y_t])|^2\,\diff t]<\infty$. Together with the facts that $\cM_\sigma\in L(\bR^n;\cH)$ and $\cS(\cdot)\in L(\cH;L^2(0,T;\cV))$, along with the convergence given in \eqref{eq_mu-approximation}, this implies that \eqref{eq_SEE-Volterra-4} holds. Consequently, by taking the limit $m\to\infty$ in \eqref{eq_SEE-mild-mu_m}, we conclude that assertion (iii) is satisfied. This completes the proof.
\end{proof}

Define
\begin{equation*}
	\Lambda:=\left\{\eta:[0,\infty)\to\cH\relmiddle|\text{$t\mapsto\eta_t$ is strongly continuous in $\cH$ and}\ \int^T_0\|\eta_t\|_\cV^2\,\diff t<\infty\ \text{for any $T>0$}\right\}
\end{equation*}
and
\begin{equation}\label{eq_Lambda-metric}
	d_\Lambda(\eta,\eta'):=\sum_{T\in\bN}\frac{1}{2^T}\big\{\|\eta-\eta'\|_{\Lambda_T}\wedge1\big\}\ \ \text{for $\eta,\eta'\in\Lambda$},
\end{equation}
where
\begin{equation}\label{eq_Lambda-norm}
	\|\eta\|_{\Lambda_T}:=\left(\sup_{t\in[0,T]}\|\eta_t\|_\cH^2+\int^T_0\|\eta_t\|_\cV^2\,\diff t\right)^{1/2}\ \ \text{for $\eta\in\Lambda$}.
\end{equation}
Then, $(\Lambda,d_\Lambda)$ is a complete separable metric space; see \cref{lemm_app-Lambda}. A sequence $\{\eta^k\}_{k\in\bN}$ in $\Lambda$ converges to an element $\eta\in\Lambda$ with respect to the metric $d_\Lambda$ if and only if $\lim_{k\to\infty}\|\eta_k-\eta\|_{\Lambda_T}=0$ for any $T>0$. By virtue of the final assertion in \cref{lemm_app-Lambda} (see also \cref{rem_app-Lambda} (ii)), for each solution $Y$ to the SEE \eqref{eq_SEE} on a complete probability space $(\Omega,\cF,\bP)$, the sample-path map $\Omega\ni\omega\mapsto(Y_t(\omega))_{t\geq0}$ defines a $\Lambda$-valued Borel measurable random variable. Under the setting in \cref{prop_SEE-well-posed}, the law of the path of a solution $Y$ on $\Lambda$ is determined by the distribution of the initial condition $Y_0$ on $\cH$. This follows from the uniqueness in law for the SEE \eqref{eq_SEE}, which is ensured by the pathwise uniqueness and a Yamada--Watanabe-type result for general SEEs (see, e.g., \cite[Appendix E]{LiRo15}). When $Y_0$ is distributed according to a prescribed initial distribution $\nu\in\cP(\cH)$, we denote the law of the solution $Y$ on the path space $\Lambda$ by $\bfP^\nu\in\cP(\Lambda)$. The Markov semigroup $\{P_t\}_{t\geq0}$ on $\cH$ associated with the SEE \eqref{eq_SEE} is then given by
\begin{equation*}
	P_t(y,A):=\bfP^{\delta_y}\big(\left\{\eta\in\Lambda\relmiddle|\eta_t\in A\right\}\big),\ \ y\in\cH,\ A\in\cB(\cH),\ t\geq0.
\end{equation*}
The estimate \eqref{eq_SEE-apriori'} in particular implies that the Markov semigroup $\{P_t\}_{t\geq0}$ on $\cH$ satisfies the Feller property. Furthermore, the Markov semigroup $\{P_t\}_{t\geq0}$ is clearly stochastically continuous, and hence it is measurable.

%% Remark

\begin{rem}\label{rem_path-distribution}
\begin{itemize}
\item[(i)]
More generally than the Feller property of $\{P_t\}_{t\geq0}$, the estimate \eqref{eq_SEE-apriori'} implies that the map $\nu\mapsto\bfP^\nu$ is continuous from $\cP(\cH)$ to $\cP(\Lambda)$, both of which are equipped with their respective topology of weak convergence of probability measures. For each $\tau\in[0,\infty)$, denote by $\sS_\tau:\Lambda\to\Lambda$ the time-shift operator on $\Lambda$, which is defined by $\sS_\tau\eta:=(\eta_{t+\tau})_{t\geq0}$ for each $\eta\in\Lambda$. The time-homogeneous Markov property implies that $\bfP^\nu\circ\sS_\tau^{-1}=\bfP^{P_\tau^*\nu}$.
\item[(ii)]
Consider the SVE \eqref{eq_SVE} with the liftable pair of kernels $(K_b,K_\sigma)$ generated by a lifting basis $(\mu,M_b,M_\sigma)$. Given a free term $x$ of the form $x(t)=\mu[\cS(t)Y_0]$ for some $\cH$-valued random variable $Y_0$ distributed according to a prescribed probability measure $\nu\in\cP(\cH)$, we denote the law of the $\diff t$-equivalence class of the solution $X=(X_t)_{t>0}$ on $L^2_\loc(0,\infty;\bR^n)$ by $\bfQ^\nu\in\cP(L^2_\loc(0,\infty;\bR^n))$. By \cref{prop_SEE-well-posed} (iii), we have $\bfQ^\nu=\bfP^\nu\circ\mu[\cdot]^{-1}$. Here, with a slight abuse of notation, $\mu[\cdot]$ is understood as a map from $\Lambda$ to $L^2_\loc(0,\infty;\bR^n)$, which maps each $\eta\in\Lambda$ to the $\diff t$-equivalence class of $(\mu[\eta_t])_{t\geq0}$. Since $\cV\ni y\mapsto\mu[y]\in\bR^n$ is a bounded linear operator, by the definition of $\Lambda$, the map $\mu[\cdot]:\Lambda\to L^2_\loc(0,\infty;\bR^n)$ is continuous. Consequently, the continuity of $\cP(\cH)\ni\nu\mapsto\bfP^\nu\in\cP(\Lambda)$ implies the continuity of $\cP(\cH)\ni\nu\mapsto\bfQ^\nu\in\cP(L^2_\loc(0,\infty;\bR^n))$ with respect to the topology of weak convergence of probability measures. With a slight abuse notation, we again denote the time-shift operator $\xi\mapsto\xi_{\cdot+\tau}$ on $L^2_\loc(0,\infty;\bR^n)$ by $\sS_\tau$ for each $\tau\in[0,\infty)$. Then, we have $\bfQ^\nu\circ\sS_\tau^{-1}=\bfQ^{P_\tau^*\nu}$.
\end{itemize}
\end{rem}

To apply the generalized Harris' theorem of \cite{HaMaSc11} (see \cref{theo_HaMaSc11}), the construction of a suitable distance(-like) function is essential. To this end, a ``change-of-norm'' technique plays a crucial role in our framework. More precisely, we introduce the following definition.

%% Definition

\begin{defi}\label{defi_Phi}
Let $(\mu,M_b,M_\sigma)$ be a lifting basis. We call a measurable map $\Phi:[0,\infty)\to\bR^{n\times n}$ an \emph{admissible weight function} if the matrix $\Phi(\theta)\in\bR^{n\times n}$ is symmetric, positive definite and satisfies
\begin{equation}\label{eq_weight-bound}
	|\Phi(\theta)|_\op\leq C_\Phi(1+\theta)^{-1/2}\ \ \text{and}\ \ \big|\Phi(\theta)^{-1}\big|_\op\leq C_\Phi(1+\theta)^{1/2}
\end{equation}
for $\mu$-a.e.\ $\theta\in[0,\infty)$ for some constant $C_\Phi>0$. For each admissible weight function $\Phi$, define
\begin{align*}
	&\|y\|_\Phi:=\left(\int_{[0,\infty)}\langle\Phi(\theta)y(\theta),y(\theta)\rangle\dmu\right)^{1/2}\ \ \text{for $y\in\cH$},\\
	&\vnorm{y}_\Phi:=\left(\int_{[0,\infty)}\theta\langle\Phi(\theta)y(\theta),y(\theta)\rangle\dmu\right)^{1/2}\ \ \text{for $y\in\cV$},\\
	&\mu_{b,\Phi}[y]:=\int_{[0,\infty)}M_b(\theta)^\top\Phi(\theta)y(\theta)\dmu\ \ \text{for $y\in\cV$},\\
	&\mu_{\sigma,\Phi}[y]:=\int_{[0,\infty)}M_\sigma(\theta)^\top\Phi(\theta)y(\theta)\dmu\ \ \text{for $y\in\cH$},\ \ \text{and}\\
	&Q_{\sigma,\Phi}:=\int_{[0,\infty)}M_\sigma(\theta)^\top\Phi(\theta)M_\sigma(\theta)\dmu.
\end{align*}
Furthermore, define a map $d_\Phi:\cH\times\cH\to[0,1]$ by
\begin{equation*}
	d_\Phi(y_1,y_2):=\|y_1-y_2\|_\Phi\wedge1\ \ \text{for $y_1,y_2\in\cH$.}
\end{equation*}
\end{defi}

%% Remark

\begin{rem}\label{rem_Phi}
\begin{itemize}
\item[(i)]
Notice that $\|y\|_\cH=\|y\|_{\Phi_0}$ and $\|y\|_\cV=(\|y\|_{\Phi_0}^2+\vnorm{y}_{\Phi_0}^2)^{1/2}$ for the admissible weight function $\Phi_0(\theta)=(1+\theta)^{-1/2}I_{n\times n}$.
\item[(ii)]
For every admissible weight function $\Phi$, the map $y\mapsto\|y\|_\Phi$ defines an equivalent norm in $\cH$. This norm is induced by the inner product
\begin{equation*}
	\langle y_1,y_2\rangle_\Phi:=\int_{[0,\infty)}\langle\Phi(\theta)y_1(\theta),y_2(\theta)\rangle\dmu,\ \ y_1,y_2\in\cH.
\end{equation*}
The map $d_\Phi:\cH\times\cH\to[0,1]$ is a distance function on $\cH$ which is compatible to the original norm-topology of $\cH$. Similarly, the map $y\mapsto(\|y\|_\Phi^2+\vnorm{y}_\Phi^2)^{1/2}$ defines an equivalent norm in $\cV$. If we identify $\cH^*$ with $\cH$ by the Riesz isomorphism $y\mapsto\langle y,\cdot\rangle_\Phi$ with respect to the inner product $\langle\cdot,\cdot\rangle_\Phi$, which may be different from the isomorphism $y\mapsto\langle y,\cdot\rangle_\cH$, then the corresponding duality pairing $\langle\cdot,\cdot\rangle_{\cV^*,\cV,\Phi}$ between $\cV^*$ and $\cV$ is given as the unique extension of the bounded bilinear functional $\cH\times\cV\ni(y_1,y_2)\mapsto\langle y_1,y_2\rangle_\Phi\in\bR$ to $\cV^*\times\cV$.
\item[(iii)]
For each admissible weight function $\Phi$, it is clear from the definition that
\begin{equation}\label{eq_A-Phi}
	\langle\cA y,y\rangle_{\cV^*,\cV,\Phi}=-\vnorm{y}_\Phi^2
\end{equation}
for any $y\in\cV$. The integrability condition \eqref{eq_integrability-M_b} on $M_b$ and the bound \eqref{eq_weight-bound} for $\Phi$ ensure that the map $\mu_{b,\Phi}[\cdot]$ is a bounded linear operator from $\cV$ to $\bR^n$. This is the adjoint operator of $\cM_b\in L(\bR^n;\cV^*)$ with respect to the duality pairing $\langle\cdot,\cdot\rangle_{\cV^*,\cV,\Phi}$, that is,
\begin{equation}\label{eq_mu-b-Phi}
	\langle\cM_b x,y\rangle_{\cV^*,\cV,\Phi}=\int_{[0,\infty)}\big\langle \Phi(\theta)M_b(\theta)x,y(\theta)\big\rangle\dmu=\langle x,\mu_{b,\Phi}[y]\rangle
\end{equation}
for any $x\in\bR^n$ and $y\in\cV$. Similarly, the integrability condition \eqref{eq_integrability-M_sigma} on $M_\sigma$ and the bound \eqref{eq_weight-bound} for $\Phi$ ensure that the map $\mu_{\sigma,\Phi}[\cdot]$ is a bounded linear operator from $\cH$ to $\bR^n$. This is the adjoint operator of $\cM_\sigma\in L(\bR^n;\cH)$ with respect to the inner product $\langle\cdot,\cdot\rangle_\Phi$, that is,
\begin{equation}\label{eq_mu-sigma-Phi}
	\langle\cM_\sigma x,y\rangle_\Phi=\int_{[0,\infty)}\big\langle\Phi(\theta)M_\sigma(\theta)x,y(\theta)\big\rangle\dmu=\langle x,\mu_{\sigma,\Phi}[y]\rangle
\end{equation}
for any $x\in\bR^n$ and $y\in\cH$. Furthermore, again by \eqref{eq_integrability-M_sigma} and \eqref{eq_weight-bound}, the matrix $Q_{\sigma,\Phi}\in\bR^{n\times n}$ is well-defined. This matrix is symmetric and satisfies
\begin{equation}\label{eq_Q-Phi}
	\big\|\cM_\sigma x\big\|_{L_2(\bR^d;(\cH,\langle\cdot,\cdot\rangle_\Phi))}^2=\tr\big[x^\top Q_{\sigma,\Phi} x\big]
\end{equation}
for any $x\in\bR^{n\times d}$. Here, $\|\cdot\|_{L_2(\bR^d;(\cH,\langle\cdot,\cdot\rangle_\Phi))}$ denotes the Hilbert--Schmidt norm for an operator from the Euclidean space $\bR^d$ to the separable Hilbert space $\cH$ equipped with the inner product $\langle\cdot,\cdot\rangle_\Phi$.
\end{itemize}
\end{rem}

Applying \cite[Theorem 4.2.5]{LiRo15} (with $\alpha=2$ in their notation)\footnote{Although \cite[Theorem 4.2.5]{LiRo15} assumes some additional integrability conditions with respect to $\bP$, such conditions can be dropped by a standard localization argument. See also \cite[Lemma 2.14]{Pa21} for a similar result (with simpler proof than \cite[Theorem 4.2.5]{LiRo15}) without any integrability assumptions with respect to $\bP$.} to our setting, we obtain the following It\^{o}'s formula for the squared norm $\|\cdot\|_\Phi^2$.

%% Lemma

\begin{lemm}\label{lemm_Ito}
Suppose that we are given a $d$-dimensional Brownian motion $W$ defined on a complete probability space $(\Omega,\cF,\bP)$, with respect to a filtration $\bF=(\cF_t)_{t\geq0}$ satisfying the usual conditions. Let $(\mu,M_b,M_\sigma)$ be a lifting basis. Let $Y$ be an $\cH$-valued continuous adapted process, and let $B$ and $\Sigma$ be $\cV^*$ and $L_2(\bR^d;\cH)$-valued progressively measurable processes, respectively. Assume that $\int^T_0\{\|Y_t\|_\cV^2+\|B_t\|_{\cV^*}^2+\|\Sigma_t\|_{L_2(\bR^d;\cH)}^2\}\,\diff t<\infty$ a.s.\ for any $T>0$, and that $Y$ satisfies
\begin{equation*}
	Y_t=Y_0+\int^t_0B_s\,\diff s+\int^t_0\Sigma_s\,\diff W_s
\end{equation*}
in $\cV^*$ for all $t\geq0$ a.s. Then, for any admissible weight function $\Phi$, it holds that
\begin{equation}\label{eq_Ito}
	\|Y_t\|_\Phi^2=\|Y_0\|_\Phi^2+\int^t_0\Big\{2\langle B_s,Y_s\rangle_{\cV^*,\cV,\Phi}+\|\Sigma_s\|_{L_2(\bR^d;(\cH,\langle\cdot,\cdot\rangle_\Phi))}^2\Big\}\,\diff s+2\int^t_0\langle Y_s,\Sigma_s\,\diff W_s\rangle_\Phi
\end{equation}
for all $t\geq0$ a.s.
\end{lemm}

%% Remark

\begin{rem}
As discussed previously, the pairing $\langle B_s,Y_s\rangle_{\cV^*,\cV,\Phi}$ in \eqref{eq_Ito} is understood in the sense of $\langle B_s,\widetilde{Y}_s\rangle_{\cV^*,\cV,\Phi}$, where $\widetilde{Y}:=Y\1_\cV(Y)$ is a $\cV$-valued progressively measurable version of $Y$; see \cref{rem_SEE-solution}. It should be emphasized that the process $Y$ is generally not a semimartingale on the Hilbert space $\cH$; indeed, it can only be viewed as a semimartingale when considered as a process in the larger space $\cV^*$. Nevertheless, the above It\^{o}'s formula remains valid for the squared (equivalent) norm $\|\cdot\|_\Phi^2$ on $\cH$, and the process $(\|Y_t\|_\Phi^2)_{t\geq0}$ is a nonnegative semimartingale.
\end{rem}

\section{Main results}\label{sec_main}

We now state our main results. The proofs of the key assertions (\cref{theo_coupling}, \cref{theo_Lyapunov} and \cref{prop_Lyapunov-sufficient}) in this section are deferred to \cref{sec_proof}, while the primary theorem on ergodicity (\cref{theo_main}) follows as a straightforward consequence of these results by virtue of the generalized Harris' theorem recalled in \cref{theo_HaMaSc11}.

For a given lifting basis $(\mu,M_b,M_\sigma)$ and measurable maps $b:\bR^n\to\bR^n$ and $\sigma:\bR^n\to\bR^{n\times d}$, we introduce the following assumption.

%% Assumption

\begin{assum}\label{assum_coupling}
\begin{itemize}
\item
$\kappa:=\inf\supp>0$.
\item
For $\mu$-a.e.\ $\theta\in[0,\infty)$, the matrix $M_\sigma(\theta)\in\bR^{n\times n}$ is symmetric and positive definite.
\item
There exist constants $C_{b,\Lip},C_{\sigma,\Lip}>0$ such that $|b(x)-b(x')|\leq C_{b,\Lip}|x-x'|$ and $|\sigma(x)-\sigma(x')|\leq C_{\sigma,\Lip}|x-x'|$ for all $x,x'\in\bR^n$.
\item
There exists a constant $C_\UE>0$ such that $\langle\sigma(x)\sigma(x)^\top\xi,\xi\rangle\geq \frac{|\xi|^2}{C_\UE}$ for all $\xi\in\bR^n$ and all $x\in\bR^n$.
\end{itemize}
\end{assum}

%% Remark

\begin{rem}\label{rem_exam-liftable}
\begin{itemize}
\item[(i)]
In \cref{assum_coupling}, we do not impose any restrictions on the values of the Lipschitz constants $C_{b,\Lip}$ and $C_{\sigma,\Lip}$ of $b:\bR^n\to\bR^n$ and $\sigma:\bR^n\to\bR^{n\times d}$. This stands in contrast to \cite{BeDeKr22,BiBoCaFr25,BiBoFr25}, where various results on long-time asymptotics (without a spectral gap) for SVEs and their Markovian lifts are established under the assumption that the Lipschitz and growth constants are sufficiently small to make the systems dissipative. Instead, we require the uniform ellipticity of the diffusion coefficient $\sigma$. It is worth noting that this does not imply that the diffusion of the SEE \eqref{eq_SEE} is uniformly elliptic on the entire state space $\cH$. On the contrary, the system is \emph{highly degenerate}, as only finite-dimensional noise $W$ acts on the infinite-dimensional state space. Notably, the noise dimension $d$ can be as small as one, provided that the dimension $n$ of the original SVE \eqref{eq_SVE} is one.
\item[(ii)]
For the lifting basis in \cref{exam_liftable} (i) generating the sum-of-exponentials type kernels $(K^{\exp}_b,K^{\exp}_\sigma)$, the first and second conditions in \cref{assum_coupling} are satisfied provided that $\kappa_i>0$ and $M_{\sigma,i}$ is symmetric and positive definite for each $i\in\{1,\dots,N\}$. In this case, we have $\kappa:=\inf\supp=\min\{\kappa_1,\dots,\kappa_N\}$. As for the lifting basis in \cref{exam_liftable} (ii) generating the tempered fractional kernels $(K^\fractional_b,K^\fractional_\sigma)$, the first and second conditions in \cref{assum_coupling} are satisfied provided that $0<\kappa_\sigma\leq\kappa_b$. In this case, we have $\kappa:=\inf\supp=\kappa_\sigma$. It should be noted that the requirement $\kappa:=\inf\supp>0$ necessitates exponential tempering of the kernels.
\end{itemize}
\end{rem}

The following theorem constitutes the key ingredient in applying the generalized Harris' theorem of \cite{HaMaSc11} (see also \cref{theo_HaMaSc11}). Furthermore, we establish an \emph{asymptotic log-Harnack inequality} which is of independent interest. In particular, it guarantees the uniqueness of the invariant probability measure (see, e.g., \cite[Theorem 2.1]{BaWaYu19}). In what follows, we use the notation introduced in \cref{defi_Phi} and \eqref{eq_defi-W_d}.

%% Theorem

\begin{theo}\label{theo_coupling}
Suppose that a lifting basis $(\mu,M_b,M_\sigma)$ and measurable maps $b:\bR^n\to\bR^n$, $\sigma:\bR^n\to\bR^{n\times d}$ satisfy \cref{assum_coupling}. Let $\{P_t\}_{t\geq0}$ be the Markov semigroup associated with the SEE \eqref{eq_SEE}. Then, there exists an admissible weight function $\Phi:[0,\infty)\to\bR^{n\times n}$, depending only on $(\mu,M_b,M_\sigma)$ and the constants $C_{b,\Lip}, C_{\sigma,\Lip}$, and $C_\UE$, such that the following assertions hold:
\begin{itemize}
\item[(i)]
For any $t\geq\frac{4\log2}{\kappa}$, the distance function $d_\Phi$ is contracting for $P_t$. Specifically, for any $y_1,y_2\in\cH$ with $d_\Phi(y_1,y_2)<1$, we have
\begin{equation}\label{eq_contracting}
	\bW_{d_\Phi}(P_t(y_1,\cdot),P_t(y_2,\cdot))\leq\frac{3}{4}d_\Phi(y_1,y_2).
\end{equation}
\item[(ii)]
For any $R>0$ and any $t\geq\frac{4R^2+2\log(8R)}{\kappa}$, the set $\overline{B}_{\Phi}(R):=\{y\in\cH\,|\,\|y\|_\Phi\leq R\}$ is $d_\Phi$-small for $P_t$. Specifically, for any $y_1,y_2\in \overline{B}_{\Phi}(R)$, we have
\begin{equation}\label{eq_small}
	\bW_{d_\Phi}(P_t(y_1,\cdot),P_t(y_2,\cdot))\leq1-\frac{1}{4}\exp(-2R^2).
\end{equation}
\item[(iii)]
The following asymptotic log-Harnack inequality holds: for any $y_1,y_2\in\cH$, any $t\geq0$ and any bounded Borel measurable function $f:\cH\to[1,\infty)$ such that
\begin{equation*}
	\|\nabla_\Phi\log f\|_\infty:=\sup_{\substack{y,y'\in\cH\\y\neq y'}}\frac{|\log f(y)-\log f(y')|}{\|y-y'\|_\Phi}<\infty,
\end{equation*}
it holds that
\begin{equation}\label{eq_Harnack}
	P_t\log f(y_1)\leq\log P_tf(y_2)+\frac{1}{2}\|y_1-y_2\|_\Phi^2+e^{-\kappa t/2}\|y_1-y_2\|_\Phi\|\nabla_\Phi\log f\|_\infty.
\end{equation}
In particular, the Markov semigroup $\{P_t\}_{t\geq0}$ possesses at most one invariant probability measure.
\end{itemize}
\end{theo}

The proof of the above theorem is provided in \cref{subsec_proof-coupling}.

%% Remark

\begin{rem}
In addition to the uniqueness of the invariant probability measure, the asymptotic log-Harnack inequality in assertion (iii) implies that the Markov semigroup $\{P_t\}_{t\geq0}$ is \emph{asymptotically strong Feller} and \emph{asymptotically irreducible}. Furthermore, it provides a \emph{gradient estimate} and an \emph{asymptotic heat kernel estimate}; for further details, see \cite[Theorem 2.1]{BaWaYu19}. In our previous work \cite[Section 4]{Ha24}, the asymptotic log-Harnack inequality was established for the Markovian lift of an SVE with a scalar completely monotone kernel, corresponding to the case $M_b=M_\sigma=I_{n\times n}$. The present assertion (iii) significantly extends this result to the case of matrix-valued kernels where $K_b$ and $K_\sigma$ may differ, allowing for cases where $K_b$ is not necessarily square-integrable.
\end{rem}

The construction of a Lyapunov function constitutes another important step in applying the generalized Harris' theorem of \cite{HaMaSc11} (see also \cref{theo_HaMaSc11}). The following assumption provides an abstract criterion for the existence of a Lyapunov function. For its statement, we again refer to the notation introduced in \cref{defi_Phi}.

%% Assumption

\begin{assum}\label{assum_Lyapunov-abstract}
There exist an admissible weight function $\Psi:[0,\infty)\to\bR^{n\times n}$ and constants $\delta\in(0,1)$, $\rho>0$ and $C_\Lyap>0$ such that
\begin{equation}\label{eq_coercivity}
	\big\langle b(\mu[y]),\mu_{b,\Psi}[y]\big\rangle+\frac{1}{2}\tr\big[\sigma(\mu[y])^\top Q_{\sigma,\Psi}\sigma(\mu[y])\big]\leq\delta\vnorm{y}_\Psi^2-\rho\|y\|_{\Psi}^2+C_\Lyap
\end{equation}
for any $y\in\cV$.
\end{assum}

%% Theorem

\begin{theo}\label{theo_Lyapunov}
Suppose that a lifting basis $(\mu,M_b,M_\sigma)$ and measurable maps $b:\bR^n\to\bR^n$, $\sigma:\bR^n\to\bR^{n\times d}$ satisfy \cref{assum_Lyapunov-abstract}. Assume that the SEE \eqref{eq_SEE} is weakly well-posed, and denote by $\{P_t\}_{t\geq0}$ the associated measurable Markov semigroup on $\cH$. Then, the function $V(y):=\|y\|_{\Psi}^2$ is a Lyapunov function for $\{P_t\}_{t\geq0}$. Furthermore, if $\pi\in\cP(\cH)$ is an invariant probability measure for $\{P_t\}_{t\geq0}$, then the following estimate holds:
\begin{equation}\label{eq_Lyapunov-square-pi}
	\int_{\cH}\Big\{\rho\|y\|_{\Psi}^2+(1-\delta)\vnorm{y}_\Psi^2\Big\}\,\pi(\diff y)\leq C_\Lyap.
\end{equation}
In particular, we have $\pi(\cV)=1$ and $\int_\cV\|y\|_\cV^2\,\pi(\diff y)<\infty$.
%Assume in addition that there exists a constant $C_{\sigma,\Bdd}>0$ such that $|\sigma(x)|_\op\leq C_{\sigma,\Bdd}$ for any $x\in\bR^n$. Then, for any $\delta>0$ satisfying
%\begin{equation}\label{eq_Lyapunov-rho_delta}
%	\rho_\delta:=\frac{\rho}{\delta}-C_{\sigma,\Bdd}^2\int_{[0,\infty)}|\Psi(\theta)^{1/2}M_\sigma(\theta)|_\op^2\dmu>0,
%\end{equation}
%the function $y\mapsto\exp(\delta\|y\|_\Psi^2)$ is also a Lyapunov function for $\{P_t\}_{t\geq0}$. Furthermore, if $\pi\in\cP(\cH)$ is an invariant probability measure for $\{P_t\}_{t\geq0}$, then we have
%\begin{equation}\label{eq_Lyapunov-exp-pi}
%	\int_{\cV}\exp\big(\delta\|y\|_\Psi^2\big)\Big\{\rho_\delta+(1-\delta)\|y\|_{\star,\Psi}^2\Big\}\,\pi(\diff y)\leq\rho_\delta\exp\left(\frac{C_\Lyap}{\rho_\delta}\right).
%\end{equation}
\end{theo}

The proof of the above theorem is provided in \cref{subsec_proof-Lyapunov}.

%% Remark

\begin{rem}\label{rem_Lyapunov-compact}
The condition in \cref{assum_Lyapunov-abstract} is related to the Lyapunov-type condition in \cite[Chapter 7]{GaMa10}, where the ultimate boundedness for a general class of monotone SPDEs on an abstract Gelfand triplet $\cV\hookrightarrow\cH\hookrightarrow\cV^*$ is investigated. In that work, it is shown that the ultimate boundedness condition implies the existence of an invariant probability measure, provided that the embedding $\cV\hookrightarrow\cH$ is compact. However, this classical result cannot be directly applied to our framework, since the embedding $\cV\hookrightarrow\cH$ is not compact in typical cases of interest; see \cref{lemm_app-compact} for a characterization of the compactness of this embedding in our framework. This lack of compactness poses a technical challenge beyond the standard result on the general framework, as \cref{assum_Lyapunov-abstract} alone is not sufficient to conclude the existence of an invariant probability measure; see \cref{rem_Lyapunov} for more detailed discussions on this difficulty. To ensure the existence (as well as the uniqueness and the exponential ergodicity) of the invariant probability measure, it is essential to combine the Lyapunov approach in \cref{theo_Lyapunov} with the contractivity and $d_\Phi$-smallness conditions established in \cref{theo_coupling} (i) and (ii) for a specific distance $d_\Phi$, as will be demonstrated in \cref{theo_main} below.
\end{rem}

It is straightforward to show that \cref{assum_Lyapunov-abstract} holds if $\kappa:=\inf\supp>0$ and the maps $b:\bR^n\to\bR^n$ and $\sigma:\bR^n\to\bR^{n\times d}$ exhibit sub-linear growth; in this case, the required condition is satisfied for \emph{any} admissible weight function $\Psi$. The following proposition, the proof of which is also deferred to \cref{subsec_proof-Lyapunov}, provides another verifiable yet non-trivial sufficient condition for \cref{assum_Lyapunov-abstract}. In this context, the flexibility in the choice of the weight function $\Psi$ turns out to be crucial.

%% Proposition

\begin{prop}\label{prop_Lyapunov-sufficient}
For a lifting basis $(\mu,M_b,M_\sigma)$ and measurable maps $b:\bR^n\to\bR^n$ and $\sigma:\bR^n\to\bR^{n\times d}$, \cref{assum_Lyapunov-abstract} is satisfied if the following conditions hold:
\begin{itemize}
\item
$\kappa:=\inf\supp>0$.
\item
For $\mu$-a.e.\ $\theta\in[0,\infty)$, the matrix $M_b(\theta)\in\bR^{n\times n}$ is symmetric and nonnegative definite.
\item
There exist constants $C_{b,\LG},C_{b,\LG}'>0$ and a constant $\gamma>0$ satisfying $\gamma\int_{[\kappa,\infty)}\theta^{-1}|M_b(\theta)|_\op\dmu<1$ such that, for any $x \in \bR^n$,
\begin{equation*}
	|b(x)|\leq C_{b,\LG}(1+|x|)\ \ \text{and}\ \ \langle b(x),x\rangle\leq\gamma|x|^2+C_{b,\LG}'.
\end{equation*}
\item
There exist constants $p\in(0,1)$ and $C_{\sigma,\subLG}>0$ such that $|\sigma(x)|\leq C_{\sigma,\subLG}(1+|x|^p)$ for all $x\in\bR^n$.
\end{itemize}
\end{prop}

%% Remark

\begin{rem}
\begin{itemize}
\item[(i)]
Within the aforementioned conditions, we do not impose any restriction on the linear growth constant of the drift coefficient $b$, other than the coercivity condition $\langle b(x),x\rangle\leq\gamma|x|^2+C_{b,\LG}'$ with the specific positive constant $\gamma$ introduced above. This kind of coercivity condition often arises in Lyapunov-type estimates for SDEs. It is worth noting, however, that controlling the system under such an instantaneous condition is far from trivial in the context of the SVE \eqref{eq_SVE}, as the dynamics exhibit path-dependence due to the presence of the kernels. Similarly, it is also non-trivial from the viewpoint of the SEE \eqref{eq_SEE}, as its coefficients involve the non-local term $\mu[y]$ with low regularity. Nevertheless, \cref{theo_Lyapunov} and \cref{prop_Lyapunov-sufficient} establish suitable Lyapunov-type estimates. A key ingredient in these estimates is the construction of an admissible weight function $\Psi$ that exploits the specific structure of the SEE \eqref{eq_SEE} to satisfy the requirements in \cref{assum_Lyapunov-abstract}. For more details, see \cref{subsec_proof-Lyapunov}.
\item[(ii)]
By \eqref{eq_integrability-mu}, \eqref{eq_integrability-M_b}, and the assumption $\kappa:=\supp>0$, the Cauchy--Schwarz inequality yields
\begin{equation*}
	\int_{[\kappa,\infty)}\theta^{-1}|M_b(\theta)|_\op\dmu\leq\left(\int_{[\kappa,\infty)}\theta^{-1/2}\dmu\right)^{1/2}\left(\int_{[\kappa,\infty)}\theta^{-3/2}|M_b(\theta)|_\op^2\dmu\right)^{1/2}<\infty.
\end{equation*}
If $(K_b,K_\sigma)$ is generated by the lifting basis $(\mu,M_b,M_\sigma)$, the quantity $\int_{[\kappa,\infty)}\theta^{-1}|M_b(\theta)|_\op\dmu$ is closely related to the integral $\int^\infty_0|K_b(t)|_\op\,\diff t$. Indeed, the representation $K_b(t)=\int_{[\kappa,\infty)}e^{-\theta t}M_b(\theta)\dmu$ and Tonelli's theorem yield
\begin{equation*}
	\int^\infty_0|K_b(t)|_\op\,\diff t\leq\int^\infty_0\int_{[\kappa,\infty)}e^{-\theta t}|M_b(\theta)|_\op\dmu\,\diff t=\int_{[\kappa,\infty)}\theta^{-1}|M_b(\theta)|_\op\dmu<\infty.
\end{equation*}
Furthermore, by Fubini's theorem, we have
\begin{equation*}
	\int^\infty_0K_b(t)\,\diff t=\int^\infty_0\int_{[\kappa,\infty)}e^{-\theta t}M_b(\theta)\dmu\,\diff t=\int_{[\kappa,\infty)}\theta^{-1}M_b(\theta)\dmu.
\end{equation*}
Thus, if there exists a vector $v\in\bR^n\setminus\{0\}$ (independent of $\theta$) such that $M_b(\theta)v=|M_b(\theta)|_\op v$ for $\mu$-a.e.\ $\theta\in[0,\infty)$, then the following identities hold:
\begin{equation*}
	\left|\int^\infty_0K_b(t)\,\diff t\right|_\op=\int^\infty_0|K_b(t)|_\op\,\diff t=\int_{[\kappa,\infty)}\theta^{-1}|M_b(\theta)|_\op\dmu=\left|\int_{[\kappa,\infty)}\theta^{-1}M_b(\theta)\dmu\right|_\op.
\end{equation*}
\item[(iii)]
For the lifting basis in \cref{exam_liftable} (i) generating the sum-of-exponentials type kernels $(K^{\exp}_b,K^{\exp}_\sigma)$, the first and second conditions in \cref{prop_Lyapunov-sufficient} are satisfied provided that $\kappa_i>0$ and $M_{b,i}$ is symmetric and nonnegative definite for each $i\in\{1,\dots,N\}$. In this case, we have $\kappa:=\inf\supp=\min\{\kappa_1,\dots,\kappa_N\}$ and $\int_{[\kappa,\infty)}\theta^{-1}|M_b(\theta)|_\op\dmu=\sum^N_{i=1}\frac{|M_{b,i}|_\op}{\kappa_i}$. As for the lifting basis in \cref{exam_liftable} (ii) generating the tempered fractional kernels $(K^\fractional_b,K^\fractional_\sigma)$, the first and second conditions in \cref{prop_Lyapunov-sufficient} are satisfied provided that $\kappa_b,\kappa_\sigma>0$. In this case, we have $\kappa:=\inf\supp=\kappa_b\wedge\kappa_\sigma$ and $\int_{[\kappa,\infty)}\theta^{-1}|M_b(\theta)|_\op\dmu=\int^\infty_0|K_b(t)|_\op\,\diff t=\kappa_b^{-\alpha_b}$. These observations are comparable with \cref{rem_exam-liftable} (ii).
\end{itemize}
\end{rem}

As a direct consequence of \cref{theo_coupling} and \cref{theo_Lyapunov}, and by virtue of the generalized Harris' theorem of \cite{HaMaSc11} (see also \cref{theo_HaMaSc11}), we obtain the following main result regarding the exponential ergodicity of the Markovian lift.

%% Theorem

\begin{theo}\label{theo_main}
Suppose that a lifting basis $(\mu,M_b,M_\sigma)$ and measurable maps $b:\bR^n\to\bR^n$, $\sigma:\bR^n\to\bR^{n\times d}$ satisfy both \cref{assum_coupling} and \cref{assum_Lyapunov-abstract}. Then, for the Markov semigroup $\{P_t\}_{t\geq0}$ associated with the SEE \eqref{eq_SEE}, the following assertions hold:
\begin{itemize}
\item[(i)]
There exist constants $r,t_0>0$ such that the following spectral gap-type estimate holds:
\begin{equation}\label{eq_spectral-gap}
	\bW_{d_{\Phi,\Psi}}(P_t^*\nu_1,P_t^*\nu_2)\leq e^{-rt}\bW_{d_{\Phi,\Psi}}(\nu_1,\nu_2)
\end{equation}
for all $t\geq t_0$ and all $\nu_1,\nu_2\in\cP(\cH)$, where $d_{\Phi,\Psi}:\cH\times\cH\to[0,\infty)$ is a distance-like function given by
\begin{equation}\label{eq_distance-like}
	d_{\Phi,\Psi}(y_1,y_2):=\sqrt{\big(\|y_1-y_2\|_\Phi\wedge1\big)\big(1+\|y_1\|_{\Psi}^2+\|y_2\|_{\Psi}^2\big)},\ \ y_1,y_2\in\cH,
\end{equation}
and $\Phi$ and $\Psi$ are the admissible weight functions provided in \cref{theo_coupling} and \cref{assum_Lyapunov-abstract}, respectively.
\item[(ii)]
There exists a unique invariant probability measure $\pi\in\cP(\cH)$ for $\{P_t\}_{t\geq0}$, which satisfies
\begin{equation}\label{eq_IPM-integrability}
	\pi(\cV)=1\ \ \text{and}\ \ \int_{\cV}\|y\|_{\cV}^2\,\pi(\diff y)<\infty.
\end{equation}
Furthermore, there exists a constant $C>0$ such that the following exponential ergodicity holds:
\begin{equation}\label{eq_ergodicity}
	\bW_{d_{\Phi,\Psi}}(P_t(y,\cdot),\pi)\leq C(1+\|y\|_{\Psi})e^{-rt}
\end{equation}
for all $y\in\cH$ and all $t\geq0$.
\end{itemize}
\end{theo}

%% Remark

\begin{rem}\label{rem_ergodicity}
\begin{itemize}
\item[(i)]
The estimate \eqref{eq_ergodicity} immediately yields that $P_t(y,\cdot)$ converges to the invariant probability measure $\pi$ weakly in $\cP(\cH)$ as $t\to\infty$ for any $y\in\cH$. Furthermore, note that there exists a constant $C>0$ such that
\begin{equation}\label{eq_H-distance-like}
	\|y_1-y_2\|_{\cH}\leq\sqrt{\big(\|y_1-y_2\|_\cH\wedge1\big)\big(1+2\|y_1\|_\cH^2+2\|y_2\|_\cH^2\big)}\leq Cd_{\Phi,\Psi}(y_1,y_2)
\end{equation}
for all $y_1,y_2\in\cH$. Consequently, \eqref{eq_ergodicity} yields the following exponential ergodicity in the $L^1$-Wasserstein distance:
\begin{equation}\label{eq_ergodicity-1}
	\bW_\cH(P_t(y,\cdot),\pi)\leq C(1+\|y\|_\cH)e^{-rt}
\end{equation}
for all $y\in\cH$ and all $t\geq0$ for some constant $C>0$, where $\bW_\cH:\cP(\cH)\times\cP(\cH)\to[0,\infty]$ is the $L^1$-Wasserstein (extended) metric\footnote{The extended metric $\bW_\cH:\cP(\cH)\times\cP(\cH)\to[0,\infty]$ is a metric on $\{\nu\in\cP(\cH)\,|\,\int_\cH\|y\|_\cH\,\nu(\diff y)<\infty\}$, while it can take an infinite value outside this set.} defined by
\begin{equation}\label{eq_Wasserstein-defi}
	\bW_\cH(\nu_1,\nu_2):=\inf_{\nu\in\sC(\nu_1,\nu_2)}\int_{\cH\times\cH}\|y_1-y_2\|_\cH\,\nu(\diff y_1,\diff y_2),\ \ \nu_1,\nu_2\in\cP(\cH).
\end{equation}
%By means of the Kantorovich--Rubinstein duality formula (see, e.g., \cite[Theorem 5.10 and Particular Case 5.16]{Vi09}), we have
%\begin{equation*}
%	\bW_\cH(\fp_1,\fp_2)=\sup_{f\in\Lip(\cH)}\left|\int_\cH f(y)\,\fp_1(\diff y)-\int_\cH f(y)\,\fp_2(\diff y)\right|
%\end{equation*}
%for any $\fp_1,\fp_2\in\cP_1(\cH)$, where
%\begin{equation*}
%	\Lip(\cH):=\left\{f:\cH\to\bR\relmiddle||f(y_1)-f(y_2)|\leq\|y_1-y_2\|_\cH\ \text{for any $y_1,y_2\in\cH$}\right\}.
%\end{equation*}
\item[(ii)]
The exponential weak ergodicity \eqref{eq_ergodicity} (or \eqref{eq_ergodicity-1}) has various applications to limit theorems, including the \emph{law of large numbers}, the \emph{central limit theorem}, the \emph{averaging principle} and the \emph{diffusion approximation}; see \cite[Chapters 5 and 6]{Ku18}. We will study these topics in more detail in future research. On the other hand, the ``spectral gap'' result \eqref{eq_spectral-gap} by itself plays a crucial role in showing the stability of invariant probability measures; see \cite[Section 4.1]{HaMaSc11}. In \cref{sec_approximation}, we exploit this spectral gap result to analyze a finite-dimensional approximation of the invariant probability measure $\pi\in\cP(\cH)$.
\end{itemize}
\end{rem}

Following the notation and observations in \cref{rem_path-distribution}, we obtain the following corollary as a direct consequence of \cref{theo_main}.

%% Corollary

\begin{cor}\label{cor_stationary-path}
Under the setting of \cref{theo_main}, the following assertions hold:
\begin{itemize}
\item[(i)]
The probability measure $\bfP^\pi\in\cP(\Lambda)$ is invariant under the time-shifts on $\Lambda$, in the sense that $\bfP^\pi\circ\sS_\tau^{-1}=\bfP^\pi$ in $\cP(\Lambda)$ for all $\tau\geq0$. Furthermore, for every $y\in\cH$, the convergence $\lim_{\tau\to\infty}\bfP^{\delta_y}\circ\sS_\tau^{-1}=\bfP^\pi$ holds weakly in $\cP(\Lambda)$.
\item[(ii)]
The probability measure $\bfQ^\pi\in\cP(L^2_\loc(0,\infty;\bR^n))$ is invariant under the time-shifts on $L^2_\loc(0,\infty;\bR^n)$, in the sense that $\bfQ^\pi\circ\sS_\tau^{-1}=\bfQ^\pi$ in $\cP(L^2_\loc(0,\infty;\bR^n))$ for all $\tau\geq0$. Furthermore, for every $y\in\cH$, the convergence $\lim_{\tau\to\infty}\bfQ^{\delta_y}\circ\sS_\tau^{-1}=\bfQ^\pi$ holds weakly in $\cP(L^2_\loc(0,\infty;\bR^n))$.
\end{itemize}
\end{cor}

%% Proof

\begin{proof}
Since $\pi\in\cP(\cH)$ is invariant with respect to $\{P_t\}_{t\geq0}$ and $\lim_{t\to\infty}P_t(y,\cdot)=\pi$ weakly in $\cP(\cH)$ for any $y\in\cH$, the assertions follow from the identities $\bfP^\nu\circ\sS_\tau^{-1}=\bfP^{P_\tau^*\nu}$ in $\cP(\Lambda)$ and $\bfQ^\nu\circ\sS_\tau^{-1}=\bfQ^{P_\tau^*\nu}$ in $\cP(L^2_\loc(0,\infty;\bR^n))$ for $\tau\geq0$ and $\nu\in\cP(\cH)$, combined with the continuity of $\cP(\cH)\ni\nu\mapsto\bfP^\nu\in\cP(\Lambda)$ and $\cP(\cH)\ni\nu\mapsto\bfQ^\nu\in\cP(L^2_\loc(0,\infty;\bR^n))$ in the weak sense (see \cref{rem_path-distribution}).
\end{proof}

%% Remark

\begin{rem}\label{rem_Q-tilde}
Strictly speaking, the invariance of $\bfQ^\pi\in\cP(L^2_\loc(0,\infty;\bR^n))$ under the time-shifts on $L^2_\loc(0,\infty;\bR^n)$ established in \cref{cor_stationary-path} (ii) does not directly imply the stationarity of the finite-dimensional distributions of the process itself. Nevertheless, we can construct a stationary solution of the SVE \eqref{eq_SVE} in the latter sense. Indeed, under the setting of \cref{theo_main}, the invariance of $\pi\in\cP(\cH)$ with respect to $\{P_t\}_{t\geq0}$ and $\pi(\cV)=1$ ensure that
\begin{align*}
	\bfP^\pi\big(\left\{\eta\in\Lambda\relmiddle|\eta_{t_1},\dots,\eta_{t_\ell}\in\cV\right\}\big)=\int_\cV\pi(\diff y_0)\int_\cV P_{t_1}(y_0,\diff y_1)\int_\cV P_{t_2-t_1}(y_1,\diff y_2)\cdots\int_\cV P_{t_\ell-t_{\ell-1}}(y_{\ell-1},\diff y_\ell)=1
\end{align*}
for any $t_1,\dots,t_\ell\in[0,\infty)$.\footnote{Thanks to $\cB(\cV)=\{A\cap\cV\,|\,A\in\cB(\cH)\}\subset\cB(\cH)$ and \cref{lemm_app-Lambda}, we have $\{\eta\in\Lambda\,|\,\eta_{t_1},\dots,\eta_{t_\ell}\in\cV\}\in\cB(\Lambda)$, or more generally $\{\eta\in\Lambda\,|\,(\eta_{t_1},\dots,\eta_{t_\ell})\in A\}\in\cB(\Lambda)$ for any $A\in\cB(\cV^\ell)=\cB(\cV)^{\otimes\ell}$. Note, however, that the set $\{\eta\in\Lambda\,|\,\eta_t\in\cV\ \text{for any $t\geq0$}\}$ may not be a Borel set in $\Lambda$, as $t\mapsto\eta_t$ is not necessarily continuous in $\cV$ for each $\eta\in\Lambda$.} Based on this observation, together with $\mu[\cdot]\in L(\cV;\bR^n)$, we can define a consistent family of finite-dimensional distributions $\widetilde{\bfQ}^\pi_{(t_1,\dots,t_\ell)}\in\cP((\bR^n)^\ell)$ by
\begin{equation*}
	\widetilde{\bfQ}_{(t_1,\dots,t_\ell)}^\pi(A_1\times\cdots \times A_\ell):=\bfP^\pi\big(\left\{\eta\in\Lambda\relmiddle|\eta_{t_1},\dots,\eta_{t_\ell}\in\cV\ \text{and}\ \mu[\eta_{t_1}]\in A_1,\dots,\mu[\eta_{t_\ell}]\in A_\ell\right\}\big)
\end{equation*}
for any $A_1,\dots,A_\ell\in\cB(\bR^n)$, $t_1,\dots,t_\ell\in[0,\infty)$ and $\ell\in\bN$. Let $(\bR^n)^{[0,\infty)}$ be the space of all maps $\xi:[0,\infty)\to\bR^n$, which is equipped with the $\sigma$-algebra generated by the family of cylinder sets $\{\{\xi\in(\bR^n)^{[0,\infty)}\,|\xi_t\in A\}\,|\,t\in[0,\infty),\,A\in\cB(\bR^n)\}$. By the Kolmogorov extension theorem, there exists a unique probability measure $\widetilde{\bfQ}^\pi$ on the measurable space $(\bR^n)^{[0,\infty)}$ such that
\begin{equation*}
	\widetilde{\bfQ}^\pi\Big(\left\{\xi\in(\bR^n)^{[0,\infty)}\relmiddle|\big(\xi_{t_1},\dots,\xi_{t_\ell}\big)\in\cdot\right\}\Big)=\widetilde{\bfQ}^\pi_{(t_1,\dots,t_\ell)}\ \ \text{in $\cP((\bR^n)^\ell)$}
\end{equation*}
for any $t_1,\dots,t_\ell\in[0,\infty)$ and $\ell\in\bN$. The probability measure $\widetilde{\bfQ}^\pi\in\cP((\bR^n)^{[0,\infty)})$ represents the law of the process $X:=\mu[Y]=(\mu[Y_t])_{t\geq0}$ itself (rather than its $\diff t$-equivalence class as $\bfQ^\pi\in\cP(L^2_\loc(0,\infty;\bR^n))$), where $Y$ is a solution of the SEE \eqref{eq_SEE} with the initial distribution $\pi\in\cP(\cH)$. By \cref{prop_SEE-well-posed} (iii), $X$ solves the SVE \eqref{eq_SVE}. Furthermore, \cref{cor_stationary-path} (i) implies that the probability measure $\widetilde{\bfQ}^\pi\in\cP((\bR^n)^{[0,\infty)})$ is invariant under the time-shifts on $(\bR^n)^{[0,\infty)}$, in the sense that $\widetilde{\bfQ}^\pi\circ\sS_\tau^{-1}=\widetilde{\bfQ}^\pi$ in $\cP((\bR^n)^{[0,\infty)})$ for all $\tau\geq0$. Here, with a slight abuse of notation, as in \cref{rem_path-distribution}, we again denote by $\sS_\tau$ the time-shift operator $(\xi_t)_{t\geq0}\mapsto(\xi_{t+\tau})_{t\geq0}$ on $(\bR^n)^{[0,\infty)}$. This shift-invariance of $\widetilde{\bfQ}^\pi$ precisely means that the $\bR^n$-valued process $X=\mu[Y]$ is strictly stationary. We emphasize that $\bfQ^\pi\in\cP(L^2_\loc(0,\infty;\bR^n))$ and $\widetilde{\bfQ}^\pi\in\cP((\bR^n)^{[0,\infty)})$ are different objects; the former characterizes the law of the $\diff t$-equivalence class of the process $X=\mu[Y]$, and the latter characterizes the finite-dimensional distributions of the process $X=\mu[Y]$ itself (which can be interpreted as a ``good representative'' among the $\diff t$-equivalence class). In contrast to \cref{cor_stationary-path} (ii), our current result does \emph{not} imply the weak convergence of shifted finite-dimensional distributions, due to the general lack of continuity of the map $\Lambda\ni\eta\mapsto(\mu[\eta_{t_1}\1_\cV(\eta_{t_1})],\dots,\mu[\eta_{t_\ell}\1_\cV(\eta_{t_\ell})])\in(\bR^n)^\ell$ for fixed $t_1,\dots,t_\ell\in[0,\infty)$.
\end{rem}

Let us further examine the stationary distributions $\bfQ^\pi\in\cP(L^2_\loc(0,\infty;\bR^n))$ and $\widetilde{\bfQ}^\pi\in\cP((\bR^n)^{[0,\infty)})$ corresponding to the SVE \eqref{eq_SVE}. Admittedly, the state space $\cH$, the Markov semigroup $\{P_t\}_{t\geq0}$ on $\cH$, and the invariant probability measure $\pi\in\cP(\cH)$ associated with the SEE \eqref{eq_SEE} all depend on the specific choice of the lifting basis $(\mu,M_b,M_\sigma)$ which generates the kernels $K_b$ and $K_\sigma$. Nevertheless, it is natural to expect that the resulting stationary distributions for the original SVE \eqref{eq_SVE} depend only on the kernels $K_b,K_\sigma$ and the coefficients $b,\sigma$, remaining independent of the particular choice of the lifting basis. The following theorem shows that this is indeed the case in some sense.

%% Theorem

\begin{theo}\label{theo_lift-independence}
Let $(K_b,K_\sigma)$ be a liftable pair of kernels, and let $b:\bR^n\to\bR^n$ and $\sigma:\bR^n\to\bR^{n\times d}$ be measurable maps. Suppose that $(\mu_1,M_{b,1},M_{\sigma,1})$ and $(\mu_2,M_{b,2},M_{\sigma,2})$ are two lifting bases, both of which generate $(K_b,K_\sigma)$, such that \cref{assum_coupling} and \cref{assum_Lyapunov-abstract} are satisfied for each. For $i\in\{1,2\}$, consider the SEE \eqref{eq_SEE} on the Gelfand triplet $\cV_i\hookrightarrow\cH_i\hookrightarrow\cV^*_i$ associated with $(\mu_i,M_{b,i},M_{\sigma,i})$, and denote by $\{\bfP_i^{\nu_i}\}_{\nu_i\in\cP(\cH_i)}\subset\cP(\Lambda_i)$, $\{P_{i,t}\}_{t\geq0}$, and $\pi_i\in\cP(\cH_i)$ the corresponding laws of the paths of the solutions on the path space $\Lambda_i$, the associated Markov semigroup, and its invariant probability measure, respectively. Furthermore, let $\bfQ_i^{\pi_i}\in\cP(L^2_\loc(0,\infty;\bR^n))$ and $\widetilde{\bfQ}_i^{\pi_i}\in\cP((\bR^n)^{[0,\infty)})$ be the stationary distributions constructed as in \cref{rem_path-distribution} (ii) and in \cref{rem_Q-tilde}. Then, it holds that
\begin{equation}\label{eq_lift-independence-L^2}
	\bfQ_1^{\pi_1}=\bfQ_2^{\pi_2}\ \ \text{in $\cP\big(L^2_\loc(0,\infty;\bR^n)\big)$}
\end{equation}
and
\begin{equation}\label{eq_lift-independence-fdd}
	\widetilde{\bfQ}_1^{\pi_1}=\widetilde{\bfQ}_2^{\pi_2}\ \ \text{in $\cP\big((\bR^n)^{[0,\infty)}\big)$}.
\end{equation}
\end{theo}

%% Proof

\begin{proof}
Since both $(\mu_1,M_{b,1},M_{\sigma,1})$ and $(\mu_2,M_{b,2},M_{\sigma,2})$ generate the same kernels $K_b$ and $K_\sigma$, the characterization of Borel measures by their Laplace transforms implies that
\begin{equation*}
	M_{b,1}(\theta)\,\mu_1(\diff\theta)=M_{b,2}(\theta)\,\mu_2(\diff\theta)\ \ \text{and}\ \ M_{\sigma,1}(\theta)\,\mu_1(\diff\theta)=M_{\sigma,2}(\theta)\,\mu_2(\diff\theta).
\end{equation*}
Furthermore, since $M_{\sigma,i}(\theta)$ is positive definite for $\mu_i$-a.e.\ $\theta\in[0,\infty)$ by \cref{assum_coupling} for each $i\in\{1,2\}$, we see that the two measures $\mu_1$ and $\mu_2$ are equivalent. Set $\overline{\mu}:=\mu_1+\mu_2$. Clearly, we have $0<\frac{\diff\mu_i}{\diff\overline{\mu}}<1$ $\overline{\mu}$-a.e. for $i\in\{1,2\}$. Define
\begin{equation*}
	\overline{M}_b(\theta):=\frac{\diff\mu_i}{\diff\overline{\mu}}(\theta)M_{b,i}(\theta)\ \ \text{and}\ \ \overline{M}_\sigma(\theta):=\frac{\diff\mu_i}{\diff\overline{\mu}}(\theta)M_{\sigma,i}(\theta)
\end{equation*}
for each $\theta\in[0,\infty)$ and $i\in\{1,2\}$. Note that $\overline{M}_b$ and $\overline{M}_\sigma$ are independent of $i\in\{1,2\}$. Moreover, $(\overline{\mu},\overline{M}_b,\overline{M}_\sigma)$ constitutes a lifting basis in the sense of \cref{defi_lifting-basis} and satisfies the first two conditions in \cref{assum_coupling}. For the given maps $b:\bR^n\to\bR^n$ and $\sigma:\bR^n\to\bR^{n\times d}$, consider the SEE \eqref{eq_SEE} on the Gelfand triplet $\overline{\cV}\hookrightarrow\overline{\cH}\hookrightarrow\overline{\cV}^*$ associated with $(\overline{\mu},\overline{M}_b,\overline{M}_\sigma)$, and denote the corresponding Markov semigroup on $\overline{\cH}$ by $\{\overline{P}_t\}_{t\geq0}$. For each $\overline{\nu}\in\cP(\overline{\cH})$, denote by $\overline{\bfP}^{\overline{\nu}}$ the law of the path of a solution with the initial distribution $\overline{\nu}$, which is a probability measure on the associated path space $\overline{\Lambda}$.

Let $i\in\{1,2\}$ be fixed. Define $\Xi_i:\cV_i^*\to\overline{\cV}^*$ by $\Xi_iy_i:=\frac{\diff\mu_i}{\diff\overline{\mu}}y_i$ for each $y_i\in\cV^*_i$. Clearly, this multiplication operator $\Xi_i$ belongs to $L(\cV^*_i,\overline{\cV}^*)$, $L(\cH_i,\overline{\cH})$ and $L(\cV_i,\overline{\cV})$. Furthermore, we have
\begin{equation}\label{eq_lift-independence-1}
	\mu_i[y_i]=\overline{\mu}[\Xi_iy_i]\ \ \text{for any $y_i\in\cV_i$}.
\end{equation}
Thus, for each solution $Y_i=(Y_{i,t})_{t\geq0}$ of the SEE \eqref{eq_SEE} corresponding to the lifting basis $(\mu_i,M_{b,i},M_{\sigma,i})$ with a given initial distribution $\nu_i\in\cP(\cH_i)$, the process $(\Xi_iY_{i,t})_{t\geq0}$ is a solution of the SEE \eqref{eq_SEE} corresponding to the lifting basis $(\overline{\mu},\overline{M}_b,\overline{M}_\sigma)$ with the initial distribution $\nu_i\circ\Xi_i^{-1}\in\cP(\overline{\cH})$. Hence, we have
\begin{equation}\label{eq_lift-independence-2}
	\bfP_i^{\nu_i}\circ\Xi_i^{-1}=\overline{\bfP}^{\nu_i\circ\Xi_i^{-1}}\ \ \text{in $\cP\big(\overline{\Lambda}\big)$ for any $\nu_i\in\cP\big(\cH_i\big)$},
\end{equation}
where, with a slight abuse of notation, we used the same symbol $\Xi_i$ as before to represent the map $\Lambda_i\ni(\eta_{i,t})_{t\geq0}\mapsto(\Xi_i\eta_{i,t})_{t\geq0}\in\overline{\Lambda}$. In particular, we have $(P_{i,t}^*\nu_i)\circ\Xi_i^{-1}=\overline{P}_t^*(\nu_i\circ\Xi_i^{-1})$ in $\cP(\overline{\cH})$ for any $\nu_i\in\cP(\cH_i)$ and $t\geq0$. Hence, for the invariant probability measure $\pi_i\in\cP(\cH_i)$ with respect to the Markov semigroup $\{P_{i,t}\}_{t\geq0}$, we have
\begin{equation*}
	\overline{P}_t^*\big(\pi_i\circ\Xi_i^{-1}\big)=(P_{i,t}^*\pi_i)\circ\Xi_i^{-1}=\pi_i\circ\Xi_i^{-1}
\end{equation*}
in $\cP(\overline{\cH})$ for any $t\geq0$. This indicates that $\pi_i\circ\Xi_i^{-1}\in\cP(\overline{\cH})$ is an invariant probability measure for the Markov semigroup $\{\overline{P}_t\}_{t\geq0}$. This conclusion remains valid for each $i\in\{1,2\}$. However, since the lifting basis $(\overline{\mu},\overline{M}_b,\overline{M}_\sigma)$ and the maps $b,\sigma$ satisfy \cref{assum_coupling}, \cref{theo_coupling} (iii) shows that the Markov semigroup $\{\overline{P}_t\}_{t\geq0}$ on $\overline{\cH}$ has at most one invariant probability measure. Hence, it must hold that
\begin{equation}\label{eq_lift-independence-3}
	\pi_1\circ\Xi_1^{-1}=\pi_2\circ\Xi_2^{-1}\ \ \text{in $\cP\big(\overline{\cH}\big)$}.
\end{equation}

Observe that, for each $i\in\{1,2\}$,
\begin{equation*}
	\bfQ_i^{\pi_i}=\bfP_i^{\pi_i}\circ\mu_i[\cdot]^{-1}=\big(\bfP_i^{\pi_i}\circ\Xi_i^{-1}\big)\circ\overline{\mu}[\cdot]^{-1}=\overline{\bfP}^{\pi_i\circ\Xi_i^{-1}}\circ\overline{\mu}[\cdot]^{-1}
\end{equation*}
in $\cP(L^2_\loc(0,\infty;\bR^n))$, where the first equality is precisely the definition of the probability measure $\bfQ_i^{\pi_i}\in\cP(L^2_\loc(0,\infty;\bR^n))$, the second equality is due to \eqref{eq_lift-independence-1}, and the third equality is due to \eqref{eq_lift-independence-2}. From this, together with \eqref{eq_lift-independence-3}, we obtain the desired first equality \eqref{eq_lift-independence-L^2}. To prove the second equality \eqref{eq_lift-independence-fdd}, take arbitrary $t_1,\dots,t_\ell\in[0,\infty)$ and $A_1,\dots,A_\ell\in\cB(\bR^n)$ with $\ell\in\bN$. Fix $i\in\{1,2\}$. Define $B_i,C_i\in\cB(\Lambda_i)$ and $\overline{B},\overline{C}\in\cB(\overline{\Lambda})$ by
\begin{align*}
	&B_i:=\left\{(\eta_{i,t})_{t\geq0}\in\Lambda_i\relmiddle|\eta_{i,t_1},\dots,\eta_{i,t_\ell}\in\cV_i\right\},\ \ C_i:=\left\{(\eta_{i,t})_{t\geq0}\in B_i\relmiddle|\mu_i[\eta_{i,t_1}]\in A_1,\dots,\mu_i[\eta_{i,t_\ell}]\in A_\ell\right\},\\
	&\overline{B}:=\left\{(\overline{\eta}_t)_{t\geq0}\in\overline{\Lambda}\relmiddle|\overline{\eta}_{t_1},\dots,\overline{\eta}_{t_\ell}\in\overline{\cV}\right\},\ \ \overline{C}:=\left\{(\overline{\eta}_t)_{t\geq0}\in\overline{B}\relmiddle|\overline{\mu}[\overline{\eta}_{t_1}]\in A_1,\dots,\overline{\mu}[\overline{\eta}_{t_\ell}]\in A_\ell\right\}.
\end{align*}
By \eqref{eq_lift-independence-1}, we see that $C_i=B_i\cap(\Xi_i^{-1}\overline{C})$. Furthermore, noting \cref{rem_Q-tilde}, we have $\bfP_i^{\pi_i}(B_i)=1$. Hence,
\begin{equation*}
	\widetilde{\bfQ}_i^{\pi_i}\left(\left\{(\xi_t)_{t\geq0}\in(\bR^n)^{[0,\infty)}\relmiddle|\xi_{t_1}\in A_1,\dots,\xi_{t_\ell}\in A_\ell\right\}\right)\\=\bfP_i^{\pi_i}(C_i)=\bfP_i^{\pi_i}\big(\Xi_i^{-1}\overline{C}\big)=\overline{\bfP}^{\pi_i\circ\Xi_i^{-1}}\big(\overline{C}\big),
\end{equation*}
where the first equality is due to the definition of the probability measure $\widetilde{\bfQ}_i^{\pi_i}\in\cP((\bR^n)^{[0,\infty)})$, the second equality is due to $C_i=B_i\cap(\Xi_i^{-1}\overline{C})$ and $\bfP_i^{\pi_i}(B_i)=1$, and the third equality is due to \eqref{eq_lift-independence-2}. The above and \eqref{eq_lift-independence-3} ensure that every finite-dimensional distributions of $\widetilde{\bfQ}_1^{\pi_1}\in\cP((\bR^n)^{[0,\infty)})$ and $\widetilde{\bfQ}_2^{\pi_2}\in\cP((\bR^n)^{[0,\infty)})$ coincide, proving the desired second equality \eqref{eq_lift-independence-fdd}. This completes the proof.
\end{proof}

%
%%% Corollary
%
%\begin{cor}\label{cor_mixing}
%Under the same setting as in \cref{theo_main}, let $Y^\st=(Y^\st_t)_{t\geq0}$ be the solution of the SEE with the initial distribution $\pi$ defined on a complete probability space $(\Omega,\cF,\bP)$. Then, we have $Y^\st_t\in\cV$ a.s.\ for any $t\geq0$. Moreover, $Y^\st$ is strictly stationary and mixing, and the following strong law of large numbers holds: for any measurable function $f:\cV\to\bR$ with $\sup_{y\in\cV}\frac{|f(y)|}{1+\|y\|_{\cV}^2}<\infty$, it holds that
%\begin{equation*}
%	\lim_{T\to\infty}\frac{1}{T}\int^T_0f(Y^\st_t)\,\diff t=\int_{\cV}f(y)\,\pi(\diff y)\ \ \text{a.s.}
%\end{equation*}
%Define an $\bR^n$-valued process $X^\st$ by $X^\st_t:=\mu[Y^\st_t]$, $t\geq0$. Then, $X^\st$ is the solution of the SVE with the forcing term $x(t)=\int_{[0,\infty)}e^{-\theta t}Y^\st_0(\theta)\dmu$, $t\geq0$. Furthermore, $X^\st$ is strictly stationary and mixing, and the following strong law of large numbers holds: for any measurable function $g:\bR^n\to\bR$ with $\sup_{x\in\bR^n}\frac{|g(x)|}{1+|x|^2}<\infty$, it holds that
%\begin{equation*}
%	\lim_{T\to\infty}\frac{1}{T}\int^T_0g(X^\st_t)\,\diff t=\int_{\cV} g(\mu[y])\,\pi(\diff y)\ \ \text{a.s.}
%\end{equation*}
%\end{cor}

%%%%%%%%%%%%%%%%%%%%%%%%%%%%%%%%%%
%%%%%%%%%%%%%%%%%%%%%%%%%%%%%%%%%%
%% Section
%%%%%%%%%%%%%%%%%%%%%%%%%%%%%%%%%%
%%%%%%%%%%%%%%%%%%%%%%%%%%%%%%%%%%

\section{Proofs of the main results}\label{sec_proof}

In this section, we provide proofs of \cref{theo_coupling}, \cref{theo_Lyapunov} and \cref{prop_Lyapunov-sufficient}. The most technically demanding part is the proof of \cref{theo_coupling}, where the ``change-of-norm'' technique and the ``generalized coupling approach'' play crucial roles; see \cref{subsec_proof-coupling}. The ``change-of-norm'' technique is also employed to construct the Lyapunov function in \cref{theo_Lyapunov} and \cref{prop_Lyapunov-sufficient}, as detailed in \cref{subsec_proof-Lyapunov}.

%%%%%%%%%%%%
%% Subsection
%%%%%%%%%%%%

\subsection{Proof of \cref{theo_coupling}: Contractivity, $d$-smallness, and the asymptotic log-Harnack inequality}\label{subsec_proof-coupling}

As shown in the work \cite{BuKuSc20}, in order to construct a ``good'' distance(-like) function satisfying the required conditions in the generalized Harris' theorem of \cite{HaMaSc11} (see \cref{theo_HaMaSc11}), a tractable strategy is to consider a \emph{generalized coupling} of probability measures. We follow this idea to show \cref{theo_coupling}. The most important and technical point is to construct a suitable admissible weight function and a generalized coupling satisfying the required conditions in \cite{BuKuSc20}; see \cref{lemm_coupling} below.

As a first step, we present the following standard lemma, which reduces the proof of \cref{theo_coupling} to the case where the maps $b:\bR^n\to\bR^n$ and $\sigma:\bR^n\to\bR^{n\times d}$ are bounded.

%% Lemma

\begin{lemm}\label{lemm_bounded}
Let $(\mu,M_b,M_\sigma)$ be a lifting basis. Suppose that $b:\bR^n\to\bR^n$ and $\sigma:\bR^n\to\bR^{n\times d}$ are measurable maps satisfying $|b(x)-b(x')|\leq C_{b,\Lip}|x-x'|$ and $|\sigma(x)-\sigma(x')|\leq C_{\sigma,\Lip}|x-x'|$ for all $x,x'\in\bR^n$, with some constants $C_{b,\Lip},C_{\sigma,\Lip}>0$. For each $N\in\bN$, define the truncated maps $b^N:\bR^n\to\bR^n$ and $\sigma^N:\bR^n\to\bR^{n\times d}$ by
\begin{equation*}
	b^N(x):=\begin{dcases}b(x)\ &\text{if $|x|\leq N$},\\b\left(\frac{N}{|x|}x\right)\ &\text{if $|x|>N$},\end{dcases}\ \ \text{and}\ \ \sigma^N(x):=\begin{dcases}\sigma(x)\ &\text{if $|x|\leq N$},\\\sigma\left(\frac{N}{|x|}x\right)\ &\text{if $|x|>N$}.\end{dcases}
\end{equation*}
Let $\{P_t\}_{t\geq0}$ and $\{P^N_t\}_{t\geq0}$ be the Markov semigroups associated with the SEEs \eqref{eq_SEE} with coefficients $(b,\sigma)$ and $(b^N,\sigma^N)$, respectively. Then, for any $t\geq0$ and any $y\in\cH$, we have $\lim_{N\to\infty}P^N_t(y,\cdot)=P_t(y,\cdot)$ weakly in $\cP(\cH)$.
\end{lemm}

%% Proof

\begin{proof}
Let $N\in\bN$ be fixed. Note that
\begin{equation}\label{eq_Lip-N}
	|b^N(x)-b^N(x')|\leq C_{b,\Lip}|x-x'|\ \ \text{and}\ \ |\sigma^N(x)-\sigma^N(x')|\leq C_{\sigma,\Lip}|x-x'|
\end{equation}
for all $x,x'\in\bR^n$, and
\begin{equation}\label{eq_Lip-N'}
	|b(x)-b^N(x)|\leq C_{b,\Lip}|x|\1_{\{|x|>N\}}\ \ \text{and}\ \ |\sigma(x)-\sigma^N(x)|\leq C_{\sigma,\Lip}|x|\1_{\{|x|>N\}}
\end{equation}
for all $x\in\bR^n$. Fix an arbitrary $y\in\cH$. By \cref{prop_SEE-well-posed}, there exist unique solutions $Y$ and $Y^N$ to the SEEs \eqref{eq_SEE} with coefficients $(b,\sigma)$ and $(b^N,\sigma^N)$, respectively, driven by a common $d$-dimensional Brownian motion $W$ on a complete probability space $(\Omega,\cF,\bP)$, such that $Y_0=Y^N_0=y$. Note that $\Law_\bP(Y_t)=P_t(y,\cdot)$ and $\Law_\bP(Y^N_t)=P^N_t(y,\cdot)$ for any $t\geq0$. Furthermore, by \eqref{eq_SEE-apriori}, we have
\begin{equation}\label{eq_N-integrability}
	\bE\left[\sup_{t\in[0,T]}\Big\{\big\|Y_t\big\|_\cH^2+\big\|Y^N_t\big\|_\cH^2\Big\}+\int^T_0\Big\{\big\|Y_t\big\|_\cV^2+\big\|Y^N_t\big\|_\cV^2\Big\}\,\diff t\right]<\infty
\end{equation}
for any $T>0$.

Applying It\^{o}'s formula for the squared norm $\|\cdot\|_\cH^2$ in \cref{lemm_Ito} to the $\cH$-valued process $Y-Y^N$ yields
\begin{equation}\label{eq_cutoff-Ito}
	\|Y_t-Y^N_t\|_\cH^2=\int^t_0F^N\big(Y_s,Y^N_s\big)\,\diff s+2\int^t_0\Big\langle Y_s-Y^N_s,\cM_\sigma\big(\sigma(\mu[Y_s])-\sigma^N(\mu[Y^N_s])\big)\,\diff W_s\Big\rangle_\cH
\end{equation}
for any $t\geq0$ a.s., where the function $F^N:\cV\times\cV\to\bR$ is given by
\begin{equation*}
	F^N(y_1,y_2):=2\Big\langle\cA(y_1-y_2)+\cM_b\big(b(\mu[y_1])-b^N(\mu[y_2])\big),y_1-y_2\Big\rangle_{\cV^*,\cV}+\big\|\cM_\sigma\big(\sigma(\mu[y_1])-\sigma^N(\mu[y_2])\big)\big\|_{L_2(\bR^d;\cH)}^2
\end{equation*}
for $y_1,y_2\in\cV$. Analogously to the proof of \cref{prop_SEE-well-posed}, it follows from \eqref{eq_Lip-N}, \eqref{eq_Lip-N'}, and \eqref{eq_A-nonpositive'}, together with \cref{lemm_mu-ep}, $\cM_b\in L(\bR^n;\cV^*)$, $\cM_\sigma\in L(\bR^n;\cH)$, and Young's inequality, that there exists a constant $C>0$ depending only on the lifting basis $(\mu,M_b,M_\sigma)$ and the constants $C_{b,\Lip}$ and $C_{\sigma,\Lip}$ such that
\begin{equation}\label{eq_cutoff-estimate}
	F^N(y_1,y_2)\leq C\big\{\|y_1-y_2\|_\cH^2+\big|\mu[y_1]\big|^2\1_{\{|\mu[y_1]|>N\}}\big\}
\end{equation}
for any $y_1,y_2\in\cV$. Furthermore, by \eqref{eq_Lip-N}, \eqref{eq_Lip-N'}, and \eqref{eq_N-integrability}, along with the fact that $\cM_\sigma\in L(\bR^n;\cH)$, the stochastic integral in the right-hand side of \eqref{eq_cutoff-Ito} is a martingale under $\bP$. Hence, by taking the expectations on both sides of \eqref{eq_cutoff-Ito}, applying \eqref{eq_cutoff-estimate} and using Gronwall's inequality, we obtain
\begin{equation*}
	\bE\big[\|Y_t-Y^N_t\|_\cH^2\big]\leq Ce^{Ct}\bE\left[\int^t_0\big|\mu[Y_s]\big|^2\1_{\{|\mu[Y_s]|>N\}}\,\diff s\right]
\end{equation*}
for any $t\geq0$. Since $\mu[\cdot]\in L(\cV;\bR^n)$ and $\bE[\int^t_0\|Y_s\|_\cV^2\,\diff s]<\infty$, by the dominated convergence theorem, the right-hand side above tends to zero as $N\to\infty$, and hence $\lim_{N\to\infty}\bE[\|Y_t-Y^N_t\|_\cH^2]=0$ for any $t\geq0$. This in particular implies that $\lim_{N\to\infty}P^N_t(y,\cdot)=P_t(y,\cdot)$ weakly in $\cP(\cH)$ for any $t\geq0$. This completes the proof.
\end{proof}

Provided that $b:\bR^n\to\bR^n$ and $\sigma:\bR^n\to\bR^{n\times d}$ satisfy the conditions in \cref{assum_coupling}, the truncated maps $b^N:\bR^n\to\bR^n$ and $\sigma^N:\bR^n\to\bR^{n\times d}$ in \cref{lemm_bounded} satisfy the same conditions with the same constants $C_{b,\Lip}$, $C_{\sigma,\Lip}$ and $C_\UE$. Furthermore, for each admissible weight function $\Phi$, the map $\bW_{d_\Phi}:\cP(\cH)\times\cP(\cH)\to[0,1]$ is a metric on $\cP(\cH)$ that generates the topology of weak convergence. Thus, in proving \cref{theo_coupling}, we may assume without loss of generality that the coefficients $b$ and $\sigma$ are bounded. The key is to construct an admissible weight function $\Phi$ that depends only on the lifting basis $(\mu,M_b,M_\sigma)$ and the constants $C_{b,\Lip}$, $C_{\sigma,\Lip}$ and $C_\UE$, while remaining independent of the specific bounds of $b$ and $\sigma$.

The following lemma constitutes the most crucial step in the proof of \cref{theo_coupling}. In what follows, for two probability measures $\nu_1,\nu_2$ on a measurable space $(E,\cE)$ such that $\nu_1$ is absolutely continuous with respect to $\nu_2$, $D_\KL(\nu_1\|\nu_2):=\int_E\log\frac{\diff\nu_1}{\diff\nu_2}\,\diff\nu_1$ denotes the Kullback--Leibler divergence, also called relative entropy.

%% Lemma

\begin{lemm}\label{lemm_coupling}
Assume that a lifting basis $(\mu,M_b,M_\sigma)$ and measurable maps $b:\bR^n\to\bR^n$, $\sigma:\bR^n\to\bR^{n\times d}$ satisfy \cref{assum_coupling}. Suppose further that the map $\sigma:\bR^n\to\bR^{n\times d}$ is bounded. Then, there exist an admissible weight function $\Phi:[0,\infty)\to\bR^{n\times n}$, a family $\{(Y^{y_1,y_2}_t,\widehat{Y}^{y_1,y_2}_t)\}_{t\geq0,\,y_1,y_2\in\cH}$ of $\cH\times\cH$-valued random variables $(Y^{y_1,y_2}_t,\widehat{Y}^{y_1,y_2}_t)$ on a probability space $(\Omega,\cF,\bP)$, and a family $\{\widehat{\bP}^{y_1,y_2}_t\}_{t\geq0,\,y_1,y_2\in\cH}$ of probability measures on $(\Omega,\cF)$ with $\widehat{\bP}^{y_1,y_2}_t\sim\bP$, such that the following properties hold for any $t\geq0$ and $y_1,y_2\in\cH$:
\begin{align}
	\label{eq_coupling-1}&\Law_\bP(Y^{y_1,y_2}_t)=P_t(y_1,\cdot),\\
	\label{eq_coupling-2}&\Law_{\widehat{\bP}^{y_1,y_2}_t}(\widehat{Y}^{y_1,y_2}_t)=P_t(y_2,\cdot),\\
	\label{eq_coupling-3}&\bE_\bP\left[\big\|Y^{y_1,y_2}_t-\widehat{Y}^{y_1,y_2}_t\big\|_\Phi\right]\leq e^{-\kappa t/2}\|y_1-y_2\|_\Phi,\\
	\label{eq_coupling-4}&D_\KL\big(\bP\big\|\widehat{\bP}^{y_1,y_2}_t\big)\leq\frac{1}{2}\|y_1-y_2\|_\Phi^2.
\end{align}
The admissible weight function $\Phi$ can be chosen to depend only on the lifting basis $(\mu,M_b,M_\sigma)$ and the constants $C_{b,\Lip}$, $C_{\sigma,\Lip}$, and $C_\UE$ in \cref{assum_coupling}.
\end{lemm}

%% Remark

\begin{rem}
The joint distribution on $\cH\times\cH$ of the pair of random variables $(Y^{y_1,y_2}_t,\widehat{Y}^{y_1,y_2}_t)$ under $\bP$ can be regarded as a generalized coupling between two probability measures $P_t(y_1,\cdot)\in\cP(\cH)$ and $P_t(y_2,\cdot)\in\cP(\cH)$ in the sense that, although the first marginal $\Law_\bP(Y^{y_1,y_2}_t)$ coincides with $P_t(y_1,\cdot)$, the second marginal $\Law_\bP(\widehat{Y}^{y_1,y_2}_t)$ does not necessarily coincide with $P_t(y_2,\cdot)$. Instead, $\widehat{Y}^{y_1,y_2}_t$ should satisfy the constraints \eqref{eq_coupling-2}, \eqref{eq_coupling-3} and \eqref{eq_coupling-4}. We define $\widehat{Y}^{y_1,y_2}_t$ as the solution to a \emph{controlled SEE} (see \eqref{eq_controlled-SEE} below), which involves a (finite-dimensional) control process in the drift term. Recall that the SEE \eqref{eq_SEE} is a highly degenerate infinite-dimensional system. Hence, in order to meet the constraints \eqref{eq_coupling-2}, \eqref{eq_coupling-3} and \eqref{eq_coupling-4}, we have to control the dynamics in the infinite-dimensional space $\cH$ by a finite-dimensional control process. This is a significant challenge for a general infinite-dimensional model with degenerate noise, and a model-dependent analysis is required. Our construction of a generalized coupling is based on the specific structure of the SEE \eqref{eq_SEE}, where a careful treatment of the integral operator $\mu[\cdot]\in L(\cV;\bR^n)$ (which is not continuous in $\cH$) is essential. In order to manage the influence of $\mu[\cdot]$, we perform a ``change-of-norm'' technique based on \cref{defi_Phi}. A key idea is to construct a suitable admissible weight function $\Phi_\fa$ depending on a control parameter $\fa$ such that the operator $\mu_{\sigma,\Phi_\fa}[\cdot]$ (see \cref{defi_Phi}) approximates $\mu[\cdot]$ in some sense.
\end{rem}

%% Proof

\begin{proof}[Proof of \cref{lemm_coupling}]
Recall that $\kappa:=\inf\supp>0$ by \cref{assum_coupling}. Fix a tuple $\fa=(m,\delta,L,R)$ of constants $m,\delta,L,R>0$, which will be determined later. Define $\Phi_\fa:[0,\infty)\to\bR^{n\times n}$ by
\begin{align*}
	\Phi_\fa(\theta)&:=M_\sigma(\theta)^{-1}\1_{[\kappa,m)\cap A_L\cap B_R}(\theta)+\left(\delta\theta^{1/2}I_{n\times n}+M_\sigma(\theta)\right)^{-1}\1_{([\kappa,m)\setminus A_L)\cap B_R}(\theta)\\
	&\hspace{1cm}+I_{n\times n}\1_{[\kappa,m)\setminus B_R}(\theta)+m^{1/2}\theta^{-1/2}I_{n\times n}\1_{[m,\infty)}(\theta)
\end{align*}
for $\theta\in[0,\infty)$, where the sets $A_L,B_R\in\cB([0,\infty))$ are defined by
\begin{align*}
	&A_L:=\left\{\theta\in[0,\infty)\relmiddle|\text{$M_\sigma(\theta)\in\bR^{n\times n}$ is symmetric, positive definite and satisfies $\big|M_\sigma(\theta)^{-1}\big|_\op\leq L$}\right\},\\
	&B_R:=\left\{\theta\in[0,\infty)\relmiddle|\text{$M_\sigma(\theta)\in\bR^{n\times n}$ is symmetric, positive definite and satisfies $\big|M_\sigma(\theta)\big|_\op\leq R$}\right\}.
\end{align*}
By construction, we see that $\Phi_\fa(\theta)\in\bR^{n\times n}$ is symmetric, positive definite and satisfies
\begin{align*}
	&\big|\Phi_\fa(\theta)\big|_\op\leq\max\{L,\delta^{-1}\kappa^{-1/2},1\}\1_{[\kappa,m)}(\theta)+m^{1/2}\theta^{-1/2}\1_{[m,\infty)}(\theta),\\
	&\big|\Phi_\fa(\theta)^{-1}\big|_\op\leq\max\{\delta m^{1/2}+R,1\}\1_{[\kappa,m)}(\theta)+m^{-1/2}\theta^{1/2}\1_{[m,\infty)}(\theta),
\end{align*}
for any $\theta\geq\kappa:=\inf\supp>0$. Hence, \eqref{eq_weight-bound} holds for $\mu$-a.e.\ $\theta\in[0,\infty)$ for some constant $C_{\Phi_\fa}>0$ (which may depend on $\fa=(m,\delta,L,R)$), and thus $\Phi_\fa$ is an admissible weight function.

Let $W$ be a $d$-dimensional Brownian motion defined on a complete probability space $(\Omega,\cF,\bP)$, with respect to a filtration $\bF=(\cF_t)_{t\geq0}$ satisfying the usual conditions. In this proof, we denote by $\bE[\cdot]$ the expectation under the probability measure $\bP$. Fix $y_1,y_2\in\cH$. Let $Y=Y^{y_1,y_2}=Y^{y_1}$ be the solution to the SEE \eqref{eq_SEE} with the initial condition $Y_0=y_1$. Clearly, the relation \eqref{eq_coupling-1} holds. Now we introduce the following \emph{controlled SEE} defined on the same probability space $(\Omega,\cF,\bP)$ as $Y$:
\begin{equation}\label{eq_controlled-SEE}
	\begin{dcases}
	\diff\widehat{Y}_t(\theta)=-\theta\widehat{Y}_t(\theta)\,\diff t+\Big\{M_b(\theta)b(\mu[\widehat{Y}_t])+\lambda_\fa M_\sigma(\theta)\mu_{\sigma,\Phi_\fa}\big[Y_t-\widehat{Y}_t\big]\Big\}\,\diff t+M_\sigma(\theta)\sigma(\mu[\widehat{Y}_t])\,\diff W_t\\
	\hspace{6cm}t\geq0,\ \theta\in[0,\infty),\\
	\widehat{Y}_0(\theta)=y_2(\theta),\ \ \theta\in[0,\infty),
	\end{dcases}
\end{equation}
where $\lambda_\fa>0$ is a constant depending on $\fa=(m,\delta,L,R)$, which will be also determined later. Here, the operator $\mu_{\sigma,\Phi_\fa}[\cdot]\in L(\cH;\bR^n)$ is defined as in \cref{defi_Phi}. Analogously to the proof of \cref{prop_SEE-well-posed}, applying the general result on monotone SPDEs \cite[Theorem 4.2.4]{LiRo15} to the controlled SEE \eqref{eq_controlled-SEE}, we see that there exists a unique solution $\widehat{Y}=\widehat{Y}^{y_1,y_2}$ of \eqref{eq_controlled-SEE} (defined in the same manner as in \cref{defi_SEE-solution}). Furthermore, for any $T>0$, we have
\begin{equation}\label{eq_controlled-SEE-apriori}
	\bE\left[\sup_{t\in[0,T]}\Big\{\big\|Y_t\big\|_\cH^2+\big\|\widehat{Y}_t\big\|_\cH^2\Big\}+\int^T_0\Big\{\big\|Y_t\big\|_\cV^2+\big\|\widehat{Y}_t\big\|_\cV^2\Big\}\,\diff t\right]<\infty.
\end{equation}

Applying It\^{o}'s formula for the squared norm $\|\cdot\|_{\Phi_\fa}^2$ in \cref{lemm_Ito} to the $\cH$-valued continuous adapted process $Y-\widehat{Y}$ shows that
\begin{align*}
	\big\|Y_t-\widehat{Y}_t\big\|_{\Phi_\fa}^2&=\big\|y_1-y_2\big\|_{\Phi_\fa}^2+2\int^t_0\Big\langle Y_s-\widehat{Y}_s,\cM_\sigma\big(\sigma(\mu[Y_s])-\sigma(\mu[\widehat{Y}_s])\big)\,\diff W_s\Big\rangle_{\Phi_\fa}\\
	&\hspace{0.5cm}+\int^t_0\Big\{2\Big\langle\cA\big(Y_s-\widehat{Y}_s\big)+\cM_b\big(b(\mu[Y_s])-b(\mu[\widehat{Y}_s])\big)-\lambda_\fa\cM_\sigma\mu_{\sigma,\Phi_\fa}\big[Y_s-\widehat{Y}_s\big],Y_s-\widehat{Y}_s\Big\rangle_{\cV^*,\cV,\Phi_\fa}\\
	&\hspace{4cm}+\Big\|\cM_\sigma\big(\sigma(\mu[Y_s])-\sigma(\mu[\widehat{Y}_s])\big)\Big\|_{L_2(\bR^d;(\cH,\langle\cdot,\cdot\rangle_{\Phi_\fa}))}^2\Big\}\,\diff s
\end{align*}
for any $t\geq0$ a.s. Recalling \cref{defi_Phi} and the formulas \eqref{eq_A-Phi}, \eqref{eq_mu-b-Phi}, \eqref{eq_mu-sigma-Phi} and \eqref{eq_Q-Phi}, the above can be rewritten as
\begin{align*}
	&\big\|Y_t-\widehat{Y}_t\big\|_{\Phi_\fa}^2+2\int^t_0\vnorm[\big]{Y_s-\widehat{Y}_s}_{\Phi_\fa}^2\,\diff s+2\lambda_\fa\int^t_0\big|\mu_{\sigma,\Phi_\fa}\big[Y_s-\widehat{Y}_s\big]\big|^2\,\diff s\\
	&=\big\|y_1-y_2\big\|_{\Phi_\fa}^2+2\int^t_0\Big\langle\mu_{\sigma,\Phi_\fa}\big[Y_s-\widehat{Y}_s\big],\big(\sigma(\mu[Y_s])-\sigma(\mu[\widehat{Y}_s])\big)\,\diff W_s\Big\rangle\\
	&\hspace{0.3cm}+\int^t_0\Big\{2\Big\langle b(\mu[Y_s])-b(\mu[\widehat{Y}_s]),\mu_{b,\Phi_\fa}\big[Y_s-\widehat{Y}_s\big]\Big\rangle+\tr\Big[\big(\sigma(\mu[Y_s])-\sigma(\mu[\widehat{Y}_s])\big)^\top Q_{\sigma,\Phi_\fa}\big(\sigma(\mu[Y_s])-\sigma(\mu[\widehat{Y}_s])\big)\Big]\Big\}\,\diff s
\end{align*}
for any $t\geq0$ a.s. Applying the usual It\^{o}'s formula to the product of the function $e^{\kappa t}$ and the real semimartingale $\|Y_t-\widehat{Y}_t\|_{\Phi_\fa}^2$, we get
\begin{align*}
	&e^{\kappa t}\big\|Y_t-\widehat{Y}_t\big\|_{\Phi_\fa}^2+2\int^t_0e^{\kappa s}\vnorm[\big]{Y_s-\widehat{Y}_s}_{\Phi_\fa}^2\,\diff s+2\lambda_\fa\int^t_0e^{\kappa s}\big|\mu_{\sigma,\Phi_\fa}\big[Y_s-\widehat{Y}_s\big]\big|^2\,\diff s\\
	&=\big\|y_1-y_2\big\|_{\Phi_\fa}^2+2\int^t_0e^{\kappa s}\Big\langle\mu_{\sigma,\Phi_\fa}\big[Y_s-\widehat{Y}_s\big],\big(\sigma(\mu[Y_s])-\sigma(\mu[\widehat{Y}_s])\big)\,\diff W_s\Big\rangle\\
	&\hspace{1cm}+\int^t_0e^{\kappa s}\Big\{\kappa\|Y_s-\widehat{Y}_s\|_{\Phi_\fa}^2+2\Big\langle b(\mu[Y_s])-b(\mu[\widehat{Y}_s]),\mu_{b,\Phi_\fa}\big[Y_s-\widehat{Y}_s\big]\Big\rangle\\
	&\hspace{4cm}+\tr\Big[\big(\sigma(\mu[Y_s])-\sigma(\mu[\widehat{Y}_s])\big)^\top Q_{\sigma,\Phi_\fa}\big(\sigma(\mu[Y_s])-\sigma(\mu[\widehat{Y}_s])\big)\Big]\Big\}\,\diff s
\end{align*}
for any $t\geq0$ a.s. Since $\kappa=\inf\supp$, we have $\kappa\|Y_s-\widehat{Y}_s\|_{\Phi_\fa}^2\leq\vnorm{Y_s-\widehat{Y}_s}_{\Phi_\fa}^2$. By this estimate and the Lipschitz continuity of the maps $b:\bR^n\to\bR^n$ and $\sigma:\bR^n\to\bR^{n\times d}$ (see \cref{assum_coupling}), we obtain
\begin{equation}\label{eq_coupling-estimate-Ito}
\begin{split}
	&e^{\kappa t}\big\|Y_t-\widehat{Y}_t\big\|_{\Phi_\fa}^2+\int^t_0e^{\kappa s}\vnorm[\big]{Y_s-\widehat{Y}_s}_{\Phi_\fa}^2\,\diff s+2\lambda_\fa\int^t_0e^{\kappa s}\big|\mu_{\sigma,\Phi_\fa}\big[Y_s-\widehat{Y}_s\big]\big|^2\,\diff s\\
	&\leq\big\|y_1-y_2\big\|_{\Phi_\fa}^2+2\int^t_0e^{\kappa s}\Big\langle\mu_{\sigma,\Phi_\fa}\big[Y_s-\widehat{Y}_s\big],\big(\sigma(\mu[Y_s])-\sigma(\mu[\widehat{Y}_s])\big)\,\diff W_s\Big\rangle+\int^t_0e^{\kappa s}F_\fa\big(Y_s-\widehat{Y}_s\big)\,\diff s
\end{split}
\end{equation}
for any $t\geq0$ a.s., where the function $F_\fa:\cV\to[0,\infty)$ is defined by
\begin{equation*}
	F_\fa(y):=2C_{b,\Lip}\big|\mu[y]\big|\big|\mu_{b,\Phi_\fa}[y]\big|+C_{\sigma,\Lip}^2\big|Q_{\sigma,\Phi_\fa}\big|_\op\big|\mu[y]\big|^2,\ \ y\in\cV.
\end{equation*}

We provide an estimate of $F_\fa(y)$ in terms of $\vnorm{y}_{\Phi_\fa}$ and $|\mu_{\sigma,\Phi_\fa}[y]|$ for each $y\in\cV$. By the triangle inequality, we have
\begin{equation}\label{eq_coupling-F-Lip}
	F_\fa(y)\leq2C_{b,\Lip}\Big\{\big|\mu_{\sigma,\Phi_\fa}[y]\big|+\big|\overline{\mu}_{\sigma,\Phi_\fa}[y]\big|\Big\}\big|\mu_{b,\Phi_\fa}[y]\big|+2C_{\sigma,\Lip}^2\big|Q_{\sigma,\Phi_\fa}\big|_\op\Big\{\big|\mu_{\sigma,\Phi_\fa}[y]\big|^2+\big|\overline{\mu}_{\sigma,\Phi_\fa}[y]\big|^2\Big\},
\end{equation}
where $\overline{\mu}_{\sigma,\Phi_\fa}[\cdot]\in L(\cV;\bR^n)$ is defined by
\begin{equation*}
	\overline{\mu}_{\sigma,\Phi_\fa}[y]:=\mu[y]-\mu_{\sigma,\Phi_\fa}[y]=\int_{[0,\infty)}\big(I_{n\times n}-M_\sigma(\theta)^\top\Phi_\fa(\theta)\big)y(\theta)\dmu,\ \ y\in\cV.
\end{equation*}
We estimate the three terms $|\overline{\mu}_{\sigma,\Phi_\fa}[y]|$, $|\mu_{b,\Phi_\fa}[y]|$ and $|Q_{\sigma,\Phi_\fa}|_\op$. First, concerning $|\overline{\mu}_{\sigma,\Phi_\fa}[y]|$, observe that
\begin{align*}
	&\left|\left(I_{n\times n}-M_\sigma(\theta)^\top\Phi_\fa(\theta)\right)\Phi_\fa(\theta)^{-1/2}\right|_\op^2\\
	&=\left|\left(I_{n\times n}-M_\sigma(\theta)\left(\delta\theta^{1/2}I_{n\times n}+M_\sigma(\theta)\right)^{-1}\right)\left(\delta\theta^{1/2} I_{n\times n}+M_\sigma(\theta)\right)^{1/2}\right|_\op^2\1_{([\kappa,m)\setminus A_L)\cap B_R}(\theta)\\
	&\hspace{0.5cm}+\big|I_{n\times n}-M_\sigma(\theta)\big|_\op^2\1_{[\kappa,m)\setminus B_R}(\theta)+m^{-1/2}\theta^{1/2}\left|I_{n\times n}-m^{1/2}\theta^{-1/2}M_\sigma(\theta)\right|_\op^2\1_{[m,\infty)}(\theta)\\
	&\leq\delta\theta^{1/2}\1_{[\kappa,\infty)\setminus A_L}(\theta)+\big(1+|M_\sigma(\theta)|_\op\big)^2\1_{[\kappa,\infty)\setminus B_R}(\theta)+m^{-1/2}\theta^{1/2}\big(1+|M_\sigma(\theta)|_\op\big)^2\1_{[m,\infty)}(\theta)
\end{align*}
for any $\theta\geq\kappa:=\inf\supp>0$. From this, together with the Cauchy--Schwarz inequality, we have
\begin{align}
	\nonumber
	&\big|\overline{\mu}_{\sigma,\Phi_\fa}[y]\big|\\
	\nonumber
	&\leq\left\{\int_{[\kappa,\infty)}\theta^{-1}\left|\left(I_{n\times n}-M_\sigma(\theta)^\top\Phi_\fa(\theta)\right)\Phi_\fa(\theta)^{-1/2}\right|_\op^2\dmu\right\}^{1/2}\left\{\int_{[0,\infty)}\theta\left|\Phi_\fa(\theta)^{1/2}y(\theta)\right|^2\dmu\right\}^{1/2}\\
	\label{eq_coupling-F-1}
	&\leq\alpha(m,\delta,L,R)^{1/2}\vnorm{y}_{\Phi_\fa},
\end{align}
where $\alpha(m,\delta,L,R)\geq0$ is given by
\begin{align*}
	\alpha(m,\delta,L,R)&:=\delta\int_{[\kappa,\infty)\setminus A_L}\theta^{-1/2}\dmu+\int_{[\kappa,\infty)\setminus B_R}\theta^{-1}\big(1+|M_\sigma(\theta)|_\op\big)^2\dmu\\
	&\hspace{1cm}+m^{-1/2}\int_{[m,\infty)}\theta^{-1/2}\big(1+|M_\sigma(\theta)|_\op\big)^2\dmu.
\end{align*}
By the integrability conditions \eqref{eq_integrability-mu} and \eqref{eq_integrability-M_sigma} on $\mu$ and $M_\sigma$, combined with the assumption $\kappa:=\inf\supp>0$, the constant $\alpha(m,\delta,L,R)$ is finite. Next, concerning $|\mu_{b,\Phi_\fa}[y]|$, observe that
\begin{align*}
	&\left|M_b(\theta)^\top\Phi_\fa(\theta)^{1/2}\right|_\op^2\\
	&=\big|M_b(\theta)^\top M_\sigma(\theta)^{-1/2}\big|_\op^2\1_{[\kappa,m)\cap A_L\cap B_R}(\theta)+\left|M_b(\theta)^\top\left(\delta\theta^{1/2}I_{n\times n}+M_\sigma(\theta)\right)^{-1/2}\right|_\op^2\1_{([\kappa,m)\setminus A_L)\cap B_R}(\theta)\\
	&\hspace{1cm}+|M_b(\theta)|_\op^2\1_{[\kappa,m)\setminus B_R}(\theta)+m^{1/2}\theta^{-1/2}|M_b(\theta)|_\op^2\1_{[m,\infty)}(\theta)\\
	&\leq\beta(m,\delta,L)\theta^{-1/2}|M_b(\theta)|^2_\op\1_{[\kappa,\infty)}(\theta)
\end{align*}
for any $\theta\geq\kappa:=\inf\supp>0$, where the constant $\beta(m,\delta,L)>0$ is given by
\begin{equation*}
	\beta(m,\delta,L):=\max\Big\{Lm^{1/2},\delta^{-1},m^{1/2}\Big\}.
\end{equation*}
Hence, using the Cauchy--Schwarz inequality, we have
\begin{align}
	\nonumber
	\big|\mu_{b,\Phi_\fa}[y]\big|&\leq\left\{\int_{[\kappa,\infty)}\theta^{-1}\left|M_b(\theta)^\top\Phi_\fa(\theta)^{1/2}\right|_\op^2\dmu\right\}^{1/2}\left\{\int_{[0,\infty)}\theta\left|\Phi_\fa(\theta)^{1/2}y(\theta)\right|^2\dmu\right\}^{1/2}\\
	\label{eq_coupling-F-2}
	&\leq\left\{\beta(m,\delta,L)\int_{[\kappa,\infty)}\theta^{-3/2}|M_b(\theta)|_\op^2\dmu\right\}^{1/2}\vnorm{y}_{\Phi_\fa}.
\end{align}
Similarly, concerning $|Q_{\sigma,\Phi_\fa}|_\op$, we have
\begin{align*}
	&\left|M_\sigma(\theta)^\top\Phi_\fa(\theta)M_\sigma(\theta)\right|_\op\\
	&=|M_\sigma(\theta)|_\op\1_{[\kappa,m)\cap A_L\cap B_R}(\theta)+\left|M_\sigma(\theta)\left(\delta\theta^{1/2}I_{n\times n}+M_\sigma(\theta)\right)^{-1}M_\sigma(\theta)\right|_\op\1_{([\kappa,m)\setminus A_L)\cap B_R}(\theta)\\
	&\hspace{1cm}+|M_\sigma(\theta)|_\op^2\1_{[\kappa,m)\setminus B_R}(\theta)+m^{1/2}\theta^{-1/2}|M_\sigma(\theta)|_\op^2\1_{[m,\infty)}(\theta)\\
	&\leq\beta(m,\delta,L)\theta^{-1/2}|M_\sigma(\theta)|_\op^2\1_{[\kappa,\infty)}(\theta)
\end{align*}
for any $\theta\geq\kappa:=\inf\supp>0$. Hence, we have
\begin{equation}\label{eq_coupling-F-3}
	\big|Q_{\sigma,\Phi_\fa}\big|_\op\leq\beta(m,\delta,L)\int_{[\kappa,\infty)}\theta^{-1/2}|M_\sigma(\theta)|_\op^2\dmu.
\end{equation}
Combining \eqref{eq_coupling-F-Lip} with \eqref{eq_coupling-F-1}, \eqref{eq_coupling-F-2} and \eqref{eq_coupling-F-3}, and using Young's inequality, we get
\begin{equation}\label{eq_coupling-estimate-F}
	F_\fa(y)\leq\left(\ep(m,\delta,L,R)+\frac{1}{2}\right)\vnorm{y}_{\Phi_\fa}^2+C(m,\delta,L)\big|\mu_{\sigma,\Phi_\fa}[y]\big|^2
\end{equation}
for any $y\in\cV$, where the constants $\ep(m,\delta,L,R)\geq0$ and $C(m,\delta,L)>0$ are given by
\begin{align*}
	\ep(m,\delta,L,R)&:=2C_{b,\Lip}\left\{\alpha(m,\delta,L,R)\beta(m,\delta,L)\int_{[\kappa,\infty)}\theta^{-3/2}|M_b(\theta)|_\op^2\dmu\right\}^{1/2}\\
	&\hspace{0.5cm}+2C_{\sigma,\Lip}^2\alpha(m,\delta,L,R)\beta(m,\delta,L)\int_{[\kappa,\infty)}\theta^{-1/2}|M_\sigma(\theta)|_\op^2\dmu
\end{align*}
and
\begin{equation*}
	C(m,\delta,L):=2\beta(m,\delta,L)\left\{C_{b,\Lip}^2\int_{[\kappa,\infty)}\theta^{-3/2}|M_b(\theta)|_\op^2\dmu+C_{\sigma,\Lip}^2\int_{[\kappa,\infty)}\theta^{-1/2}|M_\sigma(\theta)|_\op^2\dmu\right\}.
\end{equation*}
By the integrability conditions \eqref{eq_integrability-M_b} and \eqref{eq_integrability-M_sigma} on $M_b$ and $M_\sigma$, together with $\kappa:=\inf\supp>0$, all the integrals appearing in the expressions of $\ep(m,\delta,L,R)$ and $C(m,\delta,L)$ are finite.

The constant $\ep(m,\delta,L,R)\geq0$ can be arbitrarily small by choosing suitable parameters $m,\delta,L,R>0$. Indeed, observe that
\begin{align*}
	\alpha(m,\delta,L,R)\beta(m,\delta,L)&=\delta\max\Big\{Lm^{1/2},\delta^{-1},m^{1/2}\Big\}\int_{[\kappa,\infty)\setminus A_L}\theta^{-1/2}\dmu\\
	&\hspace{0.5cm}+\max\Big\{Lm^{1/2},\delta^{-1},m^{1/2}\Big\}\int_{[\kappa,\infty)\setminus B_R}\theta^{-1}\big(1+|M_\sigma(\theta)|_\op\big)^2\dmu\\
	&\hspace{0.5cm}+m^{-1/2}\max\Big\{Lm^{1/2},\delta^{-1},m^{1/2}\Big\}\int_{[m,\infty)}\theta^{-1/2}\big(1+|M_\sigma(\theta)|_\op\big)^2\dmu.
\end{align*}
Note that $\mu([\kappa,\infty)\setminus(\bigcup_{L>0}A_L))=0$ and $\mu([\kappa,\infty)\setminus(\bigcup_{R>0}B_R))=0$ by \cref{assum_coupling}. Recalling again the integrability conditions \eqref{eq_integrability-mu} and \eqref{eq_integrability-M_sigma} on $\mu$ and $M_\sigma$, as well as the assumption $\kappa:=\inf\supp>0$, the dominated convergence theorem yields
\begin{align*}
	&\lim_{L\to\infty}\int_{[\kappa,\infty)\setminus A_L}\theta^{-1/2}\dmu=0,\\
	&\lim_{R\to\infty}\int_{[\kappa,\infty)\setminus B_R}\theta^{-1}\big(1+|M_\sigma(\theta)|_\op\big)^2\dmu=0,\ \ \text{and}\\
	&\lim_{m\to\infty}\int_{[m,\infty)}\theta^{-1/2}\big(1+|M_\sigma(\theta)|_\op\big)^2\dmu=0.
\end{align*}
Setting $\delta=\delta_m>0$ and $L=L_m>0$ according to $m>0$ by
\begin{align*}
	&\delta_m:=m^{-1/2}\left\{\int_{[m,\infty)}\theta^{-1/2}\big(1+|M_\sigma(\theta)|_\op\big)^2\dmu+\frac{1}{m}\right\}^{1/2}\ \ \text{and}\\
	&L_m:=\left\{\int_{[m,\infty)}\theta^{-1/2}\big(1+|M_\sigma(\theta)|_\op\big)^2\dmu+\frac{1}{m}\right\}^{-1/2},
\end{align*}
we see that
\begin{equation*}
	\lim_{m\to\infty}\lim_{R\to\infty}\alpha(m,\delta_m,L_m,R)\beta(m,\delta_m,L_m)=0.
\end{equation*}
Hence, we have
\begin{equation*}
	\lim_{m\to\infty}\lim_{R\to\infty}\ep(m,\delta_m,L_m,R)=0.
\end{equation*}
Therefore, there exist constants $m,\delta,L,R>0$ such that $\ep(m,\delta,L,R)\leq\frac{1}{2}$. Fix such a set of parameters $\fa=(m,\delta,L,R)$, and set $\lambda_\fa=C(m,\delta,L)>0$. Then, \eqref{eq_coupling-estimate-Ito} and \eqref{eq_coupling-estimate-F} yield
\begin{equation}\label{eq_coupling-estimate-key}
\begin{split}
	&e^{\kappa t}\big\|Y_t-\widehat{Y}_t\big\|_{\Phi_\fa}^2+\lambda_\fa\int^t_0e^{\kappa s}\big|\mu_{\sigma,\Phi_\fa}\big[Y_s-\widehat{Y}_s\big]\big|^2\,\diff s\\
	&\leq\big\|y_1-y_2\big\|_{\Phi_\fa}^2+2\int^t_0e^{\kappa s}\Big\langle\mu_{\sigma,\Phi_\fa}\big[Y_s-\widehat{Y}_s\big],\big(\sigma(\mu[Y_s])-\sigma(\mu[\widehat{Y}_s])\big)\,\diff W_s\Big\rangle
\end{split}
\end{equation}
for any $t\geq0$ a.s. We stress that the constant $\lambda_\fa>0$ and the admissible weight function $\Phi_\fa$ depend only on the lifting basis $(\mu,M_b,M_\sigma)$ and the constants $C_{b,\Lip},C_{\sigma,\Lip}>0$ in \cref{assum_coupling}. By virtue of the a priori estimate \eqref{eq_controlled-SEE-apriori}, along with the facts that $\mu_{\sigma,\Phi_\fa}[\cdot]\in L(\cH;\bR^n)$ and $\mu[\cdot]\in L(\cV;\bR^n)$, and the Lipschitz continuity of $\sigma:\bR^n\to\bR^{n\times d}$, the stochastic integral on the right-hand side of \eqref{eq_coupling-estimate-key} is a martingale under the probability measure $\bP$. Hence, taking the expectations on both sides of \eqref{eq_coupling-estimate-key} with respect to $\bP$, we obtain
\begin{equation}\label{eq_coupling-estimate-key-expectation}
	\bE\left[e^{\kappa t}\big\|Y_t-\widehat{Y}_t\big\|_{\Phi_\fa}^2+\lambda_\fa\int^t_0e^{\kappa s}\big|\mu_{\sigma,\Phi_\fa}\big[Y_s-\widehat{Y}_s\big]\big|^2\,\diff s\right]\leq\big\|y_1-y_2\big\|_{\Phi_\fa}^2
\end{equation}
for any $t\geq0$.

For each $x\in\bR^n$, let $\sigma(x)^\dagger\in\bR^{d\times n}$ be the pseudo-inverse of the matrix $\sigma(x)\in\bR^{n\times d}$. By \cref{assum_coupling}, we have $\sigma(x)^\dagger=\sigma(x)^\top(\sigma(x)\sigma(x)^\top)^{-1}$, $\sigma(x)\sigma(x)^\dagger=I_{n\times n}$ and $|\sigma(x)^\dagger|_\op\leq C_\UE^{1/2}$ for any $x\in\bR^n$. Define processes $u=u^{y_1,y_2}$ and $\sE=\sE^{y_1,y_2}$ by
\begin{equation*}
	u_t:=\lambda_\fa\sigma(\mu[\widehat{Y}_t])^\dagger\mu_{\sigma,\Phi_\fa}\big[Y_t-\widehat{Y}_t\big],\ \ t\geq0,
\end{equation*}
and
\begin{equation*}
	\sE_t:=\exp\left(-\int^t_0\langle u_s,\diff W_s\rangle-\frac{1}{2}\int^t_0|u_s|^2\,\diff s\right),\ \ t\geq0.
\end{equation*}
Note that $u$ is an $\bR^d$-valued progressively measurable process such that $\int^T_0|u_t|^2\,\diff t<\infty$ $\bP$-a.s.\ for any $T>0$, and $\sE$ is a positive local martingale under $\bP$. Now we show that the stochastic exponential $\sE$ is a martingale under the probability measure $\bP$. To do so, it suffices to show that $\sE$ belongs to the class (DL), which means that the family of random variables $\{\sE_\tau\,|\,\tau\in\cT_T\}$ is uniformly integrable with respect to $\bP$ for any $T>0$, where $\cT_T$ denotes the set of all stopping times $\tau$ such that $\tau\leq T$ a.s. Let $T>0$ and $\tau\in\cT_T$ be fixed. For each $N\in\bN$, define
\begin{equation*}
	\tau_N:=\tau\wedge\inf\left\{t\in[0,T]\relmiddle|\int^t_0\big|\mu_{\sigma,\Phi_\fa}\big[Y_s-\widehat{Y}_s\big]\big|^2\,\diff s\geq N\right\}.
\end{equation*}
Since $u\1_{[0,\tau_N]}$ satisfies Novikov's condition, the stopped process $\sE_{\cdot\wedge\tau_N}$ is a $\bP$-martingale, and hence the measure $\widehat{\bP}_{\tau_N}\sim\bP$ on $(\Omega,\cF)$ defined by $\frac{\diff\widehat{\bP}_{\tau_N}}{\diff\bP}:=\sE_{\tau_N}$ is a probability measure. By Girsanov's theorem, the process
\begin{equation*}
	\widehat{W}^{\tau_N}_t:=W_t+\int^{t\wedge\tau_N}_0u_s\,\diff s,\ \ t\in[0,T],
\end{equation*}
is a Brownian motion relative to $(\cF_t)_{t\in[0,T]}$ under the probability measure $\widehat{\bP}_{\tau_N}$. By \eqref{eq_coupling-estimate-key}, we have
\begin{align*}
	&e^{\kappa(t\wedge\tau_N)}\big\|Y_{t\wedge\tau_N}-\widehat{Y}_{t\wedge\tau_N}\big\|_{\Phi_\fa}^2+\lambda_\fa\int^{t\wedge\tau_N}_0e^{\kappa s}\big|\mu_{\sigma,\Phi_\fa}\big[Y_s-\widehat{Y}_s\big]\big|^2\,\diff s\\
	&\leq\big\|y_1-y_2\big\|_{\Phi_\fa}^2+2\int^{t\wedge\tau_N}_0e^{\kappa s}\Big\langle\mu_{\sigma,\Phi_\fa}\big[Y_s-\widehat{Y}_s\big],\big(\sigma(\mu[Y_s])-\sigma(\mu[\widehat{Y}_s])\big)\,\diff\widehat{W}^{\tau_N}_s\Big\rangle\\
	&\hspace{1cm}-2\lambda_\fa\int^{t\wedge\tau_N}_0e^{\kappa s}\Big\langle\mu_{\sigma,\Phi_\fa}\big[Y_s-\widehat{Y}_s\big],\big(\sigma(\mu[Y_s])-\sigma(\mu[\widehat{Y}_s])\big)\sigma(\mu[\widehat{Y}_s])^\dagger\mu_{\sigma,\Phi_\fa}\big[Y_s-\widehat{Y}_s\big]\Big\rangle\,\diff s
\end{align*}
for any $t\in[0,T]$ $\widehat{\bP}_{\tau_N}$-a.s. Noting that $\|\sigma\|_{\op,\infty}:=\sup_{x\in\bR^n}|\sigma(x)|_\op<\infty$ by the (temporal) assumption in this lemma, we see that the stopped stochastic integral appearing in the right-hand side above is a martingale under the probability measure $\widehat{\bP}_{\tau_N}$. Furthermore, we have
\begin{align*}
	&\left|\Big\langle\mu_{\sigma,\Phi_\fa}\big[Y_s-\widehat{Y}_s\big],\big(\sigma(\mu[Y_s])-\sigma(\mu[\widehat{Y}_s])\big)\sigma(\mu[\widehat{Y}_s])^\dagger\mu_{\sigma,\Phi_\fa}\big[Y_s-\widehat{Y}_s\big]\Big\rangle\right|\\
	&\leq2\|\sigma\|_{\op,\infty}C_\UE^{1/2}\big|\mu_{\sigma,\Phi_\fa}\big[Y_s-\widehat{Y}_s\big]\big|^2\\
	&\leq2\|\sigma\|_{\op,\infty}C_\UE^{1/2}\int_{[0,\infty)}\big|M_\sigma(\theta)^\top\Phi_\fa(\theta)^{1/2}\big|_\op^2\dmu\big\|Y_s-\widehat{Y}_s\big\|^2_{\Phi_\fa},
\end{align*}
where we used the Cauchy--Schwarz inequality in the last line. Note that $\int_{[0,\infty)}|M_\sigma(\theta)^\top\Phi_\fa(\theta)^{1/2}|_\op^2\dmu<\infty$, which follows from the integrability condition \eqref{eq_integrability-M_sigma} on $M_\sigma$ and the bound \eqref{eq_weight-bound} for the admissible weight function $\Phi_\fa$. Hence, denoting by $\widehat{\bE}_{\tau_N}[\cdot]$ the expectation under $\widehat{\bP}_{\tau_N}$, we have
\begin{align*}
	&\widehat{\bE}_{\tau_N}\left[e^{\kappa(t\wedge\tau_N)}\big\|Y_{t\wedge\tau_N}-\widehat{Y}_{t\wedge\tau_N}\big\|_{\Phi_\fa}^2+\lambda_\fa\int^{t\wedge\tau_N}_0e^{\kappa s}\big|\mu_{\sigma,\Phi_\fa}\big[Y_s-\widehat{Y}_s\big]\big|^2\,\diff s\right]\\
	&\leq\big\|y_1-y_2\big\|_{\Phi_\fa}^2+C\int^t_0\widehat{\bE}_{\tau_N}\Big[e^{\kappa(s\wedge\tau_N)}\big\|Y_{s\wedge\tau_N}-\widehat{Y}_{s\wedge\tau_N}\big\|_{\Phi_\fa}^2\Big]\,\diff s
\end{align*}
for any $t\in[0,T]$, where
\begin{equation*}
	C:=4\lambda_\fa\|\sigma\|_{\op,\infty}C_\UE^{1/2}\int_{[0,\infty)}\big|M_\sigma(\theta)^\top\Phi_\fa(\theta)^{1/2}\big|_\op^2\dmu<\infty.
\end{equation*}
By Gronwall's inequality, we get
\begin{equation*}
	\widehat{\bE}_{\tau_N}\left[e^{\kappa \tau_N}\big\|Y_{\tau_N}-\widehat{Y}_{\tau_N}\big\|_{\Phi_\fa}^2+\lambda_\fa\int^{\tau_N}_0e^{\kappa t}\big|\mu_{\sigma,\Phi_\fa}\big[Y_t-\widehat{Y}_t\big]\big|^2\,\diff t\right]\leq e^{CT}\big\|y_1-y_2\big\|_{\Phi_\fa}^2.
\end{equation*}
This estimate yields
\begin{align*}
	\bE\big[\sE_{\tau_N}\log\sE_{\tau_N}\big]&=\widehat{\bE}_{\tau_N}\left[-\int^{\tau_N}_0\big\langle u_t,\diff\widehat{W}^{\tau_N}_t\big\rangle+\frac{1}{2}\int^{\tau_N}_0|u_t|^2\,\diff t\right]\\
	&\leq\frac{C_\UE\lambda_\fa^2}{2}\widehat{\bE}_{\tau_N}\left[\int^{\tau_N}_0\big|\mu_{\sigma,\Phi_\fa}\big[Y_t-\widehat{Y}_t\big]\big|^2\,\diff t\right]\\
	&\leq\frac{C_\UE\lambda_\fa}{2}e^{CT}\big\|y_1-y_2\big\|_{\Phi_\fa}^2.
\end{align*}
Letting $N\to\infty$ and utilizing Fatou's lemma, we obtain
\begin{equation*}
	\bE\big[\sE_\tau\log\sE_\tau\big]\leq\frac{C_\UE\lambda_\fa}{2}e^{CT}\big\|y_1-y_2\big\|_{\Phi_\fa}^2.
\end{equation*}
This bound holds for any $\tau\in\cT_T$, and hence the family $\{\sE_\tau\,|\,\tau\in\cT_T\}$ is uniformly integrable under $\bP$. Since $T>0$ is arbitrary, we see that the local martingale $\sE$ belongs to the class (DL). Thus, $\sE$ is a martingale under the probability measure $\bP$.

For each $t\geq0$, define a measure $\widehat{\bP}_t=\widehat{\bP}^{y_1,y_2}_t\sim\bP$ on $(\Omega,\cF)$ by $\frac{\diff\widehat{\bP}_t}{\diff\bP}:=\sE_t$. The martingale property of the stochastic exponential $\sE$ under $\bP$ implies that $\widehat{\bP}_t$ is a probability measure. Therefore, by Girsanov's theorem, the process
\begin{equation*}
	\widehat{W}_s:=W_s+\int^s_0u_r\,\diff r,\ \ s\in[0,t],
\end{equation*}
is a Brownian motion relative to $(\cF_s)_{s\in[0,t]}$ under $\widehat{\bP}_t$. Note that the process $(\widehat{Y}_s)_{s\in[0,t]}$ solves the following SEE on $(\Omega,\cF_t,\widehat{\bP}_t)$ driven by the Brownian motion $\widehat{W}$:
\begin{equation*}
	\begin{dcases}
	\diff\widehat{Y}_s(\theta)=-\theta\widehat{Y}_s(\theta)\,\diff s+M_b(\theta)b(\mu[\widehat{Y}_s])\,\diff s+M_\sigma(\theta)\sigma(\mu[\widehat{Y}_s])\,\diff\widehat{W}_s,\ \ s\in[0,t],\ \theta\in[0,\infty),\\
	\widehat{Y}_0(\theta)=y_2(\theta),\ \ \theta\in[0,\infty).
	\end{dcases}
\end{equation*}
Hence, by the uniqueness in law for the SEE \eqref{eq_SEE}, which is ensured by the pathwise uniqueness and a Yamada--Watanabe-type result for general SEEs (see, e.g., \cite[Appendix E]{LiRo15}), we obtain $\Law_{\widehat{\bP}_t}(\widehat{Y}_t)=P_t(y_2,\cdot)$, proving \eqref{eq_coupling-2}. Furthermore,
\begin{align*}
	D_\KL\big(\bP\big\|\widehat{\bP}_t\big)&=\bE\left[\log\frac{1}{\sE_t}\right]=\frac{1}{2}\bE\left[\int^t_0|u_s|^2\,\diff s\right]\\
	&\leq\frac{C_\UE\lambda_\fa^2}{2}\bE\left[\int^t_0\big|\mu_{\sigma,\Phi_\fa}\big[Y_s-\widehat{Y}_s\big]\big|^2\,\diff s\right]\\
	&\leq\frac{C_\UE\lambda_\fa}{2}\big\|y_1-y_2\big\|_{\Phi_\fa}^2,
\end{align*}
where the last inequality follows from \eqref{eq_coupling-estimate-key-expectation}. On the other hand, another application of \eqref{eq_coupling-estimate-key-expectation} shows that
\begin{equation*}
	\bE\Big[\big\|Y_t-\widehat{Y}_t\big\|_{\Phi_\fa}\Big]\leq\bE\Big[\big\|Y_t-\widehat{Y}_t\big\|_{\Phi_\fa}^2\Big]^{1/2}\leq e^{-\kappa t/2}\big\|y_1-y_2\big\|_{\Phi_\fa}.
\end{equation*}
Now we define an admissible weight function $\Phi:[0,\infty)\to\bR^{n\times n}$ by $\Phi:=C_\UE\lambda_\fa\Phi_\fa$, which depends only on the lifting basis $(\mu,M_b,M_\sigma)$ and the constants $C_{b,\Lip}$, $C_{\sigma,\Lip}$ and $C_\UE$ in \cref{assum_coupling}. Noting that $C_\UE\lambda_\fa\|\cdot\|_{\Phi_\fa}^2=\|\cdot\|_\Phi^2$, we obtain the required estimates \eqref{eq_coupling-3} and \eqref{eq_coupling-4}. This completes the proof.
\end{proof}

Now we can prove \cref{theo_coupling}.

%% Proof

\begin{proof}[Proof of \cref{theo_coupling}]
By \cref{lemm_bounded} and its subsequent discussion, without loss of generality, we may assume that $\sigma:\bR^n\to\bR^{n\times d}$ is bounded. Let $\Phi$, $\{(Y^{y_1,y_2}_t,\widehat{Y}^{y_1,y_2}_t)\}_{t\geq0,y_1,y_2\in\cH}$ and $\{\widehat{\bP}^{y_1,y_2}_t\}_{t\geq0,y_1,y_2\in\cH}$ be as constructed in \cref{lemm_coupling}. Note that the admissible weight function $\Phi$ depends only on the lifting basis $(\mu,M_b,M_\sigma)$ and the constants $C_{b,\Lip}$, $C_{\sigma,\Lip}$ and $C_\UE$, while remaining independent of the bound of $\sigma$.

For each $t\geq0$ and $y_1,y_2\in\cH$, by \eqref{eq_coupling-2}, we have
\begin{equation}\label{eq_coupling-5}
	d_\TV\big(\Law_\bP\big(\widehat{Y}^{y_1,y_2}_t\big),P_t(y_2,\cdot)\big)=d_\TV\big(\Law_\bP\big(\widehat{Y}^{y_1,y_2}_t\big),\Law_{\widehat{\bP}^{y_1,y_2}_t}\big(\widehat{Y}^{y_1,y_2}_t\big)\big)\leq d_\TV\big(\bP,\widehat{\bP}^{y_1,y_2}_t\big).
\end{equation}
However, the well-known inequalities for the total variation distance (see, e.g., \cite[Lemma 2.5 and Equation (2.25)]{Ts09}) yield
\begin{equation*}
	d_\TV\big(\bP,\widehat{\bP}^{y_1,y_2}_t\big)\leq\sqrt{\frac{1}{2}D_\KL\big(\bP\big\|\widehat{\bP}^{y_1,y_2}_t\big)}
\end{equation*}
and
\begin{equation*}
	d_\TV\big(\bP,\widehat{\bP}^{y_1,y_2}_t\big)\leq1-\frac{1}{2}\exp\left(-D_\KL\big(\bP\big\|\widehat{\bP}^{y_1,y_2}_t\big)\right).
\end{equation*}
Applying these inequalities to \eqref{eq_coupling-5} and utilizing the estimate \eqref{eq_coupling-4}, we get
\begin{equation}\label{eq_coupling-6}
	d_\TV\big(\Law_\bP\big(\widehat{Y}^{y_1,y_2}_t\big),P_t(y_2,\cdot)\big)\leq\frac{1}{2}\|y_1-y_2\|_\Phi
\end{equation}
and
\begin{equation}\label{eq_coupling-7}
	d_\TV\big(\Law_\bP\big(\widehat{Y}^{y_1,y_2}_t\big),P_t(y_2,\cdot)\big)\leq1-\frac{1}{2}\exp\left(-\frac{1}{2}\|y_1-y_2\|_\Phi^2\right)
\end{equation}
for any $t\geq0$ and $y_1,y_2\in\cH$. Assertions (i) and (ii) now follow from \eqref{eq_coupling-1}, \eqref{eq_coupling-3}, \eqref{eq_coupling-6}, and \eqref{eq_coupling-7} by \cite[Theorem 2.4]{BuKuSc20}. For the sake of completeness, we provide the details of their proofs below.

For each $t\geq0$ and $y_1,y_2\in\cH$, by the ``coupling lemma'' (see, e.g., \cite[Theorem 4.1]{Vi09}), there exists an $\cH\times\cH$-valued random variable $(\widehat{\sY}^{y_1,y_2}_t,\sY^{y_1,y_2}_t)$ on a probability space $(\Omega',\cF',\bP')$ such that $\Law_{\bP'}(\widehat{\sY}^{y_1,y_2}_t)=\Law_\bP(\widehat{Y}^{y_1,y_2}_t)$, $\Law_{\bP'}(\sY^{y_1,y_2}_t)=P_t(y_2,\cdot)$, and
\begin{equation*}
	\bP'\big(\widehat{\sY}^{y_1,y_2}_t\neq\sY^{y_1,y_2}_t\big)=d_\TV\big(\Law_\bP\big(\widehat{Y}^{y_1,y_2}_t\big),P_t(y_2,\cdot)\big).
\end{equation*}
Furthermore, by the ``gluing lemma'' (see, e.g., \cite[Lemma 4.3.2]{Ku18}), we can construct an $\cH\times\cH\times\cH$-valued random variable $(\xi^{y_1,y_2}_t,\eta^{y_1,y_2}_t,\zeta^{y_1,y_2}_t)$ on a probability space $(\Omega'',\cF'',\bP'')$ such that $\Law_{\bP''}(\xi^{y_1,y_2}_t,\eta^{y_1,y_2}_t)=\Law_\bP(Y^{y_1,y_2}_t,\widehat{Y}^{y_1,y_2}_t)$ and $\Law_{\bP''}(\eta^{y_1,y_2}_t,\zeta^{y_1,y_2}_t)=\Law_{\bP'}(\widehat{\sY}^{y_1,y_2}_t,\sY^{y_1,y_2}_t)$. Note that
\begin{equation*}
	\Law_{\bP''}(\xi^{y_1,y_2}_t)=\Law_\bP(Y^{y_1,y_2}_t)=P_t(y_1,\cdot)
\end{equation*}
by \eqref{eq_coupling-1}, and
\begin{equation*}
	\Law_{\bP''}(\zeta^{y_1,y_2}_t)=\Law_{\bP'}(\sY^{y_1,y_2}_t)=P_t(y_2,\cdot).
\end{equation*}
Hence, the joint distribution $\Law_{\bP''}(\xi^{y_1,y_2}_t,\zeta^{y_1,y_2}_t)\in\cP(\cH\times\cH)$ is a (true) coupling between the two probability measures $P_t(y_1,\cdot)\in\cP(\cH)$ and $P_t(y_2,\cdot)\in\cP(\cH)$. Furthermore, we have
\begin{equation*}
	\bE_{\bP''}\big[\big\|\xi^{y_1,y_2}_t-\eta^{y_1,y_2}_t\big\|_\Phi\big]=\bE_\bP\big[\big\|Y^{y_1,y_2}_t-\widehat{Y}^{y_1,y_2}_t\big\|_\Phi\big]
\end{equation*}
and
\begin{equation*}
	\bP''\big(\eta^{y_1,y_2}_t\neq\zeta^{y_1,y_2}_t\big)=\bP'\big(\widehat{\sY}^{y_1,y_2}_t\neq\sY^{y_1,y_2}_t\big)=d_\TV\big(\Law_\bP\big(\widehat{Y}^{y_1,y_2}_t\big),P_t(y_2,\cdot)\big).
\end{equation*}
Recalling the definition of the metric $d_\Phi(y_1,y_2):=\|y_1-y_2\|_\Phi\wedge1$, and utilizing the above properties of $(\xi^{y_1,y_2}_t,\eta^{y_1,y_2}_t,\zeta^{y_1,y_2}_t)$, we have
\begin{align*}
	\bW_{d_\Phi}(P_t(y_1,\cdot),P_t(y_2,\cdot))&\leq\bE_{\bP''}\big[d_\Phi(\xi^{y_1,y_2}_t,\zeta^{y_1,y_2}_t)\big]\\
%	&=\bE_{\bP''}\left[d_\Phi(\xi^{y_1,y_2}_t,\zeta^{y_1,y_2}_t)\1_{\{\eta^{y_1,y_2}_t=\zeta^{y_1,y_2}_t\}}\right]+\bE_{\bP''}\left[d_\Phi(\xi^{y_1,y_2}_t,\zeta^{y_1,y_2}_t)\1_{\{\eta^{y_1,y_2}_t\neq\zeta^{y_1,y_2}_t\}}\right]\\
	&\leq\bE_{\bP''}\big[\big\|\xi^{y_1,y_2}_t-\eta^{y_1,y_2}_t\big\|_\Phi\big]+\bP''\big(\eta^{y_1,y_2}_t\neq\zeta^{y_1,y_2}_t\big)\\
	&=\bE_\bP\big[\big\|Y^{y_1,y_2}_t-\widehat{Y}^{y_1,y_2}_t\big\|_\Phi\big]+d_\TV\big(\Law_\bP\big(\widehat{Y}^{y_1,y_2}_t\big),P_t(y_2,\cdot)\big).
\end{align*}
Together with \eqref{eq_coupling-3} and \eqref{eq_coupling-6}, this implies that
\begin{equation}\label{eq_contracting'}
	\bW_{d_\Phi}(P_t(y_1,\cdot),P_t(y_2,\cdot))\leq\left(e^{-\kappa t/2}+\frac{1}{2}\right)\|y_1-y_2\|_\Phi
\end{equation}
for all $t\geq0$ and all $y_1,y_2\in\cH$. Furthermore, using \eqref{eq_coupling-7} instead of \eqref{eq_coupling-6}, we obtain
\begin{equation}\label{eq_small'}
	\bW_{d_\Phi}(P_t(y_1,\cdot),P_t(y_2,\cdot))\leq e^{-\kappa t/2}\|y_1-y_2\|_\Phi+1-\frac{1}{2}\exp\left(-\frac{1}{2}\|y_1-y_2\|_\Phi^2\right)
\end{equation}
for all $t\geq0$ and all $y_1,y_2\in\cH$.

On the one hand, if $t\geq\frac{4\log2}{\kappa}$ and $d_\Phi(y_1,y_2)<1$, then the estimate \eqref{eq_contracting'}, combined with $e^{-\kappa t/2}+\frac{1}{2}\leq\frac{3}{4}$ and $\|y_1-y_2\|_\Phi=d_\Phi(y_1,y_2)$, yields \eqref{eq_contracting}. This shows that the distance function $d_\Phi$ is contracting for $P_t$, thereby establishing assertion (i). On the other hand, for any $R>0$, if $t\geq\frac{4R^2+2\log(8R)}{\kappa}$ and $y_1,y_2\in\overline{B}_{\Phi}(R):=\{y\in\cH\,|\,\|y\|_\Phi\leq R\}$, then the estimate \eqref{eq_small'}, together with $\|y_1-y_2\|_\Phi\leq2R$ and $e^{-\kappa t/2}2R+1-\frac{1}{2}e^{-2R^2}\leq1-\frac{1}{4}e^{-2R^2}$, implies the estimate \eqref{eq_small}. This indicates that the set $\overline{B}_{\Phi}(R)$ is $d_\Phi$-small for $P_t$, proving assertion (ii).

It remains to prove assertion (iii). Let $t\geq0$ and $y_1,y_2\in\cH$ be fixed, and let $f:\cH\to[1,\infty)$ be a bounded Borel measurable function with $\|\nabla_\Phi\log f\|_\infty<\infty$. By \eqref{eq_coupling-1} and the Lipschitz continuity of $\log f$, we have
\begin{equation*}
	P_t\log f(y_1)=\bE_\bP\big[\log f(Y^{y_1,y_2}_t)\big]\leq\bE_\bP\big[\log f\big(\widehat{Y}^{y_1,y_2}_t\big)\big]+\bE_\bP\big[\|Y^{y_1,y_2}_t-\widehat{Y}^{y_1,y_2}_t\|_\Phi\big]\|\nabla_\Phi\log f\|_\infty.
\end{equation*}
Furthermore, by Jensen's inequality and \eqref{eq_coupling-2}, we have
\begin{align*}
	\bE_\bP\big[\log f\big(\widehat{Y}^{y_1,y_2}_t\big)\big]&=\bE_\bP\left[\log\left(\frac{\diff\widehat{\bP}^{y_1,y_2}_t}{\diff\bP}f(\widehat{Y}^{y_1,y_2}_t)\right)\right]+\bE_\bP\left[\log\frac{\diff\bP}{\diff\widehat{\bP}^{y_1,y_2}_t}\right]\\
	&\leq\log\bE_\bP\left[\frac{\diff\widehat{\bP}^{y_1,y_2}_t}{\diff\bP}f\big(\widehat{Y}^{y_1,y_2}_t\big)\right]+\bE_\bP\left[\log\frac{\diff\bP}{\diff\widehat{\bP}^{y_1,y_2}_t}\right]\\
	&=\log P_tf(y_2)+D_\KL\big(\bP\big\|\widehat{\bP}^{y_1,y_2}_t\big).
\end{align*}
Hence, we obtain
\begin{equation*}
	P_t\log f(y_1)\leq\log P_tf(y_2)+D_\KL\big(\bP\big\|\widehat{\bP}^{y_1,y_2}_t\big)+\bE_\bP\big[\|Y_t-\widehat{Y}^{y_1,y_2}_t\|_\Phi\big]\|\nabla_\Phi\log f\|_\infty.
\end{equation*}
Combined with \eqref{eq_coupling-3} and \eqref{eq_coupling-4}, this estimate yields the desired asymptotic log-Harnack inequality \eqref{eq_Harnack} in assertion (iii). Finally, according to \cite[Theorem 2.1]{BaWaYu19}, the asymptotic log-Harnack inequality ensures that an invariant probability measure for $\{P_t\}_{t\geq0}$, if it exists, must be unique. This completes the proof.
\end{proof}

%%%%%%%%%%%%
%% Subsection
%%%%%%%%%%%%

\subsection{Proofs of \cref{theo_Lyapunov} and \cref{prop_Lyapunov-sufficient}: Construction of a Lyapunov function}\label{subsec_proof-Lyapunov}

Our next purpose is to construct a Lyapunov function for the SEE \eqref{eq_SEE}. First, we prove \cref{theo_Lyapunov} under the abstract condition in \cref{assum_Lyapunov-abstract}.

%% Proof

\begin{proof}[Proof of \cref{theo_Lyapunov}]
Let $y\in\cH$ be fixed, and let $(Y,W,\Omega,\cF,\bF,\bP)$ be a weak solution of the SEE \eqref{eq_SEE} with initial condition $y$, the existence and uniqueness in law of which are assumed in this theorem. Note that $P_t(y,\cdot)=\Law_\bP(Y_t)$ for any $t\geq0$. Let $\Psi:[0,\infty)\to\bR^{n\times n}$, $\delta\in(0,1)$, $\rho>0$ and $C_\Lyap>0$ be as in \cref{assum_Lyapunov-abstract}. By applying It\^{o}'s formula for the squared (equivalent) norm $\|\cdot\|_{\Psi}^2$ in \cref{lemm_Ito}, and utilizing the formulas \eqref{eq_A-Phi}, \eqref{eq_mu-b-Phi}, \eqref{eq_mu-sigma-Phi}, and \eqref{eq_Q-Phi}, we obtain
\begin{equation}\label{eq_Lyapunov-Ito-square}
\begin{split}
	\|Y_t\|_{\Psi}^2+2\int^t_0\vnorm{Y_s}_\Psi^2\,\diff s&=\|y\|_{\Psi}^2+2\int^t_0\left\{\big\langle b(\mu[Y_s]),\mu_{b,\Psi}[Y_s]\big\rangle+\frac{1}{2}\tr\big[\sigma(\mu[Y_s])^\top Q_{\sigma,\Psi}\sigma(\mu[Y_s])\big]\right\}\,\diff s\\
	&\hspace{1cm}+2\int^t_0\big\langle\mu_{\sigma,\Psi}[Y_s],\sigma(\mu[Y_s])\,\diff W_s\big\rangle
\end{split}
\end{equation}
for any $t\geq0$ a.s. Hence, by the condition \eqref{eq_coercivity} in \cref{assum_Lyapunov-abstract}, we see that
\begin{equation}\label{eq_Lyapunov-estimate-square}
	\|Y_t\|_{\Psi}^2+2\int^t_0\Big\{\rho\|Y_s\|_{\Psi}^2+(1-\delta)\vnorm{Y_s}_\Psi^2\Big\}\,\diff s\leq\|y\|_{\Psi}^2+2C_\Lyap t+2\int^t_0\big\langle\mu_{\sigma,\Psi}[Y_s],\sigma(\mu[Y_s])\,\diff W_s\big\rangle
\end{equation}
for any $t\geq0$ a.s. We emphasize that $\rho>0$ and $1-\delta>0$ by the assumption. For each $N\in\bN$, define a stopping time $\tau_N$ by
\begin{equation*}
	\tau_N:=\inf\left\{t\geq0\relmiddle|\int^t_0\big|\sigma(\mu[Y_s])^\top\mu_{\sigma,\Psi}[Y_s]\big|^2\,\diff s\geq N\right\}.
\end{equation*}
Note that $\lim_{N\to\infty}\tau_N=\infty$ a.s. Clearly, the stochastic integral on the right-hand side of \eqref{eq_Lyapunov-estimate-square} stopped at $\tau_N$ is a martingale. Hence, by taking the expectations on both sides of \eqref{eq_Lyapunov-estimate-square} with $t$ replaced by $t\wedge\tau_N$, we obtain
\begin{equation*}
	\bE\left[\|Y_{t\wedge\tau_N}\|_{\Psi}^2+2\int^{t\wedge\tau_N}_0\Big\{\rho\|Y_s\|_{\Psi}^2+(1-\delta)\vnorm{Y_s}_\Psi^2\Big\}\,\diff s\right]\leq\|y\|_{\Psi}^2+2C_\Lyap t.
\end{equation*}
Taking the limit $N\to\infty$ and utilizing Fatou's lemma, we obtain
\begin{equation}\label{eq_Lyapunov-estimate-square-key}
	\bE\left[\|Y_t\|_{\Psi}^2+2\int^t_0\Big\{\rho\|Y_s\|_{\Psi}^2+(1-\delta)\vnorm{Y_s}_\Psi^2\Big\}\,\diff s\right]\leq\|y\|_{\Psi}^2+2C_\Lyap t
\end{equation}
for any $t\geq0$. On the other hand, applying It\^{o}'s formula to $e^{2\rho t}\|Y_t\|_{\Psi}^2$ and utilizing \eqref{eq_Lyapunov-Ito-square}, \eqref{eq_coercivity}, along with a stopping time argument similar to the one above, shows that
\begin{equation}\label{eq_Lyapunov-estimate-square-exp}
	\bE\left[e^{2\rho t}\|Y_t\|_{\Psi}^2+2(1-\delta)\int^t_0e^{2\rho s}\vnorm{Y_s}_\Psi^2\,\diff s\right]\leq\|y\|_{\Psi}^2+\frac{C_\Lyap(e^{2\rho t}-1)}{\rho}
\end{equation}
for any $t\geq0$.

The estimate \eqref{eq_Lyapunov-estimate-square-exp} in particular implies that
\begin{equation*}
	\int_\cH\|y'\|_\Psi^2\,P_t(y,\diff y')\leq e^{-2\rho t}\|y\|_\Psi^2+\frac{C_\Lyap}{\rho}
\end{equation*}
for all $t\geq0$ and all $y\in\cH$. Hence, the function $V=\|\cdot\|_\Psi^2$ is a Lyapunov function for $\{P_t\}_{t\geq0}$.

Define a lower semi-continuous function $U:\cH\to[0,\infty]$ by $U(y):=\rho\|y\|_{\Psi}^2+(1-\delta)\vnorm{y}_\Psi^2$ for $y\in\cH$. For each $t>0$, define a probability kernel $\nu_t:\cH\times\cB(\cH)\to[0,1]$ by $\nu_t(y,A):=\frac{1}{t}\int^t_0P_s(y,A)\,\diff s$ for each $y\in\cH$ and $A\in\cB(\cH)$; this probability kernel is well-defined thanks to the measurability assumption for the Markov semigroup $\{P_t\}_{t\geq0}$. Then, by the estimate \eqref{eq_Lyapunov-estimate-square-key}, we have
\begin{equation}\label{eq_Lyapunov-estimate-square-P}
	\int_\cH U(y')\,\nu_t(y,\diff y')\leq\frac{\|y\|_{\Psi}^2}{2t}+C_\Lyap
\end{equation}
for any $y\in\cH$ and $t>0$. Suppose that we are given an invariant probability measure $\pi\in\cP(\cH)$ for $\{P_t\}_{t\geq0}$. Note that $\pi$ is invariant with respect to $\nu_t$ for any $t>0$. Hence, using the estimate \eqref{eq_Lyapunov-estimate-square-P}, for any $N\in\bN$ and $t>0$, we have
\begin{align*}
	\int_\cH\big\{U(y)\wedge N\big\}\,\pi(\diff y)&=\int_\cH\int_\cH\big\{U(y')\wedge N\big\}\,\nu_t(y,\diff y')\,\pi(\diff y)\\
	&\leq\int_\cH\left\{\int_\cH U(y')\,\nu_t(y,\diff y')\,\wedge N\right\}\,\pi(\diff y)\\
	&\leq\int_\cH\left\{\frac{\|y\|_{\Psi}^2}{2t}\wedge N\right\}\,\pi(\diff y)+C_\Lyap.
\end{align*}
The dominated convergence theorem ensures that the integral in the last line above tends to zero as $t\to\infty$ for each fixed $N\in\bN$. Hence, by taking the limit as $t\to\infty$, subsequently letting $N\to\infty$, and applying the monotone convergence theorem, we obtain $\int_\cH U(y)\,\pi(\diff y)\leq C_\Lyap$. This demonstrates that the estimate \eqref{eq_Lyapunov-square-pi} holds. Since $y\mapsto(\rho\|y\|_{\Psi}^2+(1-\delta)\vnorm{y}_\Psi^2)^{1/2}$ defines a norm on $\cV$ equivalent to $\|\cdot\|_\cV$, the estimate \eqref{eq_Lyapunov-square-pi} ensures that $\pi(\cV)=1$ and $\int_\cV\|y\|_\cV^2\,\pi(\diff y)<\infty$. This completes the proof.
\end{proof}

%% Remark

\begin{rem}\label{rem_Lyapunov}
Since the map $y\mapsto(\rho\|y\|_{\Psi}^2+(1-\delta)\vnorm{y}_\Psi^2)^{1/2}$ defines a norm on $\cV$ equivalent to $\|\cdot\|_\cV$, the estimate \eqref{eq_Lyapunov-estimate-square-P}, together with Markov's inequality, implies that
\begin{equation}\label{eq_tight}
	\lim_{R\to\infty}\sup_{t\geq1}\nu_t\big(y,\cH\setminus\overline{B}_\cV(R)\big)=0\ \ \text{for any $y\in\cH$},
\end{equation}
where $\overline{B}_{\cV}(R):=\{y\in\cV\,|\,\|y\|_{\cV}\leq R\}$. By \cref{lemm_app-compact}, if $\supp\cap[0,m]$ is a finite set for each $m>0$, then the embedding $\cV\hookrightarrow\cH$ is compact, and hence the set $\overline{B}_{\cV}(R)$ is a compact subset of $\cH$ for any $R>0$. Therefore, in this case, \eqref{eq_tight} implies that the family of probability measures $\{\nu_t(y,\cdot)\}_{t\geq1}$ is tight on $\cH$ for any $y\in\cH$, and hence (assuming that $\{P_t\}_{t\geq0}$ satisfies the Feller property) the Krylov--Bogoliubov theorem (see, e.g., \cite[Corollary 11.8]{DaPrZa14}) ensures the existence of an invariant probability measure $\pi\in\cP(\cH)$ for the Markov semigroup $\{P_t\}_{t\geq0}$. However, if $\supp\cap[0,m]$ is an infinite set for some $m>0$, which is a typical case we are interested in, then by \cref{lemm_app-compact} the embedding $\cV\hookrightarrow\cH$ is not compact, and hence \eqref{eq_tight} does no longer result in the tightness of $\{\nu_t(y,\cdot)\}_{t\geq1}$. In this case, the estimate \eqref{eq_Lyapunov-estimate-square-P} alone is not sufficient to conclude the existence of an invariant probability measure. See also the discussions in \cref{rem_Lyapunov-compact} and \cref{rem_app-compact}.
\end{rem}

Next, we prove \cref{prop_Lyapunov-sufficient}, which provides a verifiable sufficient condition for \cref{assum_Lyapunov-abstract}. As in the proof of \cref{lemm_coupling}, the ``change-of-norm'' technique plays a crucial role. A key idea is to construct an admissible weight function $\Psi_m$ depending on a parameter $m>0$ such that the operator $\mu_{b,\Psi_m}[\cdot]$ approximates the integral operator $\mu[\cdot]$ appearing in the SEE \eqref{eq_SEE} in some sense. This strategy parallels that used in the proof of \cref{lemm_coupling}, where an admissible weight function $\Phi_\fa$ was constructed to approximate $\mu[\cdot]$ via the operator $\mu_{\sigma,\Phi_\fa}[\cdot]$.

%% Proof

\begin{proof}[Proof of \cref{prop_Lyapunov-sufficient}]
Assume that the lifting basis $(\mu,M_b,M_\sigma)$ and the maps $b:\bR^n\to\bR^n$ and $\sigma:\bR^n\to\bR^{n\times d}$ satisfy the conditions in \cref{prop_Lyapunov-sufficient}. Fix a constant $m>\kappa:=\inf\supp>0$, which will be determined later. Define
\begin{equation*}
	\Psi_m(\theta):=\left(m^{-1}\theta^{1/2}I_{n\times n}+M_b(\theta)\right)^{-1}\1_{A_m}(\theta)+\theta^{-1/2}I_{n\times n}\1_{[\kappa,\infty)\setminus A_m}(\theta),\ \ \theta\in[0,\infty),
\end{equation*}
where
\begin{equation*}
	A_m:=\left\{\theta\in[\kappa,m)\relmiddle|\text{$M_b(\theta)$ is symmetric, nonnegative definite and satisfies $|M_b(\theta)|_\op\leq m$}\right\}.
\end{equation*}
Analogously to the proof of \cref{lemm_coupling}, we can show that $\Psi_m:[0,\infty)\to\bR^{n\times n}$ is an admissible weight function. Furthermore, we define a bounded linear operator $\overline{\mu}_{b,\Psi_m}[\cdot]:\cV\to\bR^n$ by
\begin{equation*}
	\overline{\mu}_{b,\Psi_m}[y]:=\mu[y]-\mu_{b,\Psi_m}[y]=\int_{[0,\infty)}\big(I_{n\times n}-M_b(\theta)^\top\Psi_m(\theta)\big)y(\theta)\dmu
\end{equation*}
for $y\in\cV$.

Let $y\in\cV$ be fixed. By the assumption, we have
\begin{align*}
	&\big\langle b(\mu[y]),\mu_{b,\Psi_m}[y]\big\rangle+\frac{1}{2}\tr\big[\sigma(\mu[y])^\top Q_{\sigma,\Psi_m}\sigma(\mu[y])\big]\\
	&=\big\langle b(\mu[y]),\mu[y]\big\rangle-\big\langle b(\mu[y]),\overline{\mu}_{b,\Psi_m}[y]\big\rangle+\frac{1}{2}\tr\big[\sigma(\mu[y])^\top Q_{\sigma,\Psi_m}\sigma(\mu[y])\big]\\
	&\leq\gamma|\mu[y]|^2+C_{b,\LG}'+C_{b,\LG}\big(1+|\mu[y]|\big)\big|\overline{\mu}_{b,\Psi_m}[y]\big|+\frac{1}{2}C_{\sigma,\subLG}^2|Q_{\sigma,\Psi_m}|_\op\big(1+|\mu[y]|^p\big)^2.
\end{align*}
Since $p\in(0,1)$, Young's inequality yields that, for any $\ep\in(0,1)$,
\begin{equation}\label{eq_Lyapunov-G}
	\big\langle b(\mu[y]),\mu_{b,\Psi_m}[y]\big\rangle+\frac{1}{2}\tr\big[\sigma(\mu[y])^\top Q_{\sigma,\Psi_m}\sigma(\mu[y])\big]\leq(\gamma+\ep)|\mu[y]|^2+\frac{C_{b,\LG}^2}{\ep}\big|\overline{\mu}_{b,\Psi_m}[y]\big|^2+C_{m,\ep},
\end{equation}
where $C_{m,\ep}>0$ is a constant which depends only on $|Q_{\sigma,\Psi_m}|_\op$, $\ep$, $C_{b,\LG}'$, $C_{\sigma,\subLG}$ and $p$. Now we estimate $|\mu[y]|^2$ and $|\overline{\mu}_{b,\Psi_m}[y]|^2$ in terms of $\vnorm{y}_{\Psi_m}^2$. First, recalling the assumption that $\kappa:=\inf\supp>0$, the Cauchy--Schwarz inequality and the definition of the admissible weight function $\Psi_m$ yield
\begin{align}
	\nonumber
	|\mu[y]|^2&\leq\left\{\int_{[\kappa,\infty)}\theta^{-1}\big|\Psi_m(\theta)^{-1/2}\big|_\op^2\dmu\right\}\left\{\int_{[0,\infty)}\theta\big|\Psi_m(\theta)^{1/2}y(\theta)\big|^2\dmu\right\}\\
	\label{eq_Lyapunov-G-1}
	&=\left\{m^{-1}\int_{A_m}\theta^{-1/2}\dmu+\int_{A_m}\theta^{-1}|M_b(\theta)|_\op\dmu+\int_{[\kappa,\infty)\setminus A_m}\theta^{-1/2}\dmu\right\}\vnorm{y}_{\Psi_m}^2.
\end{align}
Second, noting that
\begin{align*}
	&\left|\big(I_{n\times n}-M_b(\theta)^\top\Psi_m(\theta)\big)\Psi_m(\theta)^{-1/2}\right|^2_\op\\
	&=\left|\left(I_{n\times n}-M_b(\theta)\left(m^{-1}\theta^{1/2}I_{n\times n}+M_b(\theta)\right)^{-1}\right)\left(m^{-1}\theta^{1/2}I_{n\times n}+M_b(\theta)\right)^{1/2}\right|_\op^2\1_{A_m}(\theta)\\
	&\hspace{1cm}+\theta^{1/2}\big|I_{n\times n}-\theta^{-1/2}M_b(\theta)\big|_\op^2\1_{[\kappa,\infty)\setminus A_m}(\theta)\\
	&\leq m^{-1}\theta^{1/2}\1_{A_m}(\theta)+2\theta^{1/2}\1_{[\kappa,\infty)\setminus A_m}(\theta)+2\theta^{-1/2}|M_b(\theta)|_\op^2\1_{[\kappa,\infty)\setminus A_m}(\theta)
\end{align*}
for any $\theta\geq\kappa:=\inf\supp>0$, we have
\begin{align}
	\nonumber
	&\big|\overline{\mu}_{b,\Psi_m}[y]\big|^2\\
	\nonumber
	&\leq\left\{\int_{[\kappa,\infty)}\theta^{-1}\left|\big(I_{n\times n}-M_b(\theta)^\top\Psi_m(\theta)\big)\Psi_m(\theta)^{-1/2}\right|^2_\op\dmu\right\}\left\{\int_{[0,\infty)}\theta\big|\Psi_m(\theta)^{1/2}y(\theta)\big|^2\dmu\right\}\\
	\label{eq_Lyapunov-G-2}
	&\leq\left\{m^{-1}\int_{A_m}\theta^{-1/2}\dmu+2\int_{[\kappa,\infty)\setminus A_m}\theta^{-1/2}\dmu+2\int_{[\kappa,\infty)\setminus A_m}\theta^{-3/2}|M_b(\theta)|_\op^2\dmu\right\}\vnorm{y}_{\Psi_m}^2.
\end{align}
Combining \eqref{eq_Lyapunov-G}, \eqref{eq_Lyapunov-G-1} and \eqref{eq_Lyapunov-G-2}, we get
\begin{equation*}
	\big\langle b(\mu[y]),\mu_{b,\Psi_m}[y]\big\rangle+\frac{1}{2}\tr\big[\sigma(\mu[y])^\top Q_{\sigma,\Psi_m}\sigma(\mu[y])\big]\leq\delta_{m,\ep}\vnorm{y}_{\Psi_m}^2+C_{m,\ep},
\end{equation*}
where the constant $\delta_{m,\ep}>0$ is defined by
\begin{align*}
	\delta_{m,\ep}&:=(\gamma+\ep)\left\{m^{-1}\int_{A_m}\theta^{-1/2}\dmu+\int_{A_m}\theta^{-1}|M_b(\theta)|_\op\dmu+\int_{[\kappa,\infty)\setminus A_m}\theta^{-1/2}\dmu\right\}\\
	&\hspace{1cm}+\frac{C_{b,\LG}^2}{\ep}\left\{m^{-1}\int_{A_m}\theta^{-1/2}\dmu+2\int_{[\kappa,\infty)\setminus A_m}\theta^{-1/2}\dmu+2\int_{[\kappa,\infty)\setminus A_m}\theta^{-3/2}|M_b(\theta)|^2\dmu\right\}.
\end{align*}
By the integrability conditions \eqref{eq_integrability-mu} and \eqref{eq_integrability-M_b} on $\mu$ and $M_b$, together with the assumption that $\kappa:=\inf\supp>0$, all the integrals appearing in the definition of $\delta_{m,\ep}$ are finite. Furthermore, we have $\gamma\int_{[\kappa,\infty)}\theta^{-1}|M_b(\theta)|_\op\dmu<1$ and $\mu([\kappa,\infty)\setminus\bigcup_{m>\kappa}A_m)=0$ by the assumption. Hence, the dominated convergence theorem yields
\begin{equation*}
	\lim_{\ep\to0}\lim_{m\to\infty}\delta_{m,\ep}=\gamma\int_{[\kappa,\infty)}\theta^{-1}|M_b(\theta)|_\op\dmu<1.
\end{equation*}
Thus, we can take a sufficiently small $\ep\in(0,1)$ and sufficiently large $m\in(\kappa,\infty)$ such that $\delta_{m,\ep}<1$. Finally, noting that $\vnorm{y}_{\Psi_m}^2\geq\kappa\|y\|_{\Psi_m}^2$ with $\kappa:=\inf\supp>0$, we have
\begin{equation*}
	\big\langle b(\mu[y]),\mu_{b,\Psi_m}[y]\big\rangle+\frac{1}{2}\tr\big[\sigma(\mu[y])^\top Q_{\sigma,\Psi_m}\sigma(\mu[y])\big]\leq\frac{1+\delta_{m,\ep}}{2}\vnorm{y}_{\Psi_m}^2-\frac{\kappa(1-\delta_{m,\ep})}{2}\|y\|_{\Psi_m}^2+C_{m,\ep}.
\end{equation*}
This indicates that \cref{assum_Lyapunov-abstract} holds for the admissible weight function $\Psi=\Psi_m$ and constants $\delta=\frac{1+\delta_{m,\ep}}{2}\in(0,1)$, $\rho=\frac{\kappa(1-\delta_{m,\ep})}{2}>0$ and $C_\Lyap=C_{m,\ep}>0$. This completes the proof.
\end{proof}

%%%%%%%%%%%%%%%%%%%%%%%%%%%%%%%%%%
%%%%%%%%%%%%%%%%%%%%%%%%%%%%%%%%%%
%% Section
%%%%%%%%%%%%%%%%%%%%%%%%%%%%%%%%%%
%%%%%%%%%%%%%%%%%%%%%%%%%%%%%%%%%%

\section{Finite-dimensional approximation of the invariant probability measure}\label{sec_approximation}

In this section, we establish approximation results for the invariant probability measure $\pi\in\cP(\cH)$ and the stationary law $\bfP^\pi\in\cP(\Lambda)$ associated with the (infinite-dimensional) SEE \eqref{eq_SEE} via their counterparts of (finite-dimensional) SDEs. In particular, the results presented here extend previous studies on Markovian approximations of SVEs, such as \cite{AbiJaEu19,AlKe24,BaBr23}. Furthermore, our approach offers a mathematically rigorous framework that supports the validity of Markovian embedding-type procedures widely employed in statistical physics, as discussed in the \hyperref[sec_intro]{Introduction}.

Let $(\mu,M_b,M_\sigma)$ be a lifting basis. We introduce a sequence $\{(\Delta_k,a_k,M_{b,k},M_{\sigma,k})\}_{k\in\bN}$ of finite families $(\Delta_k,a_k,M_{b,k},M_{\sigma,k})=\{(\Delta_k^{(i)},a_k^{(i)},M_{b,k}^{(i)},M_{\sigma,k}^{(i)})\}^{I_k}_{i=1}$ for each $k\in\bN$, where
\begin{equation*}
	\big(\Delta_k^{(i)},a_k^{(i)},M_{b,k}^{(i)},M_{\sigma,k}^{(i)}\big)\in\cB([0,\infty))\times[0,\infty)\times\bR^{n\times n}\times\bR^{n\times n}
\end{equation*}
for $i\in\{1,\dots,I_k\}$ with $I_k\in\bN$, such that the following hold:
\begin{itemize}
\item
For each $k\in\bN$ and $i\in\{1,\dots,I_k\}$, the set $\Delta_k^{(i)}\in\cB([0,\infty))$ is bounded and satisfies $\mu(\Delta_k^{(i)})>0$.
\item
For each $k\in\bN$ and $i,j\in\{1,\dots,I_k\}$ with $i\neq j$, it holds that $\Delta_k^{(i)}\cap\Delta_k^{(j)}=\emptyset$.
\item
Setting
\begin{equation}\label{eq_ep_k}
\begin{split}
	\ep_k&:=\max_{i\in\{1,\dots,I_k\}}\mu\text{-}\esssup_{\theta\in\Delta_k^{(i)}}\frac{|\theta-a_k^{(i)}|}{1+\theta}+\left(\int_{[0,\infty)}(1+\theta)^{-3/2}\left|M_b(\theta)-\sum^{I_k}_{i=1}M_{b,k}^{(i)}\1_{\Delta_k^{(i)}}(\theta)\right|_\op^2\dmu\right)^{1/2}\\
	&\hspace{0.5cm}+\left(\int_{[0,\infty)}(1+\theta)^{-1/2}\left|M_\sigma(\theta)-\sum^{I_k}_{i=1}M_{\sigma,k}^{(i)}\1_{\Delta_k^{(i)}}(\theta)\right|_\op^2\dmu\right)^{1/2},
\end{split}
\end{equation}
it holds that $\lim_{k\to\infty}\ep_k=0$.
\end{itemize}
Noting the integrability conditions \eqref{eq_integrability-M_b} and \eqref{eq_integrability-M_sigma} on $M_b$ and $M_\sigma$, such a sequence always exists and can be easily constructed using a standard approximation argument for $L^2$-functions by simple functions. We call a sequence $\{(\Delta_k,a_k,M_{b,k},M_{\sigma,k})\}_{k\in\bN}$ satisfying the above conditions an \emph{approximating component} for the lifting basis $(\mu,M_b,M_\sigma)$.

Suppose that we are given measurable maps $b:\bR^n\to\bR^n$ and $\sigma:\bR^n\to\bR^{n\times d}$. Let $W$ be a $d$-dimensional Brownian motion defined on a complete probability space $(\Omega,\cF,\bP)$ with respect to a filtration $\bF=(\cF_t)_{t\geq0}$ satisfying the usual conditions. For each $k\in\bN$, consider the following finite-dimensional SDE on $(\bR^n)^{I_k}$ for an It\^{o} process $Z_k=(Z_{k,t})_{t\geq0}=((Z_{k,t}^{(i)})^{I_k}_{i=1})_{t\geq0}$:
\begin{equation}\label{eq_approximating-SDE}
\begin{split}
	&\diff Z_{k,t}^{(i)}=-a_k^{(i)}Z_{k,t}^{(i)}\,\diff t+M_{b,k}^{(i)}\,b\left(\sum^{I_k}_{j=1}\mu\big(\Delta_k^{(j)}\big)Z_{k,t}^{(j)}\right)\,\diff t+M_{\sigma,k}^{(i)}\,\sigma\left(\sum^{I_k}_{j=1}\mu\big(\Delta_k^{(j)}\big)Z_{k,t}^{(j)}\right)\,\diff W_t,\\
	&\hspace{7cm}t\geq0,\ i\in\{1,\dots,I_k\}.
\end{split}
\end{equation}
If the coefficients $b:\bR^n\to\bR^n$ and $\sigma:\bR^n\to\bR^{n\times d}$ are Lipschitz continuous, then the above SDE is strongly well-posed, and the solution forms a time-homogeneous Markov process on $(\bR^n)^{I_k}$.

%% Remark

\begin{rem}\label{rem_approximating-SVE}
Applying It\^{o}'s formula to the process $e^{a_k^{(i)}t}Z_{k,t}^{(i)}$ for each $i\in\{1,\dots,I_k\}$, we see that the solution $Z_k$ of the SDE \eqref{eq_approximating-SDE} satisfies
\begin{equation*}
	Z_{k,t}^{(i)}=e^{-a_k^{(i)}t}Z_{k,0}^{(i)}+\int^t_0e^{-a_k^{(i)}(t-s)}M_{b,k}^{(i)}\,b(X_{k,s})\,\diff s+\int^t_0e^{-a_k^{(i)}(t-s)}M_{\sigma,k}^{(i)}\,\sigma(X_{k,s})\,\diff W_s
\end{equation*}
for $t\geq0$ and $i\in\{1,\dots,I_k\}$, where $X_{k,t}:=\sum^{I_k}_{i=1}\mu(\Delta_k^{(i)})Z_{k,t}^{(i)}$. Multiplying both sides by $\mu(\Delta_k^{(i)})$ and summing over $i\in\{1,\dots,I_k\}$ shows that the $\bR^n$-valued progressively measurable (and continuous) process $X_k=(X_{k,t})_{t\geq0}$ solves the following SVE:
\begin{equation}\label{eq_approximating-SVE}
	X_{k,t}=x_k(t)+\int^t_0K_{b,k}(t-s)b(X_{k,s})\,\diff s+\int^t_0K_{\sigma,k}(t-s)\sigma(X_{k,s})\,\diff W_s,\ \ t\geq0,
\end{equation}
where
\begin{equation*}
	x_k(t)=\sum^{I_k}_{i=1}e^{-a_k^{(i)}t}\mu\big(\Delta_k^{(i)}\big)Z_{k,0}^{(i)},\ \ K_{b,k}(t)=\sum^{I_k}_{i=1}e^{-a_k^{(i)}t}\mu\big(\Delta_k^{(i)}\big)M_{b,k}^{(i)},\ \ K_{\sigma,k}(t)=\sum^{I_k}_{i=1}e^{-a_k^{(i)}t}\mu\big(\Delta_k^{(i)}\big)M_{\sigma,k}^{(i)},
\end{equation*}
for $t\geq0$. Note that the forcing term $x_k$ and kernels $K_{b,k}$ and $K_{\sigma,k}$ are of the (finite) sum-of-exponentials type as in \cref{exam_liftable} (i).
\end{rem}

For each $\nu_k\in\cP((\bR^n)^{I_k})$, we denote by $\bfP_k^{\nu_k}\in\cP(C([0,\infty);(\bR^n)^{I_k}))$ the law of a solution $Z_k=(Z_{k,t})_{t\geq0}$ to the approximating SDE \eqref{eq_approximating-SDE} with initial distribution $\nu_k$, where $C([0,\infty);(\bR^n)^{I_k})$ is the Polish space of $(\bR^n)^{I_k}$-valued continuous functions on $[0,\infty)$, which is equipped with the complete metric defined by \eqref{eq_metric-C}. The associated Markov semigroup $\{P_{k,t}\}_{t\geq0}$ on $(\bR^n)^{I_k}$ is defined by
\begin{equation*}
	P_{k,t}(z_k,A_k):=\bfP_k^{\delta_{z_k}}\left(\left\{\zeta_k\in C\big([0,\infty);(\bR^n)^{I_k}\big)\relmiddle|\zeta_{k,t}\in A_k\right\}\right)
\end{equation*}
for $z_k\in(\bR^n)^{I_k}$, $A_k\in\cB((\bR^n)^{I_k})$, and $t\geq0$.

The purpose of this section is to show that the invariant probability measure $\pi_k\in\cP((\bR^n)^{I_k})$ for the Markov semigroup $\{P_{k,t}\}_{t\geq0}$ ``converges'' to the invariant probability measure $\pi\in\cP(\cH)$ for the Markov semigroup $\{P_t\}_{t\geq0}$ associated with the SEE \eqref{eq_SEE} in a suitable sense. Beyond the convergence of invariant probability measures, we investigate the ``convergence'' of the stationary laws $\bfP_k^{\pi_k}\in\cP(C([0,\infty);(\bR^n)^{I_k}))$ for the finite-dimensional SDEs \eqref{eq_approximating-SDE} to the limit stationary law $\bfP^\pi\in\cP(\Lambda)$ for the SEE \eqref{eq_SEE}. More precise statements are given in \cref{theo_approximation-IPM} below.

Since the SEE \eqref{eq_SEE} and the approximating SDEs \eqref{eq_approximating-SDE} are defined on different state spaces, in order to compare the solutions, we have to embed the finite-dimensional space $(\bR^n)^{I_k}$ to $\cH$ in a suitable manner. To this end, for each $k\in\bN$, let us introduce a bounded linear operator $\Upsilon_k:(\bR^n)^{I_k}\to\cV$ by
\begin{equation*}
	(\Upsilon_kz_k)(\theta):=\sum^{I_k}_{i=1}z_k^{(i)}\1_{\Delta_k^{(i)}}(\theta),\ \ \theta\in[0,\infty),
\end{equation*}
for $z_k=(z_k^{(i)})^{I_k}_{i=1}\in(\bR^n)^{I_k}$. Note that
\begin{equation*}
	\big\|\Upsilon_kz_k\big\|_{\cH}=\left(\sum^{I_k}_{i=1}\int_{\Delta_k^{(i)}}(1+\theta)^{-1/2}\dmu\,\big|z_k^{(i)}\big|^2\right)^{1/2}\ \ \text{and}\ \ \big\|\Upsilon_kz_k\big\|_{\cV}=\left(\sum^{I_k}_{i=1}\int_{\Delta_k^{(i)}}(1+\theta)^{1/2}\dmu\,\big|z_k^{(i)}\big|^2\right)^{1/2}
\end{equation*}
for each $z_k=(z_k^{(i)})^{I_k}_{i=1}\in(\bR^n)^{I_k}$. Both the maps $z_k\mapsto\|\Upsilon_kz_k\|_\cH$ and $z_k\mapsto\|\Upsilon_kz_k\|_\cV$ define (equivalent) norms on the finite-dimensional space $(\bR^n)^{I_k}$, as each $\Delta_k^{(i)}$ is bounded and satisfies $\mu(\Delta_k^{(i)})>0$.

%% Lemma

\begin{lemm}\label{lemm_approximating-SDE}
Let a lifting basis $(\mu,M_b,M_\sigma)$ and its approximating component $\{(\Delta_k,a_k,M_{b,k},M_{\sigma,k})\}_{k\in\bN}$ be given. Assume that $b:\bR^n\to\bR^n$ and $\sigma:\bR^n\to\bR^{n\times d}$ are Lipschitz continuous.
\begin{itemize}
\item[(i)]
Suppose that we are given a $d$-dimensional Brownian motion $W$ defined on a complete probability space $(\Omega,\cF,\bP)$, with respect to a filtration $\bF=(\cF_t)_{t\geq0}$ satisfying the usual conditions. Then, for any $k\in\bN$ and any $(\bR^n)^{I_k}$-valued $\cF_0$-measurable random variable $Z_{k,0}$, the SDE \eqref{eq_approximating-SDE} has a unique solution $Z_k=(Z_{k,t})_{t\geq0}$ with the prescribed initial condition $Z_{k,0}$. This solution forms a time-homogeneous Markov process on $(\bR^n)^{I_k}$, and its Markov semigroup $\{P_{k,t}\}_{t\geq0}$ is stochastically continuous and satisfies the Feller property.
\item[(ii)]
There exists a constant $C_0>0$ such that, for each $k\in\bN$, any solution $Z_k=(Z_{k,t})_{t\geq0}$ of the SDE \eqref{eq_approximating-SDE} and any solution $Y=(Y_t)_{t\geq0}$ of the SEE \eqref{eq_SEE} driven by a common $d$-dimensional Brownian motion $W$ on a complete probability space $(\Omega,\cF,\bP)$, we have
\begin{equation}\label{eq_approximation1}
	\bE\left[\sup_{t\in[0,T]}\big\|\Upsilon_kZ_{k,t}\big\|^2_{\cH}+\int^T_0\big\|\Upsilon_kZ_{k,t}\big\|_{\cV}^2\,\diff t\relmiddle|\cF_0\right]\leq C_0e^{C_0T}\left(1+\big\|\Upsilon_kZ_{k,0}\big\|_{\cH}^2\right)
\end{equation}
and
\begin{equation}\label{eq_approximation2}
\begin{split}
	&\bE\left[\sup_{t\in[0,T]}\big\|Y_t-\Upsilon_kZ_{k,t}\big\|^2_{\cH}+\int^T_0\big\|Y_t-\Upsilon_kZ_{k,t}\big\|_{\cV}^2\,\diff t\relmiddle|\cF_0\right]\\
	&\leq C_0e^{C_0T}\big\|Y_0-\Upsilon_kZ_{k,0}\big\|_\cH^2+C_0e^{C_0T}\ep_k^2\left(1+\big\|\Upsilon_kZ_{k,0}\big\|_{\cH}^2\right),
\end{split}
\end{equation}
for any $T>0$ a.s., where $\ep_k>0$ is the constant defined in \eqref{eq_ep_k}.
\item[(iii)]
If \cref{assum_Lyapunov-abstract} further holds, then there exists a natural number $k_0$ such that, for each $k\in\bN$ with $k\geq k_0$, the Markov semigroup $\{P_{k,t}\}_{t\geq0}$ associated with the SDE \eqref{eq_approximating-SDE} possesses at least one invariant probability measure $\pi_k\in\cP((\bR^n)^{I_k})$. Furthermore, it holds that
\begin{equation}\label{eq_approximation-IPM-integrability}
	\sup_{k\geq k_0}\sup_{\pi_k\in\Pi_k}\int_{(\bR^n)^{I_k}}\big\|\Upsilon_kz_k\big\|_{\cV}^2\,\pi_k(\diff z_k)<\infty,
\end{equation}
where $\Pi_k\subset\cP((\bR^n)^{I_k})$ denotes the (non-empty) set of all invariant probability measures for $\{P_{k,t}\}_{t\geq0}$.
\end{itemize}
\end{lemm}

%% Proof

\begin{proof}
For each $k\in\bN$, since the coefficients of the SDE \eqref{eq_approximating-SDE} on $(\bR^n)^{I_k}$ are Lipschitz continuous, assertion (i) follows from standard results on finite-dimensional SDEs. Furthermore, there exists a constant $\widetilde{C}_k>0$ (which may depend on $k\in\bN$) such that
\begin{equation}\label{eq_SDE-apriori-k}
	\bE\left[\sup_{t\in[0,T]}\big|Z_{k,t}\big|^2\relmiddle|\cF_0\right]\leq\widetilde{C}_ke^{\widetilde{C}_kT}\left(1+\big|Z_{k,0}\big|^2\right)<\infty
\end{equation}
for any $T>0$ a.s. Here, $|\cdot|$ denotes the standard Euclidean norm on $(\bR^n)^{I_k}$.

To prove assertions (ii) and (iii), let us introduce additional notation for use throughout this proof. For each $k\in\bN$, we set $\cV_k:=\{\Upsilon_kz_k\,|\,z_k\in(\bR^n)^{I_k}\}$. Note that $\cV_k\subset\cV$. Furthermore, for any $y_k=\Upsilon_kz_k\in\cV_k$ with $z_k=(z_k^{(i)})^{I_k}_{i=1}\in(\bR^n)^{I_k}$, we have
\begin{equation*}
	\mu[y_k]=\mu\big[\Upsilon_kz_k\big]=\sum^{I_k}_{i=1}\mu\big(\Delta_k^{(i)}\big)z_k^{(i)}.
\end{equation*}
We define $\cA_k:\cV_k\to\cV_k$ and $\cM_{b,k},\cM_{\sigma,k}:\bR^n\to\cV_k$ by
\begin{equation*}
	\cA_ky_k:=-\sum^{I_k}_{i=1}a_k^{(i)}z_k^{(i)}\1_{\Delta_k^{(i)}}
\end{equation*}
for $y_k=\Upsilon_kz_k$ with $z_k=(z_k^{(i)})^{I_k}_{i=1}\in(\bR^n)^{I_k}$, and
\begin{equation*}
	\cM_{b,k}x:=\sum^{I_k}_{i=1}M_{b,k}^{(i)}x\1_{\Delta_k^{(i)}},\ \ \cM_{\sigma,k}x:=\sum^{I_k}_{i=1}M_{\sigma,k}^{(i)}x\1_{\Delta_k^{(i)}},
\end{equation*}
for $x\in\bR^n$. Recalling the definition \eqref{eq_ep_k} of $\ep_k$, for any $y_k=\Upsilon_kz_k\in\cV_k$ with $z_k=(z_k^{(i)})^{I_k}_{i=1}\in(\bR^n)^{I_k}$, we have
\begin{align}
	\nonumber
	\big\|(\cA_k-\cA)y_k\big\|_{\cV^*}&=\left(\sum^{I_k}_{i=1}\int_{\Delta_k^{(i)}}(1+\theta)^{-3/2}\big|\theta-a_k^{(i)}\big|^2\big|z_k^{(i)}\big|^2\dmu\right)^{1/2}\\
	\nonumber
	&\leq\ep_k\left(\sum^{I_k}_{i=1}\int_{\Delta_k^{(i)}}(1+\theta)^{1/2}\big|z_k^{(i)}\big|^2\dmu\right)^{1/2}\\
	\label{eq_A-k}
	&=\ep_k\|y_k\|_\cV.
\end{align}
Furthermore, for any $x\in\bR^n$, we have
\begin{equation}\label{eq_M_b-k}
	\big\|(\cM_{b,k}-\cM_b)x\big\|_{\cV^*}=\left(\int_{[0,\infty)}(1+\theta)^{-3/2}\left|\left(\sum^{I_k}_{i=1}M_{b,k}^{(i)}\1_{\Delta_k^{(i)}}(\theta)-M_b(\theta)\right)x\right|^2\dmu\right)^{1/2}\leq\ep_k|x|
\end{equation}
and
\begin{equation}\label{eq_M_sigma-k}
	\big\|(\cM_{\sigma,k}-\cM_\sigma)x\big\|_\cH=\left(\int_{[0,\infty)}(1+\theta)^{-1/2}\left|\left(\sum^{I_k}_{i=1}M_{\sigma,k}^{(i)}\1_{\Delta_k^{(i)}}(\theta)-M_\sigma(\theta)\right)x\right|^2\dmu\right)^{1/2}\leq\ep_k|x|.
\end{equation}

Now we prove the estimate \eqref{eq_approximation1} in assertion (ii). Recalling \eqref{eq_SDE-apriori-k}, together with the fact that $\|\Upsilon_k\cdot\|_\cH$, $\|\Upsilon_k\cdot\|_\cV$ and $|\cdot|$ are equivalent norms on the finite-dimensional space $(\bR^n)^{I_k}$, it suffices to show that there exists a constant $C_0>0$ and a number $k_1\in\bN$ such that \eqref{eq_approximation1} holds for any $k\in\bN$ with $k\geq k_1$. For a while, we fix an arbitrary $k\in\bN$. Let $Z_k=(Z_{k,t})_{t\geq0}$ be a solution to the SDE \eqref{eq_approximating-SDE}, and set $Y_{k,t}:=\Upsilon_kZ_{k,t}$ for $t\geq0$. By the above observations, the $\cV_k$-valued process $Y_k=(Y_{k,t})_{t\geq0}$ satisfies
\begin{equation}\label{eq_SEE-UpsilonZ}
	Y_{k,t}=Y_{k,0}+\int^t_0\big\{\cA_kY_{k,s}+\cM_{b,k}b(\mu[Y_{k,s}])\big\}\,\diff s+\int^t_0\cM_{\sigma,k}\sigma(\mu[Y_{k,s}])\,\diff W_s
\end{equation}
for any $t\geq0$ a.s. Applying It\^{o}'s formula for $\|\cdot\|_\cH^2$ in \cref{lemm_Ito} to $Y_k$ yields
\begin{equation}\label{eq_Ito-f-k}
	\big\|Y_{k,t}\big\|_\cH^2=\big\|Y_{k,0}\big\|_\cH^2+\int^t_0f_k(Y_{k,s})\,\diff s+2\int^t_0\big\langle Y_{k,s},\cM_{\sigma,k}\sigma(\mu[Y_{k,s}])\,\diff W_s\big\rangle_\cH
\end{equation}
for any $t\geq0$ a.s., where $f_k:\cV_k\to\bR$ is defined by
\begin{equation*}
	f_k(y_k):=2\big\langle\cA_ky_k+\cM_{b,k}b(\mu[y_k]),y_k\big\rangle_{\cV^*,\cV}+\big\|\cM_{\sigma,k}\sigma(\mu[y_k])\big\|_{L_2(\bR^d;\cH)}^2,\ \ y_k\in\cV_k.
\end{equation*}
Decompose $f_k$ as $f_k=f_k^{(1)}+f_k^{(2)}$, where
\begin{align*}
	&f_k^{(1)}(y_k):=2\big\langle\cA y_k,y_k\big\rangle_{\cV^*,\cV}+2\big\langle\cM_bb(\mu[y_k]),y_k\big\rangle_{\cV^*,\cV}+\big\|\cM_\sigma\sigma(\mu[y_k])\big\|_{L_2(\bR^d;\cH)}^2,\\
	&f_k^{(2)}(y_k):=2\big\langle(\cA_k-\cA)y_k,y_k\big\rangle_{\cV^*,\cV}+2\big\langle(\cM_{b,k}-\cM_b)b(\mu[y_k]),y_k\big\rangle_{\cV^*,\cV}\\
	&\hspace{3cm}+\big\langle(\cM_{\sigma,k}-\cM_\sigma)\sigma(\mu[y_k]),(\cM_{\sigma,k}+\cM_\sigma)\sigma(\mu[y_k])\big\rangle_{L_2(\bR^d;\cH)},
\end{align*}
for $y_k\in\cV_k$. On the one hand, analogously to the proof of \cref{prop_SEE-well-posed} (ii), by employing \eqref{eq_A-nonpositive'}, $\cM_b\in L(\bR^n;\cV)$, $\cM_\sigma\in L(\bR^n;\cH)$, the linear growth of $b:\bR^n\to\bR^n$ and $\sigma:\bR^n\to\bR^{n\times d}$, and \cref{lemm_mu-ep}, together with Young's inequality, we see that there exists a constant $C_1>0$, independent of $k$, such that
\begin{equation*}
	f_k^{(1)}(y_k)\leq C_1\|y_k\|_\cH^2-\frac{3}{2}\|y_k\|_\cV^2+C_1
\end{equation*}
for any $y_k\in\cV_k$. On the other hand, from \eqref{eq_A-k}, \eqref{eq_M_b-k} and \eqref{eq_M_sigma-k}, along with the linear growth of $b$ and $\sigma$ and the fact that $\mu[\cdot]\in L(\cV;\bR^n)$, we see that there exists a constant $C_2>0$, also independent of $k$, such that
\begin{align*}
	f_k^{(2)}(y_k)\leq C_2\ep_k\Big(1+\|y_k\|_\cV^2\Big)
\end{align*}
for any $y_k\in\cV_k$. Now, we choose $k_1\in\bN$ such that $C_2\ep_k\leq\frac{1}{2}$ for all $k\geq k_1$. This is possible since $\lim_{k\to\infty}\ep_k=0$ by assumption. Let $k\geq k_1$. Then, we have
\begin{equation}\label{eq_f-k}
	f_k(y_k)=f_k^{(1)}(y_k)+f_k^{(2)}(y_k)\leq C_1\|y_k\|_\cH^2-\|y_k\|_\cV^2+C_1+\frac{1}{2}
\end{equation}
for any $y_k\in\cV_k$. By \eqref{eq_Ito-f-k} and \eqref{eq_f-k}, we have
\begin{align*}
	&\big\|Y_{k,t}\big\|_\cH^2+\int^t_0\big\|Y_{k,s}\big\|_\cV^2\,\diff s\\
	&\leq\big\|Y_{k,0}\big\|_\cH^2+C_1\int^t_0\big\|Y_{k,s}\big\|_\cH^2\,\diff s +\left(C_1+\frac{1}{2}\right)t+2\int^t_0\big\langle Y_{k,s},\cM_{\sigma,k}\sigma(\mu[Y_{k,s}])\,\diff W_s\big\rangle_\cH
\end{align*}
for any $t\geq0$ a.s. From this estimate, by arguments analogous to the proof of \cref{prop_SEE-well-posed} (ii), it can be shown that
\begin{equation*}
	\bE\left[\sup_{t\in[0,T]}\big\|Y_{k,t}\big\|_\cH^2+\int^T_0\big\|Y_{k,t}\big\|_\cV^2\,\diff t\relmiddle|\cF_0\right]\leq C_0e^{C_0T}\left(1+\big\|Y_{k,0}\big\|_\cH^2\right)
\end{equation*}
for any $T>0$ a.s., where $C_0>0$ is a constant independent of $k$. Thus, \eqref{eq_approximation1} holds for all $k\geq k_1$. This shows assertion (i).

Next, we prove \eqref{eq_approximation2} in assertion (ii). Let $k\in\bN$ be fixed. In view of \eqref{eq_SEE-UpsilonZ}, applying It\^{o}'s formula for $\|\cdot\|_\cH^2$ in \cref{lemm_Ito} to the process $Y-Y_k$ yields
\begin{equation}\label{eq_Ito-g-k}
\begin{split}
	\big\|Y_t-Y_{k,t}\big\|_\cH^2&=\big\|Y_0-Y_{k,0}\big\|_\cH^2+\int^t_0g_k(Y_s,Y_{k,s})\,\diff s\\
	&\hspace{0.5cm}+2\int^t_0\Big\langle Y_s-Y_{k,s},\big(\cM_\sigma\sigma(\mu[Y_s])-\cM_{\sigma,k}\sigma(\mu[Y_{k,s}])\big)\,\diff W_s\Big\rangle_\cH
\end{split}
\end{equation}
for any $t\geq0$ a.s., where the function $g_k:\cV\times\cV_k\to\bR$ is defined by
\begin{align*}
	g_k(y,y_k)&:=2\big\langle\cA y+\cM_bb(\mu[y])-\cA_ky_k-\cM_{b,k}b(\mu[y_k]),y-y_k\big\rangle_{\cV^*,\cV}\\
	&\hspace{2cm}+\big\|\cM_\sigma\sigma(\mu[y])-\cM_{\sigma,k}\sigma(\mu[y_k])\big\|_{L_2(\bR^d;\cH)}^2,\ \ (y,y_k)\in\cV\times\cV_k.
\end{align*}
Decompose $g_k$ as $g_k=g_k^{(1)}+g_k^{(2)}$, where
\begin{align*}
	&g_k^{(1)}(y,y_k):=2\big\langle\cA(y-y_k),y-y_k\big\rangle_{\cV^*,\cV}+2\big\langle\cM_b\big(b(\mu[y])-b(\mu[y_k])\big),y-y_k\big\rangle_{\cV^*,\cV}\\
	&\hspace{2cm}+\big\|\cM_\sigma\big(\sigma(\mu[y])-\sigma(\mu[y_k])\big)\big\|_{L_2(\bR^d;\cH)}^2,\\
	&g_k^{(2)}(y,y_k):=2\big\langle(\cA-\cA_k)y_k,y-y_k\big\rangle_{\cV^*,\cV}+2\big\langle(\cM_b-\cM_{b,k})b(\mu[y_k]),y-y_k\big\rangle_{\cV^*,\cV}\\
	&\hspace{2cm}+2\big\langle(\cM_\sigma-\cM_{\sigma,k})\sigma(\mu[y_k]),\cM_\sigma\big(\sigma(\mu[y])-\sigma(\mu[y_k])\big)\Big\rangle_{L_2(\bR^d;\cH)}\\
	&\hspace{2cm}+\big\|(\cM_\sigma-\cM_{\sigma,k})\sigma(\mu[y_k])\big\|_{L_2(\bR^d;\cH)}^2,
\end{align*}
for $(y,y_k)\in\cV\times\cV_k$. Analogously to the proof of \cref{prop_SEE-well-posed} (ii), by using \eqref{eq_A-nonpositive'}, $\cM_b\in L(\bR^n;\cV)$, $\cM_\sigma\in L(\bR^n;\cH)$, Lipschitz continuity of $b:\bR^n\to\bR^n$ and $\sigma:\bR^n\to\bR^{n\times d}$, and \cref{lemm_mu-ep}, together with Young's inequality, we see that there exists a constant $C_3>0$, independent of $k$, such that
\begin{align*}
	g_k^{(1)}(y,y_k)&\leq C_3\|y-y_k\|_\cH^2-\frac{3}{2}\|y-y_k\|_\cV^2
\end{align*}
for any $(y,y_k)\in\cV\times\cV_k$. On the other hand, by utilizing \eqref{eq_A-k}, \eqref{eq_M_b-k}, \eqref{eq_M_sigma-k}, Lipschitz continuity of $b$ and $\sigma$, and $\mu[\cdot]\in L(\cV;\bR^n)$, together with Young's inequality, we see that there exists a constant $C_4>0$, independent of $k$, such that
\begin{align*}
	g_k^{(2)}(y,y_k)&\leq\frac{1}{2}\|y-y_k\|_\cV^2+C_4\ep_k^2\left\{1+\|y_k\|_\cV^2\right\}
\end{align*}
for any $(y,y_k)\in\cV\times\cV_k$. Combining the above estimates, we obtain
\begin{equation}\label{eq_g-k}
	g_k(y,y_k)=g_k^{(1)}(y,y_k)+g_k^{(2)}(y,y_k)\leq C_3\|y-y_k\|_\cH^2-\|y-y_k\|_\cV^2+C_4\ep_k^2\left\{1+\|y_k\|_\cV^2\right\}
\end{equation}
for any $(y,y_k)\in\cV\times\cV_k$. By \eqref{eq_Ito-g-k} and \eqref{eq_g-k}, we have
\begin{align*}
	&\big\|Y_t-Y_{k,t}\big\|_\cH^2+\int^t_0\big\|Y_s-Y_{k,s}\big\|_\cV^2\,\diff s\\
	&\leq\big\|Y_0-Y_{k,0}\big\|_\cH^2+C_3\int^t_0\big\|Y_s-Y_{k,s}\big\|_\cH^2\,\diff s+C_4\ep_k^2\int^t_0\Big\{1+\big\|Y_{k,s}\big\|_\cV^2\Big\}\,\diff s\\
	&\hspace{1cm}+2\int^t_0\Big\langle Y_s-Y_{k,s},\big(\cM_\sigma\sigma(\mu[Y_s])-\cM_{\sigma,k}\sigma(\mu[Y_{k,s}])\big)\,\diff W_s\Big\rangle_\cH
\end{align*}
for any $t\geq0$ a.s. From this, by the same manner as in the proof of \cref{prop_SEE-well-posed} (ii), we can show that there exists a constant $C_5>0$, independent $k$, such that
\begin{align*}
	&\bE\left[\sup_{t\in[0,T]}\big\|Y_t-Y_{k,t}\big\|_\cH^2+\int^T_0\big\|Y_t-Y_{k,t}\big\|_\cV^2\,\diff t\relmiddle|\cF_0\right]\\
	&\leq C_5e^{C_5T}\big\|Y_0-Y_{k,0}\big\|_\cH^2+C_5e^{C_5T}\ep_k^2\bE\left[\int^T_0\Big\{1+\big\|Y_{k,t}\big\|_\cV^2\Big\}\,\diff t\relmiddle|\cF_0\right]
\end{align*}
for any $T>0$. Combining this estimate with \eqref{eq_approximation1} ensures that \eqref{eq_approximation2} holds for any $T>0$ and $k\in\bN$ for some constant $C_0>0$.

Lastly, we prove assertion (iii). We additionally assume that \cref{assum_Lyapunov-abstract} holds. Recall the notations in \cref{defi_Phi}. Let $k\in\bN$ be fixed. Let $\Psi$ be the admissible weight function in \cref{assum_Lyapunov-abstract}. In view of \eqref{eq_SEE-UpsilonZ}, applying It\^{o}'s formula for $\|\cdot\|_{\Psi}^2$ in \cref{lemm_Ito} to $Y_k$ yields
\begin{equation}\label{eq_Ito-h-k}
	\big\|Y_{k,t}\big\|_{\Psi}^2=\big\|Y_{k,0}\big\|_{\Psi}^2+\int^t_0h_k(Y_{k,s})\,\diff s+2\int^t_0\big\langle Y_{k,s},\cM_{\sigma,k}\sigma(\mu[Y_{k,s}])\,\diff W_s\big\rangle_{\Psi}
\end{equation}
for any $t\geq0$ a.s., where $h_k:\cV_k\to\bR$ is defined by
\begin{equation*}
	h_k(y_k):=2\big\langle\cA_ky_k+\cM_{b,k}b(\mu[y_k]),y_k\big\rangle_{\cV^*,\cV,\Psi}+\big\|\cM_{\sigma,k}\sigma(\mu[y_k])\big\|_{L_2(\bR^d;(\cH,\langle\cdot,\cdot\rangle_{\Psi}))}^2,\ \ y_k\in\cV_k.
\end{equation*}
Decompose $h_k$ as $h_k=h_k^{(1)}+h_k^{(2)}$, where
\begin{align*}
	&h_k^{(1)}(y_k):=2\big\langle\cA y_k,y_k\big\rangle_{\cV^*,\cV,\Psi}+2\big\langle b(\mu[y_k]),\mu_{b,\Psi}[y_k]\big\rangle+\tr\big[\sigma(\mu[y_k])^\top Q_{\sigma,\Psi}\sigma(\mu[y_k])\big],\\
	&h_k^{(2)}(y_k):=2\big\langle(\cA_k-\cA)y_k,y_k\big\rangle_{\cV^*,\cV,\Psi}+2\big\langle(\cM_{b,k}-\cM_b)b(\mu[y_k]),y_k\big\rangle_{\cV^*,\cV,\Psi}\\
	&\hspace{3cm}+\big\langle(\cM_{\sigma,k}-\cM_\sigma)\sigma(\mu[y_k]),(\cM_{\sigma,k}+\cM_\sigma)\sigma(\mu[y_k])\big\rangle_{L_2(\bR^d;(\cH,\langle\cdot,\cdot\rangle_{\Psi}))}
\end{align*}
for $y_k\in\cV_k$. Combining the equality \eqref{eq_A-Phi} and the estimate \eqref{eq_coercivity} in \cref{assum_Lyapunov-abstract}, we have
\begin{equation*}
	h_k^{(1)}(y_k)\leq-2\Big\{\rho\|y_k\|_{\Psi}^2+(1-\delta)\vnorm{y_k}_\Psi^2\Big\}+2C_\Lyap
\end{equation*}
for any $y_k\in\cV_k$. Since the map $y\mapsto\big(\rho\|y\|_{\Psi}^2+(1-\delta)\vnorm{y}_\Psi^2)^{1/2}$ defines a norm on $\cV$ equivalent to $\|\cdot\|_\cV$, by \eqref{eq_A-k}, \eqref{eq_M_b-k} and \eqref{eq_M_sigma-k}, together with the linear growth of $b:\bR^n\to\bR^n$ and $\sigma:\bR^n\to\bR^{n\times d}$ and the fact that $\mu[\cdot]\in L(\cV;\bR^n)$, we see that there exists a constant $C_6>0$, independent of $k$, such that
\begin{equation*}
	h_k^{(2)}(y_k)\leq C_6\ep_k\Big\{1+\rho\|y_k\|_{\Psi}^2+(1-\delta)\vnorm{y_k}_\Psi^2\Big\}
\end{equation*}
for any $y_k\in\cV_k$. Now we take $k_0\in\bN$ such that $C_6\ep_k\leq1$ for any $k\in\bN$ with $k\geq k_0$, and fix such a $k$. Then, we have
\begin{equation}\label{eq_h-k}
	h_k(y_k)=h_k^{(1)}(y_k)+h_k^{(2)}(y_k)\leq-\Big\{\rho\|y_k\|_{\Psi}^2+(1-\delta)\vnorm{y_k}_\Psi^2\Big\}+2C_\Lyap+1
\end{equation}
for any $y_k\in\cV_k$. By \eqref{eq_Ito-h-k} and \eqref{eq_h-k}, we obtain
\begin{equation}\label{eq_UpsilonZ-Ito-Lyapunov'}
\begin{split}
	&\big\|Y_{k,t}\big\|_{\Psi}^2+\int^t_0\Big\{\rho\big\|Y_{k,s}\big\|_{\Psi}^2+(1-\delta)\vnorm[\big]{Y_{k,s}}_\Psi^2\Big\}\,\diff s\\
	&\leq\big\|Y_{k,0}\big\|_\Psi^2+\big(2C_\Lyap+1\big)t+2\int^t_0\big\langle Y_{k,s},\cM_{\sigma,k}\sigma(\mu[Y_{k,s}])\,\diff W_s\big\rangle_{\Psi}
\end{split}
\end{equation}
for any $t\geq0$ a.s. Let $\{P_{k,t}\}_{t\geq0}$ be the Markov semigroup on $(\bR^n)^{I_k}$ associated with the SDE \eqref{eq_approximating-SDE}. For each $t>0$, define a probability kernel $\nu_{k,t}:(\bR^n)^{I_k}\times\cB((\bR^n)^{I_k})\to[0,1]$ by $\nu_{k,t}(z_k,A_k):=\frac{1}{t}\int^t_0P_{k,s}(z_k,A_k)\,\diff s$ for $z_k\in(\bR^n)^{I_k}$ and $A_k\in\cB((\bR^n)^{I_k})$. Analogously to the proof of \cref{theo_Lyapunov} (see \cref{subsec_proof-Lyapunov}), recalling $Y_{k,t}=\Upsilon_kZ_{k,t}$, the estimate \eqref{eq_UpsilonZ-Ito-Lyapunov'} ensures that
\begin{equation*}
	\int_{(\bR^n)^{I_k}}\Big\{\rho\big\|\Upsilon_kz_k'\big\|_{\Psi}^2+(1-\delta)\vnorm[\big]{\Upsilon_kz_k'}_\Psi^2\Big\}\,\nu_{k,t}(z_k,\diff z_k')\leq\frac{\big\|\Upsilon_kz_k\big\|_{\Psi}^2}{t}+2C_\Lyap+1
\end{equation*}
for all $t>0$ and all $z_k\in(\bR^n)^{I_k}$. Since $z_k\mapsto(\rho\|\Upsilon_kz_k\|_{\Psi}^2+(1-\delta)\vnorm{\Upsilon_kz_k}_\Psi^2)^{1/2}$ defines a norm on the finite-dimensional space $(\bR^n)^{I_k}$, the above estimate immediately implies that the family of probability measures $\{\nu_{k,t}(z_k,\cdot)\}_{t\geq1}$ on $(\bR^n)^{I_k}$ is tight for any $z_k\in(\bR^n)^{I_k}$; see \cref{rem_Lyapunov}. Hence, by the Krylov--Bogoliubov theorem (see, e.g., \cite[Corollary 11.8]{DaPrZa14}), we see that there exists at least one invariant probability measure $\pi_k\in\cP((\bR^n)^{I_k})$ for $\{P_{k,t}\}_{t\geq0}$. Furthermore, based on the estimate for $\nu_{k,t}$ obtained above, an argument analogous to the final part of the proof of \cref{theo_Lyapunov} (see \cref{subsec_proof-Lyapunov}) implies that
\begin{equation*}
	\int_{(\bR^n)^{I_k}}\Big\{\rho\big\|\Upsilon_kz_k\big\|_{\Psi}^2+(1-\delta)\vnorm[\big]{\Upsilon_kz_k}_\Psi^2\Big\}\,\pi_k(\diff z_k)\leq 2C_\Lyap+1.
\end{equation*}
Since $y\mapsto(\rho\|y\|_{\Psi}^2+(1-\delta)\vnorm{y}_\Psi^2)^{1/2}$ defines a norm on $\cV$ equivalent to $\|\cdot\|_\cV$, we see that \eqref{eq_approximation-IPM-integrability} holds. This completes the proof.
\end{proof}

Under the settings of \cref{lemm_approximating-SDE} (iii), for each $\pi_k\in\Pi_k$ with $k\geq k_0$, the law $\bfP_k^{\pi_k}\in\cP(C([0,\infty);(\bR^n)^{I_k})$ of a solution $Z_k=(Z_{k,t})_{t\geq0}$ to the approximating SDE \eqref{eq_approximating-SDE} with initial distribution $\pi_k$ is invariant under the time-shifts on $C([0,\infty);(\bR^n)^{I_k})$; see \cref{cor_stationary-path} (i) for the corresponding result on the stationary law $\bfP^\pi\in\cP(\Lambda)$ of the SEE \eqref{eq_SEE}.

With a slight abuse of notation, we continue to denote by $\Upsilon_k$ the map that sends each $\zeta_k=(\zeta_{k,t})_{t\geq0}\in C([0,\infty);(\bR^n)^{I_k})$ to $(\Upsilon_k\zeta_{k,t})_{t\geq0}\in\Lambda$. Then, we can regard $\Upsilon_k$ as a continuous map from $C([0,\infty);(\bR^n)^{I_k})$ to $\Lambda$. For each $T\in(0,\infty)$, define an extended pseudo-metric $\bW_{\Lambda_T}:\cP(\Lambda)\times\cP(\Lambda)\to[0,\infty]$ on $\cP(\Lambda)$ by
\begin{equation}\label{eq_Wasserstein-Lambda-T}
	\bW_{\Lambda_T}(\fp_1,\fp_2):=\inf_{\fp\in\sC(\fp_1,\fp_2)}\int_{\Lambda\times\Lambda}\|\eta_1-\eta_2\|_{\Lambda_T}\,\fp(\diff\eta_1,\diff\eta_2)
\end{equation}
for $\fp_1,\fp_2\in\cP(\Lambda)$, where the seminorm $\|\cdot\|_{\Lambda_T}$ on $\Lambda$ is defined by \eqref{eq_Lambda-norm}.

The following theorem is the main result of this section, establishing the weak convergence of $\bfP_k^{\pi_k}\circ\Upsilon_k^{-1}$ to $\bfP^\pi$ as probability measures on the Polish space $\Lambda$.

%% Theorem

\begin{theo}\label{theo_approximation-IPM}
Suppose that we are given a lifting basis $(\mu,M_b,M_\sigma)$ and measurable maps $b:\bR^n\to\bR^n$ and $\sigma:\bR^n\to\bR^{n\times d}$ satisfying both \cref{assum_coupling} and \cref{assum_Lyapunov-abstract}. Let $\{(\Delta_k,a_k,M_{b,k},M_{\sigma,k})\}_{k\in\bN}$ be an approximating component for $(\mu,M_b,M_\sigma)$. Let $k_0\in\bN$ be the number as given in \cref{lemm_approximating-SDE} (iii). For each $k\in\bN$ with $k\geq k_0$, let $\Pi_k\subset\cP((\bR^n)^{I_k})$ be the (non-empty) set of all invariant probability measures for the Markov semigroup $\{P_{k,t}\}_{t\geq0}$ associated with the SDE \eqref{eq_approximating-SDE}. Furthermore, let $\pi\in\cP(\cH)$ be the (unique) invariant probability measure for the Markov semigroup $\{P_t\}_{t\geq0}$ associated with the SEE \eqref{eq_SEE}. Then, for any given $T\in(0,\infty)$, we have
\begin{equation}\label{eq_approximation-IPM-path}
	\lim_{k\to\infty}\sup_{\pi_k\in\Pi_k}\bW_{\Lambda_T}\left(\bfP^\pi,\bfP_k^{\pi_k}\circ\Upsilon_k^{-1}\right)=0.
\end{equation}
\end{theo}

%% Remark

\begin{rem}\label{rem_approximation-IPM}
\begin{itemize}
\item[(i)]
Recall the definition \eqref{eq_Lambda-metric} of the (bounded) metric $d_\Lambda:\Lambda\times\Lambda\to[0,1]$ on the Polish space $\Lambda$. Denote by $\bW_{\Lambda}:\cP(\Lambda)\times\cP(\Lambda)\to[0,1]$ the associated Wasserstein metric, that is,
\begin{equation*}
	\bW_{\Lambda}(\fp_1,\fp_2):=\inf_{\fp\in\sC(\fp_1,\fp_2)}\int_{\Lambda\times\Lambda}d_\Lambda(\eta_1,\eta_2)\,\fp(\diff\eta_1,\diff\eta_2)
\end{equation*}
for $\fp_1,\fp_2\in\cP(\Lambda)$. The convergence result \eqref{eq_approximation-IPM-path} with arbitrary $T\in(0,\infty)$ immediately implies that
\begin{equation*}
	\lim_{k\to\infty}\sup_{\pi_k\in\Pi_k}\bW_{\Lambda}\left(\bfP^\pi,\bfP_k^{\pi_k}\circ\Upsilon_k^{-1}\right)=0,
\end{equation*}
which is equivalent to the (uniform) weak convergence of $\bfP_k^{\pi_k}\circ\Upsilon_k^{-1}$ to $\bfP^\pi$ in $\cP(\Lambda)$. The result \eqref{eq_approximation-IPM-path} incorporates the weak convergence with a kind of uniform integrability of the convergent sequence.
\item[(ii)]
Since the evaluation map $\Lambda\ni\eta=(\eta_t)_{t\geq0}\mapsto\eta_0\in\cH$ is Lipschitz continuous with respect to the semi-norm $\|\cdot\|_{\Lambda_T}$, the convergence result \eqref{eq_approximation-IPM-path} for the stationary laws on the path space $\Lambda$ implies the following convergence result for the invariant probability measures with respect to the $L^1$-Wasserstein (extended) metric $\bW_\cH:\cP(\cH)\times\cP(\cH)\to[0,\infty]$ defined by \eqref{eq_Wasserstein-defi}:
\begin{equation}\label{eq_approximation-IPM-1}
	\lim_{k\to\infty}\sup_{\pi_k\in\Pi_k}\bW_\cH\big(\pi,\pi_k\circ\Upsilon_k^{-1}\big)=0,
\end{equation}
%or equivalently,
%\begin{equation}\label{eq_approximation-IPM-1-dual}
%	\lim_{k\to\infty}\sup_{\pi_k\in\Pi_k}\sup_{f\in\Lip(\cH)}\left|\int_\cH f(y)\,\pi(\diff y)-\int_\cH f(y)\,\big(\pi_k\circ\Upsilon_k^{-1}\big)(\diff y)\right|=0.
%\end{equation}
which furthermore implies the (uniform) weak convergence of $\pi_k\circ\Upsilon_k^{-1}$ to $\pi$ in $\cP(\cH)$. However, our result does \emph{not} imply the weak convergence of $\pi_k\circ\Upsilon_k^{-1}$ to $\pi$ in the topology of $\cP(\cV)$, although they can be seen as probability measures on $\cV$. This is because the evaluation map is not necessarily continuous in $\cV$; it is continuous only in $\cH$ in general.
\end{itemize}
\end{rem}

Before going into the technical details, let us outline the main idea of the proof of \cref{theo_approximation-IPM}. As a first step, we will show the following convergence result with respect to the invariant probability measures:
\begin{equation}\label{eq_approximation-IPM-d}
	\lim_{k\to\infty}\sup_{\pi_k\in\Pi_k}\bW_{d_{\Phi,\Psi}}\big(\pi,\pi_k\circ\Upsilon_k^{-1}\big)=0,
\end{equation}
where the distance-like function $d_{\Phi,\Psi}:\cH\times\cH\to[0,\infty)$ is defined by \eqref{eq_distance-like} in \cref{theo_main} (i). Recalling the standard inequalities in \eqref{eq_H-distance-like}, we see that \eqref{eq_approximation-IPM-d} implies the $L^1$-Wasserstein-convergence \eqref{eq_approximation-IPM-1} for the invariant probability measures. In the next step, combining \eqref{eq_approximation-IPM-1} with the stability estimate \eqref{eq_approximation2}, we will show the convergence \eqref{eq_approximation-IPM-path} of the stationary laws on the path space $\Lambda$ with respect to the semi-norm $\|\cdot\|_{\Lambda_T}$ for any $T\in(0,\infty)$. The most important step in the proof is to derive \eqref{eq_approximation-IPM-d}. Here, the ``spectral gap'' result for the Markov semigroup $\{P_t\}_{t\geq0}$ with respect to the distance-like function $d_{\Phi,\Psi}$ in \cref{theo_main} (i) plays a crucial role. In this step, we borrow the idea from \cite[Section 4.1]{HaMaSc11}.

%% Proof

\begin{proof}[Proof of \cref{theo_approximation-IPM}]
First, we prove the convergence result \eqref{eq_approximation-IPM-d} for the invariant probability measures. Note that
\begin{equation*}
	\sup_{k\geq k_0}\sup_{\pi_k\in\Pi_k}\bW_{d_{\Phi,\Psi}}\big(\pi,\pi_k\circ\Upsilon_k^{-1}\big)<\infty,
\end{equation*}
which follows from the definition \eqref{eq_distance-like} of the distance-like function $d_{\Phi,\Psi}$ and the estimates \eqref{eq_IPM-integrability} and \eqref{eq_approximation-IPM-integrability}. Furthermore, analogously to the proof of \cite[Lemma 4.14]{HaMaSc11}, we can easily check that there exists a constant $C_1>0$ such that
\begin{equation*}
	d_{\Phi,\Psi}(y_1,y_3)\leq C_1\big\{d_{\Phi,\Psi}(y_1,y_2)+d_{\Phi,\Psi}(y_2,y_3)\big\}
\end{equation*}
for any $y_1,y_2,y_3\in\cH$. This implies that
\begin{equation}\label{eq_triangle-like}
	\bW_{d_{\Phi,\Psi}}(\nu_1,\nu_3)\leq C_1\big\{\bW_{d_{\Phi,\Psi}}(\nu_1,\nu_2)+\bW_{d_{\Phi,\Psi}}(\nu_2,\nu_3)\big\}
\end{equation}
for any $\nu_1,\nu_2,\nu_3\in\cP(\cH)$.

Let $k\in\bN$ with $k\geq k_0$ and $\pi_k\in\Pi_k$ be fixed. Let $t_1:=t_0\vee\frac{\log(2C_1)}{r}$, where $r>0$ and $t_0>0$ are the constants arising in \cref{theo_main} (i), and $C_1>0$ is the constant appearing in \eqref{eq_triangle-like}. Then, we have
\begin{align*}
	\bW_{d_{\Phi,\Psi}}\big(\pi,\pi_k\circ\Upsilon_k^{-1}\big)&=\bW_{d_{\Phi,\Psi}}\big(P_{t_1}^*\pi,\big(P_{k,t_1}^*\pi_k\big)\circ\Upsilon_k^{-1}\big)\\
	&\leq C_1\bW_{d_{\Phi,\Psi}}\big(P_{t_1}^*\pi,P_{t_1}^*\big(\pi_k\circ\Upsilon_k^{-1}\big)\big)+C_1\bW_{d_{\Phi,\Psi}}\big(P_{t_1}^*\big(\pi_k\circ\Upsilon_k^{-1}\big),\big(P_{k,t_1}^*\pi_k\big)\circ\Upsilon_k^{-1}\big)\\
	&\leq C_1e^{-rt_1}\bW_{d_{\Phi,\Psi}}\big(\pi,\pi_k\circ\Upsilon_k^{-1}\big)+C_1\bW_{d_{\Phi,\Psi}}\big(P_{t_1}^*\big(\pi_k\circ\Upsilon_k^{-1}\big),\big(P_{k,t_1}^*\pi_k\big)\circ\Upsilon_k^{-1}\big),
\end{align*}
where we used the invariance of $\pi$ and $\pi_k$ with respect to $P_{t_1}$ and $P_{k,t_1}$, respectively, in the first line, the estimate \eqref{eq_triangle-like} in the second line, and the spectral gap estimate \eqref{eq_spectral-gap} in the third line. Since $\bW_{d_{\Phi,\Psi}}(\pi,\pi_k\circ\Upsilon_k^{-1})<\infty$ and $C_1e^{-rt_1}\leq\frac{1}{2}$, we obtain
\begin{equation}\label{eq_estimate-approximation-1}
	\bW_{d_{\Phi,\Psi}}\big(\pi,\pi_k\circ\Upsilon_k^{-1}\big)\leq 2C_1\bW_{d_{\Phi,\Psi}}\big(P_{t_1}^*\big(\pi_k\circ\Upsilon_k^{-1}\big),\big(P_{k,t_1}^*\pi_k\big)\circ\Upsilon_k^{-1}\big).
\end{equation}

Now we estimate the right-hand side of \eqref{eq_estimate-approximation-1}. Let $(\Omega,\cF,\bP)$ be a complete probability space which supports an $(\bR^n)^{I_k}$-valued random variable $Z_{k,0}$ distributed according to $\pi_k$ and a $d$-dimensional Brownian motion $W$ independent of $Z_{k,0}$. Let $\bF=(\cF_t)_{t\geq0}$ be the augmentation of the filtration generated by $Z_{k,0}$ and $W$. Let $Y=(Y_t)_{t\geq0}$ be the solution to the SEE \eqref{eq_SEE} with the $\cH$-valued initial condition $Y_0=\Upsilon_kZ_{k,0}$, and let $Z_k=(Z_{k,t})_{t\geq0}$ be the solution to the SDE \eqref{eq_approximating-SDE} with the $(\bR^n)^{I_k}$-valued initial condition $Z_{k,0}$, both of which are driven by the same Brownian motion $W$. Then, the joint distribution of the pair of random variables $(Y_{t_1},\Upsilon_kZ_{k,t_1})$ on $\cH\times\cH$ under $\bP$ constitutes a coupling between the probability measures $P_{t_1}^*(\pi_k\circ\Upsilon_k^{-1})\in\cP(\cH)$ and $(P_{k,t_1}^*\pi_k)\circ\Upsilon_k^{-1}\in\cP(\cH)$. Hence, we have
\begin{equation*}
	\bW_{d_{\Phi,\Psi}}\big(P_{t_1}^*\big(\pi_k\circ\Upsilon_k^{-1}\big),\big(P_{k,t_1}^*\pi_k\big)\circ\Upsilon_k^{-1}\big)\leq\bE\big[d_{\Phi,\Psi}\left(Y_{t_1},\Upsilon_kZ_{k,t_1}\right)\big].
\end{equation*}
Recalling the definition \eqref{eq_distance-like} of the distance-like function $d_{\Phi,\Psi}$, by the above estimate and the Cauchy--Schwarz inequality, we have
\begin{align*}
	&\bW_{d_{\Phi,\Psi}}\big(P_{t_1}^*\big(\pi_k\circ\Upsilon_k^{-1}\big),\big(P_{k,t_1}^*\pi_k\big)\circ\Upsilon_k^{-1}\big)\\
	&\leq\bE\left[\big\|Y_{t_1}-\Upsilon_kZ_{k,t_1}\big\|_\Phi\wedge1\right]^{1/2}\left\{1+\bE\left[\big\|Y_{t_1}\big\|_{\Psi}^2\right]+\bE\left[\big\|\Upsilon_kZ_{k,t_1}\big\|_{\Psi}^2\right]\right\}^{1/2}.
\end{align*}
Recall that $\|\cdot\|_\Phi$ and $\|\cdot\|_{\Psi}$ are norms on $\cH$ and equivalent to $\|\cdot\|_\cH$. Since $Y_0=\Upsilon_kZ_{k,0}$ and $\Law_\bP(Z_{k,0})=\pi_k\in\cP((\bR^n)^{I_k})$, by using the estimates \eqref{eq_SEE-apriori}, \eqref{eq_approximation1} and \eqref{eq_approximation2}, we obtain
\begin{align}
	\nonumber
	\bW_{d_{\Phi,\Psi}}\big(P_{t_1}^*\big(\pi_k\circ\Upsilon_k^{-1}\big),\big(P_{k,t_1}^*\pi_k\big)\circ\Upsilon_k^{-1}\big)&\leq C_2\ep_k^{1/2}\left(1+\bE\left[\big\|\Upsilon_kZ_{k,0}\big\|_\cH^2\right]\right)^{3/4}\\
	\label{eq_estimate-approximation-2}
	&=C_2\ep_k^{1/2}\left(1+\int_{(\bR^n)^{I_k}}\big\|\Upsilon_kz_k\big\|_\cH^2\,\pi_k(\diff z_k)\right)^{3/4},
\end{align}
where $C_2>0$ is a constant which does not depend on $k$ or $\pi_k$.

By \eqref{eq_estimate-approximation-1} and \eqref{eq_estimate-approximation-2}, together with \eqref{eq_approximation-IPM-integrability}, we obtain
\begin{equation*}
	\sup_{\pi_k\in\Pi_k}\bW_{d_{\Phi,\Psi}}\left(\pi,\pi_k\circ\Upsilon_k^{-1}\right)\leq C_3\ep_k^{1/2}
\end{equation*}
for all $k\in\bN$ with $k\geq k_0$, where $C_3>0$ is a constant independent of $k$. Since $\lim_{k\to\infty}\ep_k=0$, we obtain \eqref{eq_approximation-IPM-d}. Recalling the standard estimate \eqref{eq_H-distance-like}, the convergence \eqref{eq_approximation-IPM-d} for the distance-like function $d_{\Phi,\Psi}$ implies the $L^1$-Wasserstein-convergence \eqref{eq_approximation-IPM-1} for the invariant probability measures.

Fix an arbitrary $T\in(0,\infty)$. Our next purpose is to prove \eqref{eq_approximation-IPM-path}. To this end, let $k\geq k_0$, $\pi_k\in\Pi_k$ and $\nu_k\in\sC(\pi,\pi_k\circ\Upsilon_k^{-1})$ be fixed. Define $\Gamma_k:\cH\times(\bR^n)^{I_k}\to\cH\times\cH$ by $\Gamma_k(y,z_k):=(y,\Upsilon_kz_k)$ for $(y,z_k)\in\cH\times(\bR^n)^{I_k}$. Clearly, the map $\Gamma_k$ is continuous and injective. Hence, by the Lusin--Suslin theorem (see, e.g., \cite[Theorem 15.1]{Ke95}), we obtain $\Gamma_k(\Xi)\in\cB(\cH\times\cH)$ for any $\Xi\in\cB(\cH\times(\bR^n)^{I_k})$. Consequently we can define a set function $\widetilde{\nu}_k:\cB(\cH\times(\bR^n)^{I_k})\to[0,1]$ by $\widetilde{\nu}_k(\Xi):=\nu_k(\Gamma_k(\Xi))$ for $\Xi\in\cB(\cH\times(\bR^n)^{I_k})$. Again by the injectivity of $\Gamma_k$, we see that the set function $\widetilde{\nu}_k$ is a $\sigma$-additive measure. Furthermore, since
\begin{equation*}
	\widetilde{\nu}_k(\cH\times(\bR^n)^{I_k})=\nu_k(\cH\times\Upsilon_k((\bR^n)^{I_k}))=\pi_k(\Upsilon_k^{-1}(\Upsilon_k((\bR^n)^{I_k})))=\pi_k((\bR^n)^{I_k})=1,
\end{equation*}
we have $\widetilde{\nu}_k\in\cP(\cH\times(\bR^n)^{I_k})$. Observe that
\begin{align}
	\label{eq_tilde-Q-1}
	&\widetilde{\nu}_k(A\times(\bR^n)^{I_k})=\nu_k(A\times\Upsilon_k((\bR^n)^{I_k}))=\nu_k(A\times\cH)=\pi(A)
\ \ \text{for any $A\in\cB(\cH)$},\\
	\label{eq_tilde-Q-2}
	&\widetilde{\nu}_k(\cH\times B)=\nu_k(\cH\times\Upsilon_k(B))=\pi_k(\Upsilon_k^{-1}(\Upsilon_k(B)))=\pi_k(B)\ \ \text{for any $B\in\cB\big((\bR^n)^{I_k}\big)$, and}\\
	\label{eq_tilde-Q-3}
	&\widetilde{\nu}_k(\Gamma_k^{-1}(C))=\nu_k(\Gamma_k(\Gamma_k^{-1}(C)))=\nu_k(C\cap(\cH\times\Upsilon_k((\bR^n)^{I_k})))=\nu_k(C)\ \ \text{for any $C\in\cB(\cH\times\cH)$}.
\end{align}

Consider a complete probability space $(\Omega,\cF,\bP)$ which supports an $\cH\times(\bR^n)^{I_k}$-valued random variable $(Y_0,Z_{k,0})$ distributed according to $\widetilde{\nu}_k\in\cP(\cH\times(\bR^n)^{I_k})$ and a $d$-dimensional Brownian motion $W$ independent of $(Y_0,Z_{k,0})$. We equip $(\Omega,\cF,\bP)$ with the augmented filtration $\bF$ generated by $(Y_0,Z_{k,0})$ and $W$. Under this setting, let $Y=(Y_t)_{t\geq0}$ be the solution to the SEE \eqref{eq_SEE} with the prescribed $\cH$-valued initial condition $Y_0$, and let $Z_k=(Z_{k,t})_{t\geq0}$ be the solution to the approximating SDE \eqref{eq_approximating-SDE} with the prescribed $(\bR^n)^{I_k}$-valued initial condition $Z_{k,0}$. By \eqref{eq_tilde-Q-1} and \eqref{eq_tilde-Q-2}, the joint distribution of the paths of $(Y_t)_{t\geq0}$ and $(\Upsilon_kZ_{k,t})_{t\geq0}$ on $\Lambda\times\Lambda$ constitutes a coupling between the probability measures $\bfP^\pi\in\cP(\Lambda)$ and $\bfP^{\pi_k}_k\circ\Upsilon_k^{-1}\in\cP(\Lambda)$. Thus, we have
\begin{align*}
	&\bW_{\Lambda_T}\big(\bfP^\pi,\bfP^{\pi_k}_k\circ\Upsilon_k^{-1}\big)\\
	&\leq\bE\left[\left(\sup_{t\in[0,T]}\big\|Y_t-\Upsilon_kZ_{k,t}\big\|_\cH^2+\int^T_0\big\|Y_t-\Upsilon_kZ_{k,t}\big\|_\cV^2\,\diff t\right)^{1/2}\right]\\
	&\leq\bE\left[C_0^{1/2}e^{C_0T/2}\big\|Y_0-\Upsilon_kZ_{k,0}\big\|_\cH+C_0^{1/2}e^{C_0T/2}\ep_k\left(1+\big\|\Upsilon_kZ_{k,0}\big\|_\cH\right)\right]\\
	&=\int_{\cH\times(\bR^n)^{I_k}}\left\{C_0^{1/2}e^{C_0T/2}\big\|y-\Upsilon_kz_k\big\|_\cH+C_0^{1/2}e^{C_0T/2}\ep_k\left(1+\big\|\Upsilon_kz_k\big\|_\cH\right)\right\}\,\widetilde{\nu}_k(\diff y,\diff z_k)\\
	&=C_0^{1/2}e^{C_0T/2}\int_{\cH\times\cH}\big\|y-y'\big\|_\cH\,\nu_k(\diff y,\diff y')+C_0^{1/2}e^{C_0T/2}\ep_k\left(1+\int_{(\bR^n)^{I_k}}\big\|\Upsilon_kz_k\big\|_\cH\,\pi_k(\diff z_k)\right),
\end{align*}
where we used the estimate \eqref{eq_approximation2} and the conditional Jensen's inequality in the second inequality and the relations \eqref{eq_tilde-Q-2} and \eqref{eq_tilde-Q-3} in the final equality. Taking the infimum over $\nu_k\in\sC(\pi,\pi_k\circ\Upsilon_k^{-1})$ and then the supremum over $\pi_k\in\Pi_k$, we obtain
\begin{align*}
	&\sup_{\pi_k\in\Pi_k}\bW_{\Lambda_T}\big(\bfP^\pi,\bfP^{\pi_k}_k\circ\Upsilon_k^{-1}\big)\\
	&\leq C_0^{1/2}e^{C_0T/2}\sup_{\pi_k\in\Pi_k}\bW_\cH\big(\pi,\pi_k\circ\Upsilon_k^{-1}\big)+C_0e^{C_0T/2}\ep_k\left(1+\sup_{\pi_k\in\Pi_k}\int_{(\bR^n)^{I_k}}\big\|\Upsilon_kz_k\big\|_\cH\,\pi_k(\diff z_k)\right).
\end{align*}
By \eqref{eq_approximation-IPM-1} and \eqref{eq_approximation-IPM-integrability}, we see that the right-hand side above tends to zero as $k\to\infty$, and hence \eqref{eq_approximation-IPM-path} holds. This completes the proof.
\end{proof}

%% Remark

\begin{rem}
The proof of \cref{theo_approximation-IPM} also yields a quantitative estimate for the convergence \eqref{eq_approximation-IPM-path}. Specifically, we obtain a convergence order of $O(\ep_k^{1/2})$, where $\ep_k$ is defined by \eqref{eq_ep_k}. The exponent $1/2$ stems from the square root in the definition \eqref{eq_distance-like} of the distance-like function $d_{\Phi,\Psi}$. Note, however, that this order is not necessarily optimal.
\end{rem}

An important consequence of the convergence result \eqref{eq_approximation-IPM-path} for the stationary laws on the path space $\Lambda$, rather than merely the convergence \eqref{eq_approximation-IPM-1} for the invariant probability measures, is that it enables us to derive the corresponding convergence for the original SVE \eqref{eq_SVE}, as we will show in \cref{cor_approximation-IPM} below. Recall the observations in \cref{rem_path-distribution} (ii), \cref{cor_stationary-path} (ii) and \cref{rem_Q-tilde}, where the stationary laws $\bfQ^\pi\in\cP(L^2_\loc(0,\infty;\bR^n))$ and $\widetilde{\bfQ}^\pi\in\cP((\bR^n)^{[0,\infty)})$ for the original SVE \eqref{eq_SVE} were constructed from $\bfP^\pi\in\cP(\Lambda)$; the former characterizes the distribution of the $\diff t$-equivalence class of the stationary solution, while the latter characterizes the finite-dimensional distributions $\widetilde{\bfQ}^\pi_{(t_1,\dots,t_\ell)}\in\cP((\bR^n)^\ell)$ of the stationary solution itself. Recall also that $\bfQ^\pi$ and $\widetilde{\bfQ}^\pi$ are independent of the choice of the lifting basis $(\mu,M_b,M_\sigma)$ in the sense of \cref{theo_lift-independence}.

To establish the corresponding results on approximations of $\bfQ^\pi$ and $\widetilde{\bfQ}^\pi$ by means of the finite-dimensional SDEs \eqref{eq_approximating-SDE}, let us introduce the following additional notations. As in \cref{rem_path-distribution} (ii), for each $k\in\bN$ and $\nu_k\in\cP((\bR^n)^{I_k})$, we set $\bfQ_k^{\nu_k}:=(\bfP_k^{\nu_k}\circ\Upsilon_k^{-1})\circ\mu[\cdot]^{-1}\in\cP(L^2_\loc(0,\infty;\bR^n))$, where $\mu[\cdot]$ is understood as a continuous map from $\Lambda$ to $L^2_\loc(0,\infty;\bR^n)$. This represents the law of the $\diff t$-equivalence class of the $\bR^n$-valued process $X_{k,t}:=\sum^{I_k}_{i=1}\mu(\Delta_k^{(i)})Z_{k,t}^{(i)}$, $t\geq0$, where the $(\bR^n)^{I_k}$-valued process $Z_k=(Z_{k,t})_{t\geq0}$ is a solution to the approximating SDE \eqref{eq_approximating-SDE} with the initial distribution $\nu_k$. Recall that $X_k=(X_{k,t})_{t\geq0}$ solves the approximating SVE \eqref{eq_approximating-SVE} with sum-of-exponentials type kernels; see \cref{rem_approximating-SVE}. Furthermore, we denote by $\widetilde{\bfQ}_k^{\nu_k}\in\cP((\bR^n)^{[0,\infty)})$ the law of $X_k=(X_{k,t})_{t\geq0}$ (not its $\diff t$-equivalence class). This probability measure is consistent to the family of the finite-dimensional distributions $\widetilde{\bfQ}_{k,(t_1,\dots,t_\ell)}^{\nu_k}\in\cP((\bR^n)^\ell)$ with $t_1,\dots,t_\ell\geq0$ and $\ell\in\bN$, where
\begin{equation*}
	\widetilde{\bfQ}_{k,(t_1,\dots,t_\ell)}^{\nu_k}(A_1\times\cdots\times A_\ell)=\bfP_k^{\nu_k}\left(\left\{\zeta_k=(\zeta_{k,t})_{t\geq0}\in C([0,\infty);(\bR^n)^{I_k})\relmiddle|\mu[\Upsilon_k\zeta_{k,t_1}]\in A_1,\dots,\mu[\Upsilon_k\zeta_{k,t_\ell}]\in A_\ell\right\}\right)
\end{equation*}
for any $A_1,\dots,A_\ell\in\cB(\bR^n)$. These notations are comparable to the ones introduced in \cref{rem_Q-tilde}. Note that, given an invariant probability measure $\pi_k\in\cP((\bR^n)^{I_k})$ for the Markov semigroup $\{P_{k,t}\}_{t\geq0}$, the laws $\bfQ_k^{\pi_k}\in\cP(L^2_\loc(0,\infty;\bR^n))$ and $\widetilde{\bfQ}_k^{\pi_k}\in\cP((\bR^n)^{[0,\infty)})$ are invariant under the time-shifts on $L^2_\loc(0,\infty;\bR^n)$ and $(\bR^n)^{[0,\infty)}$, respectively.

Similarly to \eqref{eq_Wasserstein-Lambda-T}, for each $T\in(0,\infty)$, we define an extended pseudo-metric $\bW_{L^2(0,T;\bR^n)}:\cP(L^2_\loc(0,\infty;\bR^n))\times\cP(L^2_\loc(0,\infty;\bR^n))\to[0,\infty]$ on $\cP(L^2_\loc(0,\infty;\bR^n))$ by
\begin{equation*}
	\bW_{L^2(0,T;\bR^n)}(\fq_1,\fq_2):=\inf_{\fq\in\sC(\fq_1,\fq_2)}\int_{L^2_\loc(0,\infty;\bR^n)\times L^2_\loc(0,\infty;\bR^n)}\big\|\xi_1-\xi_2\big\|_{L^2(0,T;\bR^n)}\,\fq(\diff\xi_1,\diff\xi_2)
\end{equation*}
for $\fq_1,\fq_2\in\cP(L^2_\loc(0,\infty;\bR^n))$. Furthermore, for each $\ell\in\bN$, we define an extended metric $\bW_{(\bR^n)^\ell}:\cP((\bR^n)^\ell)\times\cP((\bR^n)^\ell)\to[0,\infty]$ on $\cP((\bR^n)^\ell)$ by
\begin{equation*}
	\bW_{(\bR^n)^\ell}\big(\widetilde{\fq}_1,\widetilde{\fq}_2\big):=\inf_{\widetilde{\fq}\in\sC(\widetilde{\fq}_1,\widetilde{\fq}_2)}\int_{(\bR^n)^\ell\times(\bR^n)^\ell}\max_{j\in\{1,\dots,\ell\}}\big|x_{1,j}-x_{2,j}\big|\,\widetilde{\fq}\big(\diff(x_{1,j})^\ell_{j=1},\diff(x_{2,j})^\ell_{j=1}\big)
\end{equation*}
for $\widetilde{\fq}_1,\widetilde{\fq}_2\in\cP((\bR^n)^\ell)$, where $|\cdot|$ denotes the standard Euclidean norm on $\bR^n$.

Now we are ready to state a corollary of \cref{theo_approximation-IPM}.

%% Corollary

\begin{cor}\label{cor_approximation-IPM}
Under the same setting of \cref{theo_approximation-IPM}, for any $T\in(0,\infty)$ and $\ell\in\bN$, we have
\begin{equation}\label{eq_approximation-IPM-Q-L^2}
	\lim_{k\to\infty}\sup_{\pi_k\in\Pi_k}\bW_{L^2(0,T;\bR^n)}\big(\bfQ^\pi,\bfQ^{\pi_k}_k\big)=0,
\end{equation}
and
\begin{equation}\label{eq_approximation-IPM-Q-finite}
	\lim_{k\to\infty}\sup_{\pi_k\in\Pi_k}\sup_{t_1,\dots,t_\ell\in[0,T]}\bW_{(\bR^n)^\ell}\left(\widetilde{\bfQ}^\pi_{(t_1,\dots,t_\ell)},\widetilde{\bfQ}^{\pi_k}_{k,(t_1,\dots,t_\ell)}\right)=0.
\end{equation}
\end{cor}

%% Remark

\begin{rem}
The convergence result \eqref{eq_approximation-IPM-Q-L^2} for the laws of the $\diff t$-equivalence classes is an immediate consequence of \eqref{eq_approximation-IPM-path}, by virtue of the continuity of the map $\mu[\cdot]:\Lambda\to L^2_\loc(0,\infty;\bR^n)$, as demonstrated in the first part of the proof below. In contrast, the convergence result \eqref{eq_approximation-IPM-Q-finite} for the finite-dimensional distributions is less trivial, as the pointwise evaluation $\eta\mapsto(\mu[\eta_{t_1}\1_\cV(\eta_{t_1})],\dots,\mu[\eta_{t_\ell}\1_\cV(\eta_{t_\ell})])$ is not necessarily continuous from $\Lambda$ to $(\bR^n)^{\ell}$, and \eqref{eq_approximation-IPM-path} does not necessarily imply weak convergence in $\cP(\cV)$ for a fixed time parameter; see also \cref{rem_approximation-IPM} (ii). Nevertheless, by leveraging the property of the map $\mu[\cdot]:\cV\to\bR^n$ in \cref{lemm_mu-ep}, combined with the integrability results \eqref{eq_IPM-integrability} and \eqref{eq_approximation-IPM-integrability} for the invariant probability measures with respect to the $\cV$-norm, we can derive \eqref{eq_approximation-IPM-Q-finite} from \eqref{eq_approximation-IPM-path}. This derivation is detailed in the second part of the proof.
\end{rem}

%% Proof

\begin{proof}[Proof of \cref{cor_approximation-IPM}]
The convergence \eqref{eq_approximation-IPM-Q-L^2} immediately follows from \eqref{eq_approximation-IPM-path} and the following estimate:
\begin{equation*}
	\|\mu[\eta]\|_{L^2(0,T;\bR^n)}\leq\|\mu[\cdot]\|_{L(\cV;\bR^n)}\left(\int^T_0\|\eta_t\|_\cV^2\,\diff t\right)^{1/2}\leq\|\mu[\cdot]\|_{L(\cV;\bR^n)}\|\eta\|_{\Lambda_T}
\end{equation*}
for any $\eta\in\Lambda$. Indeed, for any $k\geq k_0$, $\pi_k\in\Pi_k$ and $\fp_k\in\sC(\bfP^\pi,\bfP^{\pi_k}_k\circ\Upsilon_k^{-1})$, notice that the probability measure $\fq_k\in\cP(L^2_\loc(0,\infty;\bR^n)\times L^2_\loc(0,\infty;\bR^n))$ given by
\begin{equation*}
	\fq_k(D):=\fp_k\left(\left\{(\eta,\eta')\in\Lambda\times\Lambda\relmiddle|\big(\mu[\eta],\mu[\eta']\big)\in D\right\}\right),\ \ D\in \cB\big(L^2_\loc(0,\infty;\bR^n)\times L^2_\loc(0,\infty;\bR^n)\big),
\end{equation*}
constitutes a coupling between the probability measures $\bfQ^\pi\in\cP(L^2_\loc(0,\infty;\bR^n))$ and $\bfQ_k^{\pi_k}\in\cP(L^2_\loc(0,\infty;\bR^n))$. Hence,
\begin{align*}
	\bW_{L^2(0,T;\bR^n)}\big(\bfQ^\pi,\bfQ_k^{\pi_k}\big)&\leq\int_{L^2_\loc(0,\infty;\bR^n)\times L^2_\loc(0,\infty;\bR^n)}\big\|\xi-\xi'\big\|_{L^2(0,T;\bR^n)}\,\fq_k(\diff\xi,\diff\xi')\\
	&=\int_{\Lambda\times\Lambda}\big\|\mu[\eta-\eta']\big\|_{L^2(0,T;\bR^n)}\,\fp_k(\diff\eta,\diff\eta')\\
	&\leq\|\mu[\cdot]\|_{L(\cV;\bR^n)}\int_{\Lambda\times\Lambda}\|\eta-\eta'\|_{\Lambda_T}\,\fp_k(\diff\eta,\diff\eta').
\end{align*}
Taking the infimum over $\fp_k\in\sC(\bfP^\pi,\bfP^{\pi_k}_k\circ\Upsilon_k^{-1})$ and then the supremum over $\pi_k\in\Pi_k$, we obtain
\begin{equation*}
	\sup_{\pi_k\in\Pi_k}\bW_{L^2(0,T;\bR^n)}\big(\bfQ^\pi,\bfQ_k^{\pi_k}\big)\leq\|\mu[\cdot]\|_{L(\cV;\bR^n)}\sup_{\pi_k\in\Pi_k}\bW_{\Lambda_T}\big(\bfP^\pi,\bfP^{\pi_k}_k\circ\Upsilon_k^{-1}\big).
\end{equation*}
Since $\|\mu[\cdot]\|_{L(\cV;\bR^n)}<\infty$, by \eqref{eq_approximation-IPM-path}, we obtain \eqref{eq_approximation-IPM-Q-L^2}.

To prove the convergence \eqref{eq_approximation-IPM-Q-finite} of finite-dimensional distributions, let $k\geq k_0$, $\pi_k\in\Pi_k$, $t_1,\dots,t_\ell\in[0,T]$ and $\fp_k\in\sC(\bfP^\pi,\bfP^{\pi_k}_k\circ\Upsilon_k^{-1})$ be fixed. Noting \cref{rem_Q-tilde} and $\Upsilon_k:(\bR^n)^{I_k}\to\cV$, we see that
\begin{equation*}
	\fp_k\left(\left\{(\eta,\eta')\in\Lambda\times\Lambda\relmiddle|\eta_{t_1},\dots,\eta_{t_\ell},\eta'_{t_1},\dots,\eta'_{t_\ell}\in\cV\right\}\right)=1.
\end{equation*}
Define $\widetilde{\fq}_k\in\cP((\bR^n)^\ell\times(\bR^n)^\ell)$ by
\begin{equation*}
	\widetilde{\fq}_k(E):=\fp_k\left(\left\{(\eta,\eta')\in\Lambda\times\Lambda\relmiddle|
	\begin{aligned}
	&\eta_{t_1},\dots,\eta_{t_\ell},\eta'_{t_1},\dots,\eta'_{t_\ell}\in\cV\ \text{and}\\
	&\left(\big(\mu\big[\eta_{t_j}\big]\big)^\ell_{j=1},\big(\mu\big[\eta'_{t_j}\big]\big)^\ell_{j=1}\right)\in E
	\end{aligned}
	\right\}\right),\ \ E\in\cB\big((\bR^n)^\ell\times(\bR^n)^\ell\big).
\end{equation*}
Then, we see that $\widetilde{\fq}_k$ is a coupling between the probability measures $\widetilde{\bfQ}^\pi_{(t_1,\dots,t_\ell)}\in\cP((\bR^n)^\ell)$ and $\widetilde{\bfQ}^{\pi_k}_{k,(t_1,\dots,t_\ell)}\in\cP((\bR^n)^\ell)$. Hence, we have
\begin{align*}
	\bW_{(\bR^n)^\ell}\left(\widetilde{\bfQ}^\pi_{(t_1,\dots,t_\ell)},\widetilde{\bfQ}^{\pi_k}_{k,(t_1,\dots,t_\ell)}\right)&\leq\int_{(\bR^n)^\ell\times(\bR^n)^\ell}\max_{j\in\{1,\dots,\ell\}}\big|x_j-x'_j\big|\,\widetilde{\fq}_k\big(\diff(x_j)^\ell_{j=1},\diff(x'_j)^\ell_{j=1}\big)\\
	&=\int_{\Lambda\times\Lambda}\max_{j\in\{1,\dots,\ell\}}\big|\mu\big[\eta_{t_j}-\eta'_{t_j}\big]\big|\1_{\cV^{2\ell}}\big(\eta_{t_1},\dots,\eta_{t_\ell},\eta'_{t_1},\dots,\eta'_{t_\ell}\big)\,\fp_k(\diff\eta,\diff\eta').
\end{align*}
Fix an arbitrary number $\ep>0$, and let $C_{\mu,\ep}>0$ be the constant arising in \cref{lemm_mu-ep}. Then, we have
\begin{align*}
	\bW_{(\bR^n)^\ell}\left(\widetilde{\bfQ}^\pi_{(t_1,\dots,t_\ell)},\widetilde{\bfQ}^{\pi_k}_{k,(t_1,\dots,t_\ell)}\right)&\leq\int_{\Lambda\times\Lambda}\max_{j\in\{1,\dots,\ell\}}\Big(\ep\big\|\eta_{t_j}-\eta'_{t_j}\big\|_\cV^2+C_{\mu,\ep}\big\|\eta_{t_j}-\eta'_{t_j}\big\|_\cH^2\Big)^{1/2}\,\fp_k(\diff\eta,\diff\eta')\\
	&\leq\ep^{1/2}\sum^\ell_{j=1}\left\{\int_{\Lambda}\|\eta_{t_j}\|_\cV\,\bfP^\pi(\diff\eta)+\int_{C([0,\infty);(\bR^n)^{I_k})}\big\|\Upsilon_k\zeta_{k,t_j}\big\|_\cV\,\bfP^{\pi_k}_k(\diff\zeta_k)\right\}\\
	&\hspace{1cm}+C_{\mu,\ep}^{1/2}\int_{\Lambda\times\Lambda}\sup_{t\in[0,T]}\|\eta_t-\eta'_t\|_\cH\,\fp_k(\diff\eta,\diff\eta'),
\end{align*}
where the first inequality follows from \cref{lemm_mu-ep}, while the second is a consequence of the fact that $\fp_k\in\sC(\bfP^\pi,\bfP^{\pi_k}_k\circ\Upsilon_k^{-1})$. By the invariance of $\pi\in\cP(\cH)$ and $\pi_k\in\cP((\bR^n)^{I_k})$ with respect to $\{P_t\}_{t\geq0}$ and $\{P_{k,t}\}_{t\geq0}$, respectively, and the inequality $\sup_{t\in[0,T]}\|\eta_t-\eta'_t\big\|_\cH\leq\|\eta-\eta'\|_{\Lambda_T}$, we obtain
\begin{align*}
	&\bW_{(\bR^n)^\ell}\left(\widetilde{\bfQ}^\pi_{(t_1,\dots,t_\ell)},\widetilde{\bfQ}^{\pi_k}_{k,(t_1,\dots,t_\ell)}\right)\\
	&\leq\ep^{1/2}\ell\left\{\int_\cH\|y\|_\cV\,\pi(\diff y)+\int_{(\bR^n)^{I_k}}\big\|\Upsilon_kz_k\big\|_\cV\,\pi_k(\diff z_k)\right\}+C_{\mu,\ep}^{1/2}\int_{\Lambda\times\Lambda}\|\eta-\eta'\|_{\Lambda_T}\,\fp_k(\diff\eta,\diff\eta').
\end{align*}
Taking the infimum over $\fp_k\in\sC(\bfP^\pi,\bfP^{\pi_k}_k\circ\Upsilon_k^{-1})$ and then the supremum over $t_1,\dots,t_\ell\in[0,T]$ and $\pi_k\in\Pi_k$, we obtain
\begin{align*}
	&\sup_{\pi_k\in\Pi_k}\sup_{t_1,\dots,t_\ell\in[0,T]}\bW_{(\bR^n)^\ell}\left(\widetilde{\bfQ}^\pi_{(t_1,\dots,t_\ell)},\widetilde{\bfQ}^{\pi_k}_{k,(t_1,\dots,t_\ell)}\right)\\
	&\leq\ep^{1/2}\ell\left\{\int_\cH\|y\|_\cV\,\pi(\diff y)+\sup_{\pi_k\in\Pi_k}\int_{(\bR^n)^{I_k}}\big\|\Upsilon_kz_k\big\|_\cV\,\pi_k(\diff z_k)\right\}+C_{\mu,\ep}^{1/2}\sup_{\pi_k\in\Pi_k}\bW_{\Lambda_T}\big(\bfP^\pi,\bfP^{\pi_k}_k\circ\Upsilon_k^{-1}\big).
\end{align*}
By using \eqref{eq_approximation-IPM-path}, we have
\begin{align*}
	&\limsup_{k\to\infty}\sup_{\pi_k\in\Pi_k}\sup_{t_1,\dots,t_\ell\in[0,T]}\bW_{(\bR^n)^\ell}\left(\widetilde{\bfQ}^\pi_{(t_1,\dots,t_\ell)},\widetilde{\bfQ}^{\pi_k}_{k,(t_1,\dots,t_\ell)}\right)\\
	&\leq\ep^{1/2}\ell\left\{\int_\cH\|y\|_\cV\,\pi(\diff y)+\sup_{k\geq k_0}\sup_{\pi_k\in\Pi_k}\int_{(\bR^n)^{I_k}}\big\|\Upsilon_kz_k\big\|_\cV\,\pi_k(\diff z_k)\right\}.
\end{align*}
In view of the integrability conditions \eqref{eq_IPM-integrability} and \eqref{eq_approximation-IPM-integrability}, taking the limit $\ep\downarrow0$ yields the desired assertion \eqref{eq_approximation-IPM-Q-finite}. This completes the proof.
\end{proof}

%%%%%%%%%%%%%%%%%%%%%%%%%%%%%%%%%%
%%%%%%%%%%%%%%%%%%%%%%%%%%%%%%%%%%
%% Appendix
%%%%%%%%%%%%%%%%%%%%%%%%%%%%%%%%%%
%%%%%%%%%%%%%%%%%%%%%%%%%%%%%%%%%%

\appendix
\setcounter{theo}{0}
\setcounter{equation}{0}

\section{Appendix}\label{appendix}

%%%%%%%%%%%%%%
%% Section
%%%%%%%%%%%%%%

In the context of monotone SPDEs defined on a Gelfand triplet $\cV\hookrightarrow\cH\hookrightarrow\cV^*$, it is essential to determine whether the embedding $\cV\hookrightarrow\cH$ is compact. Indeed, while various general theories exist, several fundamental results---including the well-posedness for SPDEs with fully local monotone coefficients \cite{LiRo15,RoShZh24} and the ergodicity derived from ultimate boundedness \cite[Chapter 7]{GaMa10}---require this compactness property. The following lemma provides a characterization of the compactness of the embedding in our framework. It reveals that, in typical cases of interest, the embedding $\cV\hookrightarrow\cH$ is not compact, which means that the aforementioned results cannot be directly applied to our setting.

%% Lemma

\begin{lemm}\label{lemm_app-compact}
For the Banach spaces $\cH$ and $\cV$ defined in \cref{sec_pre}, the embedding $\cV\hookrightarrow\cH$ is compact if and only if $\supp\cap[0,m]$ is a finite set for any $m\in(0,\infty)$.
\end{lemm}

%% Remark

\begin{rem}\label{rem_app-compact}
By virtue of the Lebesgue decomposition of the Borel measure $\mu$ on $[0,\infty)$, the condition that $\supp\cap[0,m]$ is a finite set for every $m\in(0,\infty)$ is equivalent to $\mu$ being of the form $\mu=\sum^\infty_{i=1}c_i\delta_{\theta_i}$ for some $c_i,\theta_i\geq0$ with $\lim_{i\to\infty}\theta_i=\infty$. Specifically, for the lifting basis generating the sum-of-exponentials type kernels in \cref{exam_liftable} (i), the embedding $\cV\hookrightarrow\cH$ is indeed compact. In contrast, in typical cases of interest, such as the tempered fractional kernels in \cref{exam_liftable} (ii), the embedding $\cV\hookrightarrow\cH$ is never compact.
\end{rem}

%% Proof

\begin{proof}[Proof of \cref{lemm_app-compact}]
To prove the ``if part'', suppose that $\supp\cap[0,m]$ is a finite set for any $m\in(0,\infty)$. If $\supp$ itself is a finite set, then the Hilbert space $\cH$ is finite-dimensional, making the embedding $\cV\hookrightarrow\cH$ trivially compact. Otherwise, let $\supp=\{\theta_i\}_{i\in\bN}$ be an infinite set, where $\{\theta_i\}_{i\in\bN}\subset[0,\infty)$ is a strictly increasing sequence such that $\lim_{i\to\infty}\theta_i=\infty$. In this case, the measure $\mu$ is of the form $\mu=\sum^\infty_{i=1}c_i\delta_{\theta_i}$ with $c_i>0$. Let $\overline{B}_\cV:=\{y\in\cV\,|\,\|y\|_\cV\leq1\}$ be the unit ball in $\cV$, and take an arbitrary sequence $\{y_k\}_{k\in\bN}\subset\overline{B}_\cV$. Noting that $\{y_k(\theta_i)\}_{k\in\bN}\subset\bR^n$ is bounded for each $i\in\bN$, by Cantor's diagonal argument, we can find an increasing sequence $\{k_\ell\}_{\ell\in\bN}\subset\bN$ and a sequence $\{e_i\}_{i\in\bN}\subset\bR^n$ such that $\lim_{\ell\to\infty}y_{k_\ell}(\theta_i)=e_i$ in $\bR^n$ for any $i\in\bN$. Define $y:=\sum^\infty_{i=1}e_i\1_{\{\theta_i\}}$. By Fatou's lemma, we have
\begin{equation*}
	\|y\|_\cV^2=\sum^\infty_{i=1}c_i(1+\theta_i)^{1/2}|e_i|^2\leq\liminf_{\ell\to\infty}\sum^\infty_{i=1}c_i(1+\theta_i)^{1/2}|y_{k_\ell}(\theta_i)|^2=\liminf_{\ell\to\infty}\big\|y_{k_\ell}\big\|_\cV^2\leq1,
\end{equation*}
and hence $y\in\overline{B}_\cV$. Moreover, since the sequence $\{\theta_i\}_{i\in\bN}\subset[0,\infty)$ is increasing, we have
\begin{equation*}
	\|y_{k_\ell}-y\|_\cH^2=\sum^\infty_{j=1}c_j(1+\theta_j)^{-1/2}\big|y_{k_\ell}(\theta_j)-e_j\big|^2\leq\sum^i_{j=1}c_j(1+\theta_j)^{-1/2}\big|y_{k_\ell}(\theta_j)-e_j\big|^2+(1+\theta_{i+1})^{-1}\|y_{k_\ell}-y\|_\cV^2
\end{equation*}
for any $\ell\in\bN$ and $i\in\bN$. For each $i\in\bN$, the first term in the right-hand side tends to zero as $\ell\to\infty$. The second term tends to zero as $i\to\infty$ uniformly in $\ell\in\bN$, since $\|y_{k_\ell}-y\|_\cV\leq2$ and $\lim_{i\to\infty}\theta_i=\infty$. Thus, $\lim_{\ell\to\infty}\|y_{k_\ell}-y\|_\cH=0$, proving that the set $\overline{B}_\cV$ is sequentially compact in $\cH$. Consequently, the embedding $\cV\hookrightarrow\cH$ is compact.

To prove the ``only if'' part, suppose that there exists an $m\in(0,\infty)$ such that $\supp\cap[0,m]$ is an infinite set. We can then choose a strictly increasing sequence $\{\theta_i\}_{i\in\bN}\subset\supp\cap(0,m)$. Let $\theta_0=0$, and define $A_i:=(\frac{\theta_{i-1}+\theta_i}{2},\frac{\theta_i+\theta_{i+1}}{2})$ for each $i\in\bN$. By construction, $\{A_i\}_{i\in\bN}$ is a sequence of disjoint open intervals contained in $[0,m]$ such that $\mu(A_i)\in(0,\infty)$ for any $i\in\bN$. Take an $e\in\bR^n$ such that $|e|=1$. For each $i\in\bN$, define $c_i:=\int_{A_i}(1+\theta)^{1/2}\dmu\in(0,\infty)$ and $y_i:=\frac{1}{\sqrt{c_i}}e\1_{A_i}\in\cV$. Then, we have $\|y_i\|_{\cV}=1$ for any $i\in\bN$. However, for any $i,j\in\bN$ with $i\neq j$, we have
\begin{equation*}
	\|y_i-y_j\|_{\cH}^2=\frac{\int_{A_i}(1+\theta)^{-1/2}\dmu}{\int_{A_i}(1+\theta)^{1/2}\dmu}+\frac{\int_{A_j}(1+\theta)^{-1/2}\dmu}{\int_{A_j}(1+\theta)^{1/2}\dmu}\geq\frac{2}{1+m}.
\end{equation*}
Thus, $\{y_i\}_{i\in\bN}$ does not have a Cauchy subsequence in $\cH$, implying that the embedding $\cV\hookrightarrow\cH$ is not compact.
\end{proof}

The following lemma establishes some fundamental properties of the path space of the Markovian lift. More generally, it addresses the path space of solutions to general monotone SPDEs defined on a Gelfand triplet. Here, we consider a general setting that goes beyond the specific framework of Markovian lifts.

%% Lemma

\begin{lemm}\label{lemm_app-Lambda}
Let $(\cH,\|\cdot\|_\cH,\langle\cdot,\cdot\rangle_\cH)$ be a separable Hilbert space and $\cV$ be a dense subspace of $\cH$. Assume that $\cV$ is equipped with a norm $\|\cdot\|_\cV$ such that $(\cV,\|\cdot\|_\cV)$ is a reflexive separable Banach space continuously embedded into $\cH$. Define
\begin{equation*}
	\Lambda:=\left\{\eta:[0,\infty)\to\cH\relmiddle|\text{$t\mapsto\eta_t$ is strongly continuous in $\cH$ and}\ \int^T_0\|\eta_t\|_\cV^2\,\diff t<\infty\ \text{for any $T>0$}\right\}
\end{equation*}
and
\begin{equation*}
	d_\Lambda(\eta,\eta'):=\sum_{T\in\bN}\frac{1}{2^T}\big\{\|\eta-\eta'\|_{\Lambda_T}\wedge1\big\}\ \ \text{for $\eta,\eta'\in\Lambda$},
\end{equation*}
where
\begin{equation*}
	\|\eta\|_{\Lambda_T}:=\left(\sup_{t\in[0,T]}\|\eta_t\|_\cH^2+\int^T_0\|\eta_t\|_\cV^2\,\diff t\right)^{1/2}\ \ \text{for $\eta\in\Lambda$ and $T\in(0,\infty)$}.
\end{equation*}
Then, $(\Lambda,d_\Lambda)$ is a complete separable metric space, and $\{\|\cdot\|_{\Lambda_T}\}_{T\in\bN}$ is a countable family of seminorms on $\Lambda$ generating the topology of $(\Lambda,d_\Lambda)$. Furthermore, the Borel $\sigma$-algebra $\cB(\Lambda)$ coincides with the $\sigma$-algebra generated by the family $\{\{\eta\in\Lambda\,|\,\eta_t\in A\}\,|\,t\in[0,\infty),\,A\in\cB(\cH)\}$.
\end{lemm}

%% Remark

\begin{rem}\label{rem_app-Lambda}
\begin{itemize}
\item[(i)]
Since $\cV$ and $\cH$ are separable Banach spaces with continuous embedding $\cV\subset\cH$, the Lusin--Suslin theorem (see, e.g., \cite[Theorem 15.1]{Ke95}) implies that  $\cB(\cV)=\{A\cap\cV\,|\,A\in\cB(\cH)\}\subset\cB(\cH)$. Furthermore, by defining $\|y\|_\cV=\infty$ for $y\in\cH\setminus\cV$, the map $\cH\ni y\mapsto\|y\|_\cV\in[0,\infty]$ is lower semi-continuous on $\cH$ and hence $\cB(\cH)$-measurable (see, e.g., \cite[Exercise 4.2.3]{LiRo15}). Hence, the integral $\int^T_0\|\eta_t\|_\cV^2\,\diff t\in[0,\infty]$ is well-defined for all $T>0$ for every $\cB(\cH)$-measurable function $\eta:[0,\infty)\to\cH$.
\item[(ii)]
By the final assertion of \cref{lemm_app-Lambda}, for every $\cH$-valued measurable process $Y=(Y_t)_{t\geq0}$ on a complete probability space $(\Omega,\cF,\bP)$ such that $t\mapsto Y_t$ is strongly continuous in $\cH$ and $\int^T_0\|Y_t\|_\cV^2\,\diff t<\infty$ a.s.\ for any $T>0$, the sample-path map $\omega\mapsto(Y_t(\omega))_{t\geq0}$ defines a Borel-measurable random variable taking values in the Polish space $\Lambda$, and its law on $\Lambda$ is uniquely characterized by the ``finite-dimensional distributions'' $\bP((Y_{t_1},\dots,Y_{t_\ell})\in A)$ with $t_1,\dots,t_\ell\in[0,\infty)$, $A\in\cB(\cH^\ell)$ and $\ell\in\bN$.
\end{itemize}
\end{rem}

Although the results above are expected to be well-known and have been implicitly invoked in studies on monotone SPDEs (e.g., \cite[Appendix E]{LiRo15} and \cite{Ha25a}), a formal proof is often omitted. For completeness, we provide a proof below. Recall that for a separable Banach space $(\cX,\|\cdot\|_\cX)$, the space $C([0,\infty);\cX)$ of $\cX$-valued continuous functions on $[0,\infty)$, equipped with the metric defined by \eqref{eq_metric-C}, is a complete separable metric space.

%% Proof

\begin{proof}[Proof of \cref{lemm_app-Lambda}]
Obviously, $d_\Lambda$ defines a metric on $\Lambda$, and $\{\|\cdot\|_{\Lambda_T}\}_{T\in\bN}$ is a countable family of seminorms on $\Lambda$ generating the topology of $(\Lambda,d_\Lambda)$. In the following, we prove the completeness and separability of $\Lambda$, as well as the identification of its associated Borel $\sigma$-algebra $\cB(\Lambda)$. To this end, we observe from the embedding $\cV\hookrightarrow\cH$ and the definition of the metric $d_\Lambda$ that
\begin{equation}\label{eq_app_Lambda-embedding}
	C([0,\infty);\cV)\hookrightarrow\Lambda\hookrightarrow C([0,\infty);\cH).
\end{equation}

To prove the completeness of $\Lambda$, let $\{\eta^k\}_{k\in\bN}$ be a Cauchy sequence in $\Lambda$. By the continuous embedding \eqref{eq_app_Lambda-embedding}, $\{\eta^k\}_{k\in\bN}$ is also a Cauchy sequence in the complete metric space $C([0,\infty);\cH)$. Hence there exists an $\eta\in C([0,\infty);\cH)$ such that $\lim_{k\to\infty}\sup_{t\in[0,T]}\|\eta^k_t-\eta_t\|_\cH=0$ for any $T\in(0,\infty)$. Furthermore, the lower semi-continuity of the map $\cH\ni y\mapsto\|y\|_\cV\in[0,\infty]$ and Fatou's lemma yield
\begin{equation*}
	\int^T_0\|\eta_t\|_\cV^2\,\diff t\leq\int^T_0\liminf_{k\to\infty}\big\|\eta^k_t\big\|_\cV^2\,\diff t\leq\liminf_{k\to\infty}\int^T_0\big\|\eta^k_t\big\|_\cV^2\,\diff t<\infty
\end{equation*}
and
\begin{equation*}
	\limsup_{k\to\infty}\int^T_0\big\|\eta^k_t-\eta_t\big\|_\cV^2\,\diff t\leq\limsup_{k\to\infty}\int^T_0\liminf_{\ell\to\infty}\big\|\eta^k_t-\eta^\ell_t\big\|_\cV^2\,\diff t\leq\lim_{k,\ell\to\infty}\int^T_0\big\|\eta^k_t-\eta^\ell_t\big\|_\cV^2\,\diff t=0
\end{equation*}
for any $T\in(0,\infty)$. Hence, we have $\eta\in\Lambda$ and $\lim_{k\to\infty}d_\Lambda(\eta^k,\eta)=0$. This demonstrates that the metric space $(\Lambda,d_\Lambda)$ is complete.

To prove the separability of $\Lambda$, recalling the continuous embedding \eqref{eq_app_Lambda-embedding} and the separability of $C([0,\infty);\cV)$, it suffices to show that $C([0,\infty);\cV)$ is dense in $\Lambda$. Let $\eta\in\Lambda$ be fixed. We extend the domain of $\eta$ to $\bR$ by setting $\eta_t:=\eta_{-t}$ for $t\in(-\infty,0)$. Then, we can regard $\eta$ as a strongly continuous map from $\bR$ to $\cH$ such that $\int_{[-T,T]}\|\eta_t\|_\cV^2\,\diff t<\infty$ for any $T\in(0,\infty)$. For each $\ep>0$, let $\phi_\ep:\bR\to\bR$ be the standard mollifier supported on $[-\ep,\ep]$, and define $\eta^\ep_t:=\int_\bR\phi_\ep(t-s)\eta_s\,\diff s$ for $t\in[0,\infty)$, where the integral is well-defined as a Bochner integral on $\cH$. On the one hand, thanks to the integrability $\int_{[-T,T]}\|\eta_t\|_\cV^2\,\diff t<\infty$ for each $T>0$, we see that $\eta^\ep\in C([0,\infty);\cV)$ for each $\ep>0$ and $\lim_{\ep\downarrow0}\int_{[0,T]}\|\eta^\ep_t-\eta_t\|_\cV^2\,\diff t=0$ for any $T>0$. On the other hand, thanks to the strong continuity of $\bR\ni t\mapsto\eta_t$ in $\cH$, we have $\lim_{\ep\downarrow0}\sup_{t\in[0,T]}\|\eta^\ep_t-\eta_t\|_\cH=0$ for any $T>0$. Thus, $\eta^\ep$ converges to $\eta$ in $\Lambda$ as $\ep\downarrow0$, showing the density of $C([0,\infty);\cV)$ in $\Lambda$. Hence, $\Lambda$ is separable.

We prove the last assertion for the Borel $\sigma$-algebra $\cB(\Lambda)$ on $\Lambda$. First, since both $\Lambda$ and $C([0,\infty);\cH)$ are Polish spaces and $\Lambda\hookrightarrow C([0,\infty);\cH)$, the Lusin--Suslin theorem (see, e.g., \cite[Theorem 15.1]{Ke95}) implies that $\Lambda\in\cB(C([0,\infty);\cH))$ and $\cB(\Lambda)=\{B\cap\Lambda\,|\,B\in\cB(C([0,\infty);\cH))\}$. Denote by $\cG(C([0,\infty);\cH))$ the $\sigma$-algebra on $C([0,\infty);\cH)$ generated by the family of cylinder sets $\{\{\eta\in C([0,\infty);\cH)\,|\,\eta_t\in A\}\,|\,t\in[0,\infty),\,A\in\cB(\cH)\}$. As a next step, we show that $\cB(C([0,\infty);\cH))=\cG(C([0,\infty);\cH))$. On the one hand, since every evaluation map $\eta\mapsto\eta_t$ for $t\in[0,\infty)$ is continuous from $C([0,\infty);\cH)$ to $\cH$, we have $\cG(C([0,\infty);\cH))\subset\cB(C([0,\infty);\cH))$. On the other hand, for any $\check{\eta}\in C([0,\infty);\cH)$, $T>0$ and $r>0$, we have
\begin{equation*}
	\left\{\eta\in C([0,\infty);\cH)\relmiddle|\sup_{t\in[0,T]}\big\|\eta_t-\check{\eta}_t\big\|_\cH<r\right\}=\bigcup_{q\in(0,r)\cap\bQ}\bigcap_{s\in[0,T]\cap\bQ}\left\{\eta\in C([0,\infty);\cH)\relmiddle|\big\|\eta_s-\check{\eta}_s\big\|_\cH\leq q\right\},
\end{equation*}
which belongs to $\cG(C([0,\infty);\cH))$. Since the topology of $C([0,\infty);\cH)$ is generated by the family of sets of the above forms, we obtain $\cB(C([0,\infty);\cH))\subset\cG(C([0,\infty);\cH))$, concluding that $\cB(C([0,\infty);\cH))=\cG(C([0,\infty);\cH))$. Therefore, the Borel $\sigma$-algebra $\cB(\Lambda)$ on $\Lambda$, which is equal to $\{B\cap\Lambda\,|\,B\in\cB(C([0,\infty);\cH))\}$, coincides with $\{B\cap\Lambda\,|\,B\in\cG(C([0,\infty);\cH))\}$. By the standard monotone-class argument, we see that the latter is also equal to the $\sigma$-algebra generated by the family $\{\{\eta\in\Lambda\,|\,\eta_t\in A\}\,|\,t\in[0,\infty),\,A\in\cB(\cH)\}$. This completes the proof.
\end{proof}

%%%%%%%%%%%%%%%%%%%%%%%%%%%%
%%%%%% Acknowledgements
%%%%%%%%%%%%%%%%%%%%%%%%%%%%

\section*{Acknowledgments}
The author would like to thank Kiyoshi Kanazawa for insightful discussions and valuable comments regarding the physical background of this work.

%%%%%%%%%%%%%%%%%%%%%%%%%%%%
%%%%%% References
%%%%%%%%%%%%%%%%%%%%%%%%%%%%

\end{document}